\documentclass[a4paper,12pt]{amsart}

\usepackage{amsmath}
\usepackage{amsfonts}
\usepackage{amssymb}
\usepackage{amsthm}
\usepackage[all]{xy}
\usepackage{caption}
\usepackage[CJKbookmarks]{hyperref}
\usepackage{fancyhdr}
\usepackage{tikz-cd}
\usepackage{geometry}

\theoremstyle{definition}
\newtheorem{definition}{Definition}[section]
\newtheorem{remark}[definition]{Remark}

\newtheorem{example}[definition]{Example}
\newtheorem{question}[definition]{Question}

\newtheorem{hypothesis}[definition]{Hypothesis}
\newtheorem{setting}[definition]{Setting}

\theoremstyle{plain}
\newtheorem{theorem}[definition]{Theorem}

\newtheorem{lemma}[definition]{Lemma}
\newtheorem{proposition}[definition]{Proposition}
\newtheorem{corollary}[definition]{Corollary}

\DeclareMathOperator{\cA}{\mathcal{A}}
\DeclareMathOperator{\cB}{\mathcal{B}}
\DeclareMathOperator{\cC}{\mathcal{C}}

\DeclareMathOperator{\cG}{\mathcal{G}}
\DeclareMathOperator{\cH}{\mathcal{H}}
\DeclareMathOperator{\cI}{\mathcal{I}}
\DeclareMathOperator{\cL}{\mathcal{L}}
\DeclareMathOperator{\cM}{\mathcal{M}}
\DeclareMathOperator{\cN}{\mathcal{N}}

\DeclareMathOperator{\cK}{\mathcal{K}}
\DeclareMathOperator{\cR}{\mathcal{R}}
\DeclareMathOperator{\cS}{\mathcal{S}}

\DeclareMathOperator{\cU}{\mathcal{U}}
\DeclareMathOperator{\cV}{\mathcal{V}}
\DeclareMathOperator{\cX}{\mathcal{X}}
\DeclareMathOperator{\cY}{\mathcal{Y}}
\DeclareMathOperator{\cZ}{\mathcal{Z}}

\DeclareMathOperator{\add}{add}

\DeclareMathOperator{\Coker}{Coker}

\DeclareMathOperator{\depth}{depth}

\DeclareMathOperator{\End}{End}
\DeclareMathOperator{\gEnd}{\underline{End}}
\DeclareMathOperator{\Ext}{Ext}
\DeclareMathOperator{\gExt}{\underline{Ext}}

\DeclareMathOperator{\gr}{gr}

\DeclareMathOperator{\gldim}{gldim}

\DeclareMathOperator{\GL}{GL}
\DeclareMathOperator{\Gr}{Gr}

\DeclareMathOperator{\hdet}{hdet}
\DeclareMathOperator{\Hom}{Hom}
\DeclareMathOperator{\gHom}{\underline{Hom}}

\DeclareMathOperator{\id}{id}
\DeclareMathOperator{\idim}{idim}
\DeclareMathOperator{\image}{Im}

\DeclareMathOperator{\im}{Im}

\DeclareMathOperator{\Ker}{Ker}

\DeclareMathOperator{\Lim}{Lim}
\DeclareMathOperator{\lcd}{lcd}

\DeclareMathOperator{\MCM}{MCM}
\DeclareMathOperator{\Mod}{Mod}

\DeclareMathOperator{\qgr}{qgr}
\DeclareMathOperator{\QGr}{QGr}

\DeclareMathOperator{\pdim}{pdim}

\DeclareMathOperator{\proj}{proj}

\DeclareMathOperator{\rank}{rank}

\DeclareMathOperator{\Rep}{Rep}

\DeclareMathOperator{\soc}{soc}

\DeclareMathOperator{\SL}{SL}
\DeclareMathOperator{\Tor}{Tor}
\DeclareMathOperator{\tor}{tor}

\newcommand{\dlim}{\underrightarrow{\Lim}}

\DeclareMathOperator{\D}{\mathbf{D}}

\begin{document}
\title[Noncommutative resolutions of isolated singularities]{Noncommutative resolutions of noncommutative isolated singularities}
\author{Haonan Li}
\thanks{Haonan Li: School of Mathematics,
Shanghai University of Finance and Economics,
 Shanghai, 200433,
China.
Email: lihaonan@mail.shufe.edu.cn}
\author{Quanshui Wu}
\thanks{Quanshui Wu: School of Mathematical Sciences,
Fudan University,
Shanghai, 200433,
 China. 
 Email: qswu@fudan.edu.cn}
\thanks{Haonan Li is supported by the Fundamental Research Funds for the Central Universities. Quanshui Wu is supported by the NSFC (Grant No. 12471032).}

\keywords{Noncommutative resolution, noncommutative isolated singularity, noncommutative projective scheme, commonly graded algebra, Hopf  action}
\subjclass[2020]{14A22, 16S38, 16W50, 16E65, 16S40, 14E15}

\newgeometry{left=3.18cm,right=3.18cm,top=2.54cm,bottom=2.54cm}

\begin{abstract}
Noncommutative resolutions of AS-Gorenstein isolated singularities are investigated in \cite{LSW}. However, establishing their existence and constructing such resolutions are generally difficult, even when they exist. 
In this paper, we study the homological properties of noncommutative projective schemes associated with noncommutative isolated singularities and consider conditions under which a given commonly graded AS-regular algebra serves as a noncommutative resolution of an isolated singularity. 
This leads to a more general definition of noncommutative resolutions of balanced Cohen--Macaulay isolated singularities. 
We show that the existence of such resolutions is equivalent to the existence of cluster tilting modules over balanced CM isolated singularities.
The corresponding noncommutative analogue of the Bondal-Orlov conjecture is established in dimensions $2$ and $3$.
As an application, we study Hopf actions on commonly graded AS-Gorenstein algebras and investigate noncommutative resolutions of invariant rings. We present three examples of noncommutative resolutions, including one in which the noncommutative isolated singularity is not connected graded.
\end{abstract}

\maketitle

\tableofcontents

\section{Introduction}
Van den Bergh \cite{V1,V2} introduced the notion of noncommutative crepant resolutions as a bridge to explore the Bondal-Orlov conjecture \cite{BO1,BO2}, which concerns the uniqueness of the crepant resolutions of singularities. Inspired by Van den Bergh's work, the objects
to be resolved have also been extended to noncommutative rings. Work by Iyama, Reiten, and Wemyss in \cite{I2,IR,IW1,IW2} focuses on noncommutative resolutions of module-finite algebras, showing that such resolutions can be realized as endomorphism rings of certain cluster tilting modules.
In \cite{QWZ}, Qin, Wang, and Zhang introduced noncommutative quasi-resolutions for graded Auslander-Gorenstein algebras via equivalences between certain quotient categories.
Subsequently, He and Ye \cite{HY} introduced the notion of a noncommutative pre-resolution for a noncommutative isolated singularity. 
In this paper, we continue the work presented in \cite{LSW}, by studying noncommutative resolutions of noncommutative isolated singularities.

As mentioned above, noncommutative resolutions are realized as endomorphism rings of finitely generated modules. In the graded setting, such endomorphism rings should be bounded below $\mathbb{Z}$-graded locally finite algebras, which 
are called commonly graded algebras in the literature. 
Therefore, commonly graded algebras should play a prominent role in the theory of noncommutative resolutions. This perspective also motivates the extension of concepts and results from the connected graded case to the commonly graded setting.
Noncommutative projective schemes, originally introduced for $\mathbb{N}$-graded algebras in \cite{AZ}, play a central role in noncommutative projective geometry. 
In this study, we propose a minor modification to the definition of a noncommutative projective scheme to insure it satisfies the hypotheses of ``the noncommutative Serre theorem".
More precisely, for a right noetherian commonly graded algebra $A$, let $\qgr A$ denote the quotient of the category of finitely generated graded modules by the full subcategory of finite-dimensional graded modules. In this context, the image of $A$ in $\qgr A$ is denoted by $\cA$, and the auto-equivalence of $\qgr A$ induced by the shift functor is denoted by $s$.
When the triple $(\qgr A,\cA,s)$ satisfies the conditions outlined in Theorem \ref{noncommutative Serre theorem for commonly graded algebra}, 
it is referred to as a noncommutative projective scheme, and the pair $(\qgr A,s)$ is referred to as the noncommutative quasi-projective space associated with $A$. 
If the global dimension of $\qgr A$ is finite, then $A$ is called a noncommutative isolated singularity, which is defined in \cite{Jo, Ue} as an analogue of an isolated singularity. For a detailed discussion of this notion, see \cite[Corollary 4.13]{LW1}.

AS-regular algebras and AS-Gorenstein algebras serve as noncommutative analogues of polynomial rings and Gorenstein rings, respectively.
Commonly graded algebras with properties analogous to those of AS-Gorenstein (regular) algebras are termed commonly graded AS-Gorenstein (regular) algebras; see Definition \ref{def-generalized-AS-Gorenstein}.
Since AS-regular algebras are regarded as coordinate rings of noncommutative projective spaces $\text{q}\mathbb{P}^n$, 
the noncommutative projective schemes associated with commonly graded AS-regular algebras are regarded as smooth objects in noncommutative algebraic geometry.
Resolving a singularity consists of finding a smooth object that is equivalent, in an appropriate sense, to the singular object. In algebraic geometry, birational equivalence serves as the relevant notion of equivalence.
Following the approach of \cite{QWZ}, we replace birational equivalence with equivalence of noncommutative quasi-projective spaces.
Consequently, a noncommutative resolution of a commonly graded AS-Gorenstein isolated singularity is expected to be a noetherian commonly graded AS-regular algebra whose associated noncommutative quasi-projective space is equivalent to that of the singularity.
The following is the definition of a noncommutative resolution as given in \cite{LSW}.

\begin{definition}\label{def-nc-reso-AS-G}
Let $A$ be a commonly graded AS-Gorenstein isolated singularity of dimension $d\geqslant 2$. A {\it noncommutative resolution} of $A$ is a graded ring $B=\gEnd_A(M)$ for some maximal Cohen-Macaulay (for short, MCM) generator $M_A$ such that $B$ has finite graded global dimension $d$ and $B_B$ is an MCM $B$-module.
\end{definition}

It is proved in \cite[Section 6]{LSW} that a noncommutative resolution $B$ is a commonly graded AS-regular algebra of dimension $d$. Let $N=\gHom_A(M,A)$. 
Then $-\otimes_A N$ and $-\otimes_B M$ induce an equivalence between the noncommutative quasi-projective spaces:
$$-\otimes_{\cA}\cN:(\qgr A,s)\rightleftarrows (\qgr B,s):-\otimes_{\cB}\cM.$$
Observe that $N$ is a projective $B$-module and that $A\cong \gEnd_B(N)$.
Moreover, it can be shown that $M_A$ is a $(d-1)$-cluster tilting module.

Indeed, constructing a noncommutative resolution of a commonly graded AS-Gorenstein isolated singularity is generally challenging.
Considering the problem from the opposite perspective leads to the following question.

\begin{question}\label{question}
    Given a noetherian commonly graded AS-regular algebra, under what conditions is it a noncommutative resolution of a commonly graded AS-Gorenstein isolated singularity?
\end{question}

By the facts of noncommutative resolutions introduced above, 
if a commonly graded AS-regular algebra $B$ is a noncommutative resolution of $A$, then $A$ must be the graded endomorphism ring of a finitely generated projective $B$-module.

Next suppose that $B$ is a noetherian commonly graded AS-regular algebra, and let $M'$ be a finitely generated graded projective $B$-module. Let $A=\gEnd_B(M')$.
We wonder when $B$ is a noncommutative resolution of $A$. 
Note that, as $M'_B$ is projective, the natural morphism $M'\otimes_BM\to A$ is surjective, where $M=\gHom_B(M',B)$. If the cokernel of $M\otimes_AM'\to B$ is finite-dimensional, then $-\otimes_A M'$ and $-\otimes_B M$ induce an equivalence; see Theorem \ref{equ. of quot. cat. induced by Morita context for commonly graded}:
$$-\otimes_{\cA}\cM':(\qgr A,s)\rightleftarrows (\qgr B,s):-\otimes_{\cB}\cM.$$

The advantage is that, for a commonly graded AS-regular algebra, there are only finitely many indecomposable projective modules up to isomorphism and shift. So when a commonly graded AS-regular algebra is given, after passing to basic projective modules (i.e. the projective modules whose direct summands are non-isomorphic up to shift), only finitely many graded endomorphism algebras need to be considered up to graded Morita equivalence.

Based on the discussion above, we make the following hypothesis. See Section \ref{preliminaries} for notations. 

\begin{hypothesis}\label{hypothesis}
\begin{itemize}
    \item [(H1)] $B$ is a noetherian commonly graded AS-regular algebra of dimension $d\geqslant 2$ with $U=D(R^d\Gamma_B(B))$ being the dualizing module.
    \item [(H2)] $M'_B$ is a finitely generated projective $B$-module. 
    \item [(H3)] Let $A=\gEnd_B(M')$ and $M=\gHom_B(M',B)$. The cokernel of the natural morphism $M\otimes_A M'\to B$ is finite-dimensional.
\end{itemize}  
\end{hypothesis}

In the ungraded setting, assuming that $B$ is a Calabi-Yau algebra and that $B/BeB$ is finite-dimensional, Amiot, Iyama, and Reiten \cite[Section 2]{AIR} studied the endomorphism ring $eBe$ of the projective module $eB$, where $e$ is an idempotent of $B$. They proved that $eBe$ is Gorenstein.
By \cite{RR}, $\mathbb N$-graded Calabi--Yau algebras are AS-regular. 
Thus, Hypothesis \ref{hypothesis} may be regarded as a generalization of the assumptions of Amiot, Iyama, and Reiten.
However, commonly graded AS-regular algebras generally do not possess the symmetry characteristic of Calabi--Yau algebras. Consequently, the graded endomorphism ring $A$ of $M$ in Hypothesis \ref{hypothesis} is usually not Gorenstein.

In fact, under the Hypothesis \ref{hypothesis}, it is demonstrated that $A$ is a balanced CM isolated singularity (see Definition \ref{balanced CM isolated singularity}). This suggests that we should extend our focus beyond merely the noncommutative resolutions of AS-Gorenstein isolated singularities. 

Following the principle that ``resolving a singularity is to find some smooth object which is equivalent to the singular one in some sense," suppose that,
for a noetherian noncommutative isolated singularity $A$, there exists a noetherian commonly graded AS-regular algebra $B$ such that the noncommutative quasi-projective spaces associated with $A$ and $B$ are equivalent.
Then $B$ is called a noncommutative quasi-resolution of $A$ (Definition \ref{def-noncom-quasi-resolution}). This terminology originates from \cite{QWZ} (see Remark \ref{remark of nqr}).
By the theory of equivalences of noncommutative quasi-projective spaces, 
the equivalence $(\qgr A,s)\cong (\qgr B,s)$ is induced by a graded Morita context defined by a graded $A$-module $M_A$, whose properties are closely related to those of cluster-tilting modules (Theorem \ref{d-1 CT and NQR}).
If $M_A$ is also a generator, then $B$ is called a noncommutative resolution of $A$ (Definition \ref{def-noncom-reso}), generalizing Definition \ref{def-nc-reso-AS-G} (see Proposition \ref{M is MCM iff M is a generator}).

The following theorem provides, in a certain sense, an answer to Question \ref{question}.

\begin{theorem}[Theorems \ref{hypo. is equi. to B is ncr of A}-\ref{when A is GAS Gorenstein}]\label{main theorem}
Suppose that Hypothesis \ref{hypothesis} holds. Then
\begin{itemize}
    \item [(1)] $A$ is a balanced CM isolated singularity of dimension $d$;
    \item [(2)] the balanced dualizing complex of $A$ is $(M'\otimes_BU\otimes_BM)[d]$;
    \item [(3)] $M_A$ is a $(d-1)$-cluster tilting module;
    \item [(4)] $A$ is commonly graded AS-Gorenstein if and only if $M'\otimes_BU\otimes_BM$ is a projective $A$-module (or $A^o$-module).
\end{itemize}

Hypothesis \ref{hypothesis} holds if and only if $B$ is a noncommutative resolution of the noetherian noncommutative projective coordinate ring $A$ given by $(M,M')$.
\end{theorem}

Moreover, as in the case of noncommutative resolutions of AS-Gorenstein isolated singularities \cite{LSW}, noncommutative resolutions of balanced CM isolated singularities are given by cluster tilting modules, and conversely, cluster tilting modules yield noncommutative resolutions.

\begin{theorem}[Theorem \ref{M_A is CT module}]\label{main theorem CT ncr}
Let $A$ be a balanced CM isolated singularity of dimension $d\geqslant 2$, and let $M\in\MCM A$. Set $B=\gEnd_A(M)$ and $M'=\gHom_A(M,A)$. Then the following conditions are equivalent:
\begin{itemize}
    \item [(1)] $M$ is a $(d-1)$-cluster tilting $A$-module.
    \item [(2)] $B$ is a noncommutative resolution of $A$.
\end{itemize}
When these conditions hold, $B$ is a noetherian commonly graded AS-regular algebra of global dimension $d$.
\end{theorem}

The corresponding noncommutative version of the Bondal-Orlov conjecture is shown to hold in dimensions $2$ and $3$; see Theorem \ref{BO conjecture}.

Before proving Theorem \ref{main theorem}, we establish several preliminary results concerning homological dimensions in the category $\qgr A$.
In particular, we prove left-right symmetry for noncommutative isolated singularities.

\begin{theorem}[Theorem \ref{A and A^o are noncommutative iso. sing.}]\label{main-thm-iso-sing.}
    Let $A$ be a noetherian commonly graded algebra with a balanced dualizing complex. Then
    \begin{itemize}
        \item [(1)] $A$ is a noncommutative isolated singularity if and only if its opposite algebra $A^o$ is a noncommutative isolated singularity.
        \item [(2)] $\gldim(\qgr A)=\gldim(\qgr A^o)$.
    \end{itemize}
\end{theorem}

When $A$ is a balanced CM algebra of dimension $d\geqslant 2$ with $\gldim(\qgr A)<\infty$, $\gldim(\qgr A)$ need not equal $d-1$; see Example \ref{triangular balanced CM example}. 
The equality $\gldim(\qgr A)=d-1$ holds if and only if the canonical module $\Omega_A=D(R^d\Gamma_A(A))$ is modulo-torsion-invertible. 
And in this case, every nonzero object of $\qgr A$ has both projective and injective dimension $d-1$, and the same holds in $\qgr A^o$; see Theorem \ref{canonical bimodule and isolated singularity}. 
We also compare the injective and projective dimensions of objects in the derived categories of $\qgr A$ and $\QGr A$.

\newlength{\widest}
\settowidth{\widest}{balanced CM isolated singularities}
\begin{figure*}[ht]
\begin{center}
\begin{tikzcd}[column sep=-10em, row sep=1.0em]
& \text{AS-regular algebras}\\
& \text{AS-Gorenstein isolated singularities}\\
\makebox[\widest][c]{\text{AS-Gorenstein algebras}} & & \text{balanced CM isolated singularities}\\
& \text{balanced CM algebras}\\
&\text{$\chi$-condition and finite local cohomology dimension}
\arrow[from=1-2,Rightarrow]{2-2}[]{}
\arrow[from=2-2,Rightarrow]{3-1}[]{}
\arrow[from=2-2,Rightarrow]{3-3}[]{}
\arrow[from=3-1,Rightarrow]{4-2}[]{}
\arrow[from=3-3,Rightarrow]{4-2}[]{}
\arrow[from=4-2,Rightarrow]{5-2}[]{}
\end{tikzcd}
\caption*{Diagram $1$}
\end{center}
\end{figure*}
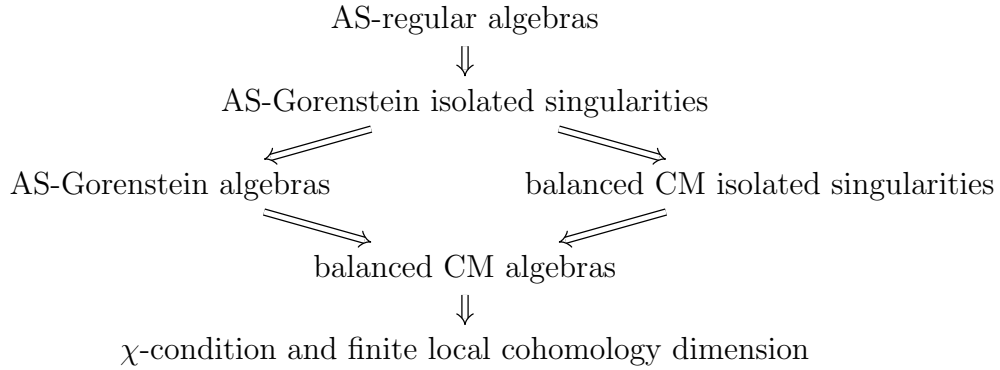

Diagram $1$ illustrates the relationships among important homological properties of graded algebras, all of which are assumed to be noetherian. We consider the setting in which two noetherian commonly graded algebras $A$ and $B$ have equivalent noncommutative quasi-projective spaces. Assuming that $B$ satisfies certain properties outlined in Diagram $1$, we investigate the corresponding properties of $A$.

To provide substantial examples of noncommutative resolutions, we consider Hopf actions on commonly graded AS-regular (Gorenstein) algebras.

Jørgensen and Zhang \cite{JZ} proved that, for a connected graded AS-Gorenstein algebra,  
the invariant subring under the action of a finite group generated by automorphisms with trivial homological determinants is AS-Gorenstein. Kirkman, Kuzmanovich, and Zhang \cite{KKZ} extended this result from group actions to Hopf actions and established an analogous conclusion.

Auslander \cite{A1,A2} proved that if $G$ is a finite subgroup of $\GL_n(k)$ containing no pseudo-reflections, then there is a ring isomorphism $R\#G\cong\End_{R^G}(R)$, where $R\#G$ denotes the smash product and $R^G$ denotes the invariant subring under the group action. This theorem provides numerous examples of noncommutative crepant resolutions in the sense of Van den Bergh.

Chan, Kirkman, Walton, and Zhang \cite{CKWZ1} considered Auslander's result in the noncommutative setting. They conjectured that if a finite-dimensional semisimple Hopf algebra $H$ acts inner-faithfully on a noetherian AS-regular algebra $A$ with trivial homological determinant, then $A\#H\cong \gEnd_{A^H}(A)$.
This conjecture has been proved for many Hopf actions on AS-regular algebras \cite{CKWZ2,HOZ,BHZ1,BHZ2,GKMW,ChKZ} and is referred to as the noncommutative Auslander theorem. 
In many cases where the noncommutative Auslander theorem holds, the invariant subrings are also noncommutative isolated singularities \cite{MU,CKWZ1,BHZ2,GKMW,ChKZ,Z}, among others.
Mori and Ueyama \cite{MU} investigate a specific class of group actions, termed ample actions, which ensure both the noncommutative Auslander theorem holds and that the invariant subring is a noncommutative isolated singularity.

In the commonly graded setting, we define the homological determinants of Hopf actions on commonly graded basic AS-Gorenstein algebras and generalize the results of \cite{JZ,KKZ}.

\begin{theorem}[Theorem \ref{A^H is AS Gorenstein}]\label{intro-A^H is AS Gorenstein}
Let $H$ be a finite-dimensional semisimple Hopf algebra, and let $A$ be a noetherian basic commonly graded AS-Gorenstein algebra of dimension $d$. Assume that $A$ is a left $H$-module algebra, the $H$-action on $A$ is homogeneous, and the homological determinant $\hdet_{\mathbf e}$ of the $H$-action on $A$ is trivial. Then $A^H$ is a noetherian commonly graded AS-Gorenstein algebra of dimension $d$.
\end{theorem}

Theorem \ref{intro-A^H is AS Gorenstein} also generalizes the result of \cite{Wei2}, which studies group actions on preprojective algebras.

Motivated by ample group actions on connected graded algebras, we define ample Hopf actions on commonly graded algebras (Definition \ref{def-ample-Hopf-action}).
Let $H$ be a finite-dimensional semisimple Hopf algebra acting on a noetherian commonly graded algebra $A$.
If $-\otimes_{A\#H}A$ and $-\otimes_{A^H}A$ induce an equivalence
$$-\otimes_{\cA\#\cH}\cA:(\qgr A\#H,s)\rightleftarrows (\qgr A^H,s):-\otimes_{\cA^{\cH}}\cA,$$
then the $H$-action is said to be ample.
We study the noncommutative resolution of the invariant subring under an ample Hopf action without assuming that the homological determinant is trivial.

\begin{theorem}[Theorem \ref{ample Hopf action when hdet not trivial}]
Let $A$ be a noetherian commonly graded AS-regular algebra of dimension $d\geqslant 2$, and let $H$ be a finite-dimensional semisimple Hopf algebra. Suppose that $A$ is a left $H$-module algebra and that $H$ acts homogeneously on $A$. Then the following statements are equivalent.
\begin{itemize}
    \item [(1)] The $H$-action on $A$ is ample.
    \item [(2)] The functor $-\otimes_{A\#H}A$ induces an equivalence:
    $$-\otimes_{\cA\#\cH}\cA:(\qgr A\#H,s)\to (\qgr A^H,s).$$
    \item [(3)] The quotient $A\#H/(1\#\int)$ is finite-dimensional, where $(1\#\int)$ is the ideal generated by $1\#\int$.
    \item [(4)]  $\gldim(\qgr A^H)=d-1$, and the natural map
    $$A\#H\to \gEnd_{A^H}(A), \, a\#h \mapsto \big(b\mapsto a(h\rightharpoonup b)\big)$$
    is an isomorphism of graded algebras.
    \item [(5)] $A\#H$ is a noncommutative resolution of $A^H$, given by $({}_{A\#H}A_{A^H},{}_{A^H}A_{A\#H})$.
\end{itemize}
\end{theorem}

As a corollary, we consider the case in which the homological determinant is trivial in Theorem \ref{ample Hopf action}.

In the final part of the paper, we present three examples of noncommutative resolutions. First, we show that a specific group action on a noetherian $\mathbb{N}$-graded AS-regular algebra of dimension $2$, introduced in \cite{Wei1, Wei2}, is ample. Second, for a noetherian $\mathbb{N}$-graded AS-regular algebra $A$ generated by $A_1$ over $A_0$, we show that the $r$-th quasi-Veronese algebra $A^{[r]}$ is a noncommutative resolution of the $r$-th Veronese algebra $A^{(r)}$. The algebras in these two examples are not necessarily connected graded.
Third, for a finite subgroup $G$ of $\GL_n(k)$ acting freely on the polynomial ring, we prove that the action is ample, without assuming $G \subseteq \SL_n(k)$, and that  $R\#G$ is a noncommutative resolution of $R^G$.

\begin{theorem}[Theorem \ref{G acts freely on polynomial}]
    Suppose that $G$ is a finite subgroup of $\GL_n(k)$ acting on $R=k[x_1,\cdots,x_n]$ freely. Then
    \begin{itemize}
        \item [(1)] the $G$-action on $R$ is ample.
        \item [(2)] $R\#G$ is a noncommutative resolution of $R^G$, given by $({}_{R\#G}R_{R^G},{}_{R^G}R_{R\#G})$.
    \end{itemize}
\end{theorem}

This paper is organized as follows. Section \ref{preliminaries} introduces the basic notation and background.
Section \ref{homological dimensions in qgr A} studies the relationship among homological dimensions in $\qgr A$ and $\QGr A$, investigates the left-right symmetry of noncommutative isolated singularities, and characterizes balanced CM isolated singularities in terms of the modulo-torsion invertibility of their canonical modules.
In Section \ref{equivalent noncommutative quasi-projective spaces}, we investigate the behavior of the properties listed in Diagram 1 under equivalences of noncommutative quasi-projective spaces and introduce noncommutative quasi-resolutions. 
We show that the graded modules defining noncommutative quasi-resolutions satisfy an Ext-orthogonality similar to that of cluster tilting modules.
In Section \ref{noncommutative resolutions of isolated singularities}, we address Question \ref{question}, introduce a more general notion of noncommutative resolution, and prove Theorems \ref{main theorem}, \ref{main theorem CT ncr}. 
In Section \ref{noncommutative resolutions of invariant rings}, we define the homological determinant for Hopf actions on commonly graded algebras. 
We investigate the Gorenstein property of the invariant subring when the homological determinant is trivial and study noncommutative resolutions of the invariant subring when the Hopf action is ample.
In Section \ref{examples}, we illustrate our results with three  examples: invariant rings of group actions on a two-dimensional preprojective algebra, noncommutative resolutions arising from quasi-Veronese algebras, and finite group actions that are free on polynomial rings.

\section{Preliminaries}\label{preliminaries}
\subsection{Notations and conventions}
Let $k$ be a field. A $\mathbb{Z}$-graded $k$-space $X$ is called \textit{locally finite} if $X_i$ is finite-dimensional for all $i\in\mathbb{Z}$; $X$ is called \textit{bounded below} (resp. \textit{bounded above}) if there is an integer $n$ such that $X_i=0$ for all $i<n$ (resp. $i>n$).

A $k$-algebra $A$ with a $\mathbb{Z}$-graded vector space decomposition $$A=\cdots\oplus A_{-1}\oplus A_0\oplus A_1\oplus A_2\oplus\cdots$$ is called a \textit{$\mathbb{Z}$-graded algebra} if $A_iA_j\subseteq A_{i+j}$ for all $i,j\in\mathbb{Z}$. A $\mathbb{Z}$-graded algebra $A$ is called an \textit{$\mathbb{N}$-graded algebra} if $A_i=0$ for all $i<0$. An $\mathbb{N}$-graded algebra $A$ is called \textit{connected graded} if $A_0=k$. A $\mathbb{Z}$-graded algebra $A$ is called \textit{commonly graded} if $A$ is locally finite and bounded below. A noetherian algebra means that it is both left and right noetherian.

In this paper, a graded $A$-module will always mean a right $A$-module $M$ which is a graded $k$-space $M=\oplus_{j \in \mathbb{Z}}M_j$ such that $M_jA_i \subseteq M_{i+j}$ for all $i$ and $j$. Let $A^o$ be the opposite algebra of $A$ and $A^e=A\otimes_k A^o$. 

For any graded $A$-modules $M$ and $N$, a degree-zero graded $A$-module morphism $f:M\to N$ is an $A$-module morphism such that $f(M_i)\subseteq N_i$ for all $i\in\mathbb{Z}$.
The category of graded $A$-modules with morphisms of degree $0$ is denoted by $\Gr A$, and the full subcategory consisting of finitely generated graded $A$-modules is denoted by $\gr A$.

Let $M(n)$ be the $n$-th shift of  a graded module $M$ with $M(n)_i =M_{n+i}$. 
For any $M, N\in \Gr A$, let 
$$\gHom_A(M,N):=\oplus_i\Hom_{\Gr A}(M,N(i)).$$ 
The $n$-th derived functor of $\gHom_A(-,-)$ is denoted by $\gExt_A^n(-,-)$. Let $D(-)=\gHom_k(-,k)$ denote the graded vector space dual, referred to as  Matlis duality. 

For any abelian category $\mathcal{C}$, let $\D(\mathcal{C})$, $\D^+(\mathcal{C}), \D^-(\mathcal{C})$ and $\D^b(\mathcal{C})$ denote the derived categories consisting of unbounded, bounded below, bounded above and bounded complexes of $\mathcal{C}$ respectively. For a cochain complex $X^\bullet$, let $X^\bullet[n]$ be its $n$-th shift, with $(X^\bullet[n])^i=X^{n+i}$.

Suppose $\cM^\bullet\in \D(\cC)$. The \textit{injective dimension} of $\cM^\bullet$ is
    $$\idim_{\cC} \cM^\bullet:=\sup\{i\mid \Hom_{\D(\cC)}(\cX,\cM^\bullet[i])\neq 0  \textrm{ for some } \cX\in\cC\}.$$
The \textit{projective dimension} of $\cM^\bullet$ is
    $$\pdim_{\cC} \cM^\bullet:=\sup\{i\mid \Hom_{\D(\cC)}(\cM^\bullet,\cX[i])\neq 0 \textrm{ for some } \cX\in \cC\}.$$
The \textit{global dimension} of $\cC$ is
    $$\gldim(\cC):=\sup\{i\mid \Hom_{\D(\cC)}(\cM,\cX[i])\neq 0 \textrm{ for some } \cM,\cX\in \cC\}.$$

If $A$ is a bounded below $\mathbb Z$-graded algebra, then $\gldim (\Gr A)=\gldim (\Mod A)$, where $\Mod A$ denotes the category of all right $A$-modules, as shown in \cite[Corollary 7.8]{NO}.
Accordingly, the global dimension of $\Gr A$ is denoted by $\gldim A$ and referred to as the graded global dimension of $A$.

\subsection{Local cohomology}
Let $A$ be a commonly graded algebra, and let $J_A$ (or simply $J$) denote its graded Jacobson radical.
For a graded $A$-module $M$, define
$$\Gamma_A(M):=\{m\in M\mid m\cdot J^i =0 \text{ for } i\gg 0\}.$$
Then, $\Gamma_A(M)\cong \underrightarrow{\Lim}\gHom_A(A/J^i,M)$.
If $J/J^2$ is finite-dimensional, then, by \cite[Lemma 2.6]{LSW},
$\Gamma_A(M)=\{m\in M\mid m\cdot A_{\geqslant i}=0 \text{ for } i\gg 0\}$. If $\Gamma_A(M)=M$ (respectively, $\Gamma_A(M)=0$), then $M$ is called \textit{torsion} (respectively, \textit{torsion-free}).
Let $R^n\Gamma_A$ denote the $n$-th right derived functor of the left exact functor $\Gamma_A\colon \Gr A \to \Gr A$.
The \textit{$n$-th local cohomology} of $M$ is defined by
$$R^n\Gamma_A(M)\cong \underrightarrow{\Lim}\gExt_A^n(A/J^i,M).$$
The \textit{local cohomological dimension} of $A$ is defined by
$$\lcd A:=\sup\{i\mid R^i\Gamma_A(M)\neq 0  \textrm{ for some } M\in \Gr A\}.$$
If $A$ is right noetherian, then $R^i\Gamma_A$ commutes with direct limits. Hence,
$$\lcd A=\sup\{i\mid R^i\Gamma_A(M)\neq 0 \textrm{ for some } M\in \gr A\}.$$

For any graded $A$-module $M$, the \textit{depth} of $M$ is defined as 
$$
\depth_AM:=\inf\{i\mid \gExt_A^i(A/J,M)\neq 0\}.
$$

The following lemma describes the relationship between depth and local cohomology.

\begin{lemma}\cite[Lemma 2.9]{LSW}\label{depth and local cohomology}
    For a graded $A$-module $M$,
    $$\min\{i\mid \gExt_A^i(A/J,M)\neq 0\}=\min\{i\mid R^i\Gamma_A(M)\neq 0\}.$$
\end{lemma}

Next we recall the definition of maximal Cohen-Macaulay modules.

\begin{definition}
    Let $A$ be a noetherian commonly graded algebra with $\lcd A=d$. A finitely generated graded $A$-module $M$ is called a maximal Cohen-Macaulay (for short MCM) $A$-module if $\depth_AM=d$.
\end{definition}

Let $\MCM (A)$ denote the full subcategory of $\Gr A$ consisting of all MCM $A$-modules.
Important examples of MCM modules are cluster tilting modules. The notion of cluster tilting was introduced in higher Auslander-Reiten theory \cite{I1,I2}; these modules were originally termed maximal orthogonal modules.

Let $\add_A M$ denote the full subcategory of $\Gr A$ consisting of all objects that are direct summands of finite direct sums of shifts of $M$.
\begin{definition} \label{cluster tilting module}
    Let $A$ be a right noetherian commonly graded algebra with $\lcd A<\infty$. An MCM $A$-module $M$ is called an \textit{$n$-cluster tilting module} if
    \begin{align*}
        \add_A M&=\{N\in \MCM (A)\mid \gExt_A^i(M,N)=0,\forall\, 0<i<n\}\\
        &=\{N\in \MCM (A)\mid \gExt_A^i(N,M)=0,\forall\, 0<i<n\}.
    \end{align*}
\end{definition}

\subsection{Quotient categories}
For the general theory of quotient categories we mainly refer to \cite[Chapter 4]{Po}.
Suppose that $A$ is a right noetherian commonly graded algebra. Let $\Tor A$ be the full subcategory consisting of graded torsion $A$-modules and $\tor A=\Tor A\cap \gr A$. Let
$$\QGr A=\Gr A/\Tor A \text { and } \qgr A=\gr A/\tor A.$$

Note that both $\QGr A$ and $\qgr A$ are abelian categories, and $\qgr A$ can be regarded as a full subcategory of $\QGr A$ consisting of noetherian objects.
The quotient functor $\Gr A\to \QGr A$ is denoted by $\pi_A$ (or simply $\pi$). Let $\omega_A: \QGr A\to \Gr A$ (or simply $\omega$) be the right adjoint of $\pi:\Gr A\to \QGr A$. Then $\pi\omega\cong \id_{\QGr A}$. Let $s$ denote the auto-equivalent functor of $\QGr A$ induced by the shift functor $(1): \Gr A \to \Gr A$.

For any $\cM,\cN\in \QGr A$, let
$$\Hom_{\cA}(\cM,\cN):=\Hom_{\QGr A}(\cM,\cN) \text{ and } \gHom_{\cA}(\cM,\cN):=\oplus_n\Hom_{\cA}(\cM,s^n\cN).$$
Since $\QGr A$ has enough injective objects, namely, the images of torsion-free injective graded $A$-modules, the derived functor of $\Hom_{\cA}(\cM,-)$ exists.
Let
$$\Ext_{\cA}^i(\cM,\cN)=\Ext_{\QGr A}^i(\cM,\cN) \text{ and } \gExt_{\cA}^i(\cM,\cN)=\oplus_n\Ext_{\cA}^i(\cM,s^n\cN).$$


Although $\qgr A$ may not have enough injective objects or projective objects, Ext groups in $\qgr A$ can be defined through its derived category. The $i$-th derived functor of $\Hom_{\qgr A}(-,-)$ is defined by
$$\Ext_{\qgr A}^i(\cM^\bullet,\cN^\bullet):=\Hom_{\D(\qgr A)}(\cM^\bullet,\cN^\bullet[i])$$
for any $\cM^\bullet,\cN^\bullet\in \D(\qgr A)$.
In fact, for any $\cM^\bullet,\cN^\bullet\in \D^-(\qgr A)$, by \cite[Lemma 4.1]{LW1},
$$\Ext_{\qgr A}^i(\cM^\bullet,\cN^\bullet)\cong \Ext_{\cA}^i(\cM^\bullet,\cN^\bullet).$$
So, we will not distinguish $\Ext_{\qgr A}^i(\cM^\bullet,\cN^\bullet)$ and $\Ext_{\cA}^i(\cM^\bullet,\cN^\bullet)$ for $\cM^\bullet,\cN^\bullet\in\D^-(\qgr A)$ in the rest of this paper.

Throughout this paper, we use a calligraphic letter, such as $\cM$, to denote $\pi M$ for a graded $A$-module $M$.
The following lemmas will be used frequently.

\begin{lemma}\label{cohomology in tails}
    Suppose that $A$ is a right noetherian commonly graded algebra and $M$ is a graded $A$-module. Then
    $R^{i+1}\Gamma_A(M)\cong \gExt_{\cA}^i(\cA,\cM)$ for any $i\geqslant 1$.
\end{lemma}
\begin{proof}
The proof is the same as that of \cite[Proposition 7.2]{AZ} for the $\mathbb{N}$-graded case.
\end{proof}

\begin{lemma}\cite[Corollary 3.4]{LSW}\label{omega pi and depth 2}
Let $A$ be a commonly graded algebra. Then for any $X\in\Gr A$, $\depth_AX\geqslant 2$ if and only if $\omega\pi X\cong X$ canonically.
\end{lemma}


Let $A$ and $B$ be right noetherian commonly graded algebras, and let $F:\gr A\to \gr B$ be a functor. If there is a functor $F':\qgr A\to \qgr B$ such that $\pi_BF\cong F'\pi_A$, then $F'$ is called \textit{induced by} $F$.

We  refer to \cite{Ja} for the notations related to  Morita theory.

\begin{theorem}\cite[Theorem 2.16]{LSW}\label{equ. of quot. cat. induced by Morita context for commonly graded}
Let $(A,B,M,M',\tau,\mu)$ be a graded Morita context where $A,B$ are right noetherian commonly graded algebras and $M_A,M'_B$ are finitely generated. Suppose that ${}_BM$ and ${}_AM'$ are finitely presented. Then the functors 
$$-\otimes_AM':\gr A \rightleftarrows \gr B :-\otimes_BM $$
associated to the Morita context $(A,B,M,M',\tau,\mu)$ induce an equivalence:
$$-\otimes_{\cA}\cM':\qgr A \rightleftarrows \qgr B :-\otimes_{\cB}\cM $$
 if and only if both $\Coker \tau$ and $\Coker\mu$ are finite-dimensional.
\end{theorem}

When Theorem \ref{equ. of quot. cat. induced by Morita context for commonly graded} holds, we also say the equivalence
$$-\otimes_{\cA}\cM':\qgr A \rightleftarrows \qgr B :-\otimes_{\cB}\cM $$
is \textit{induced by} the graded Morita context $(A,B,M,M',\tau,\mu)$.
Theorem \ref{equ. of quot. cat. induced by Morita context for commonly graded} leads to the following definition of modulo-torsion-invertible $(B,A)$-bimodules.

\begin{definition}\label{definition of modulo-torsion-invertible}
Let $A$ and $B$ be noetherian commonly graded algebras. If there is a graded Morita context $(A,B,{}_BM_A,{}_AM'_B,\tau,\mu)$ with ${}_BM_A$ and ${}_AM'_B$ finitely generated on both sides such that $\Coker \tau$ and $\Coker \mu$ are finite-dimensional, then ${}_BM_A$ is called a \textit{modulo-torsion-invertible $(B,A)$-bimodule} with the inverse ${}_AM'_B$ associated to $(A,B,{}_BM_A,{}_AM'_B,\tau,\mu)$.
\end{definition}

\subsection{Noncommutative projective schemes}
The $\chi$-condition introduced in \cite{AZ} plays an important role in noncommutative projective geometry.

\begin{definition}\label{def-chi}
Let $A$ be a commonly graded $k$-algebra, and let $M$ be a graded $A$-module.
\begin{itemize}
    \item [(1)] If $\dim_k\gExt^j_A(A/J,M)<\infty$  for all $j\leqslant i$, then we say that $\chi_i(M)$ holds.

    \item [(2)] If $\chi_i(M)$ holds for every finitely generated graded $A$-module $M$, then we say that $\chi_i$ holds for $A$ or $A$ satisfies the  $\chi_i$-condition.

    \item [(3)] If  $\chi_i$ holds for $A$ for all $i \geqslant 0$, then we say  $\chi$ holds for $A$ or $A$ satisfies the  $\chi$-condition.
\end{itemize}
\end{definition}

\begin{proposition}\label{facts about chi condition}
    Let $A$ be a right noetherian commonly graded algebra.
\begin{itemize}
    \item [(1)] $\chi_i$ holds for $A$ if and only if $R^j\Gamma_A(M)$ is bounded above for all $j\leqslant i$ and all $M\in \gr A$.
    \item [(2)] If $A$ satisfies $\chi_i$, then for any finitely generated graded $A$-modules $M$, $N$ and $j< i$, the canonical map from $\gExt_A^j(M,N)$ to $\gExt_{\cA}^j(\cM,\cN)$ has bounded above kernel and cokernel.
\end{itemize}
\end{proposition}
\begin{proof}
    The proofs are similar to those of \cite[Corollary 3.6(3), Corollary 7.3]{AZ} in the $\mathbb{N}$-graded case.
\end{proof}

To define noncommutative projective schemes, we first introduce the notions of an algebraic triple and an ample pair.

Let $\cC$ be a $k$-linear abelian category, $O$ be an object in $\cC$ and $s$ be a $k$-linear auto-equivalence of $\cC$. Then $(\cC,O,s)$ is called an \textit{algebraic triple}. 

\begin{definition}\cite{AZ}\label{conditions of ample}
Let $(\cC,O,s)$ be an algebraic triple.
The pair $(O,s)$ is called \textit{ample} in $\cC$ if
\begin{itemize}
    \item [(1)] for every object $\cM \in \cC$, there are some positive integers $r_1,r_2,\cdots,r_p$ and an epimorphism $\oplus_{i=1}^ps^{-r_i}O\to\cM$ in $\cC$;
    \item [(2)] for every epimorphism $\cM\to\cN$ in $\cC$, there is an integer $n_0$ such that the induced map $\Hom_{\cC}(s^{-n}O,\cM)\to\Hom_{\cC}(s^{-n}O,\cN)$ is surjective for every $n\geqslant n_0$.
\end{itemize}
\end{definition}

Let $(\cC,O,s)$ be an algebraic triple. Then
$B(\cC,O,s):=\oplus_{i\in \mathbb{Z}}\Hom_{\cC}(O,s^iO)$
is a $\mathbb{Z}$-graded algebra with a natural multiplication, and $\oplus_{i\in \mathbb{Z}}\Hom_{\cC}(O,s^i\cM)$ has a natural graded right $B(\cC,O,s)$-module structure for any $\cM\in\cC$.

The following theorem is a generalization of \cite[Theorem 4.5]{AZ}, which is a noncommutative version of Serre's theorem \cite{Se}.

\begin{theorem}\cite[Theorem  3.9]{LSW}\label{noncommutative Serre theorem for commonly graded algebra}
\begin{itemize}
\item [(1)] Let $(\cC,O,s)$ be an algebraic triple. Suppose that it satisfies
   \begin{itemize}
     \item [H1.] $O\in \cC$ is a noetherian object;
     \item [H2.] $\Hom_{\cC}(O,\cM)$ is finite-dimensional for all $\cM\in \cC$;
     \item [H3.] $(O,s)$ is ample in $\cC$;
     \item [H4.] $B(\cC,O,s)$ is bounded below.
   \end{itemize}

Then $B:=B(\cC,O,s)$ is a right noetherian commonly graded $k$-algebra  satisfying the $\chi_1$-condition and $\depth_B(B)\geqslant 2$. The functor
$$\cC\to \qgr B,\,\,\cM\mapsto \pi(\oplus_{i\in \mathbb{Z}}\Hom_{\cC}(O,s^i\cM))$$
gives an isomorphism between algebraic triples
$(\cC,O,s)$ and $(\qgr B,\cB,s)$.
\item [(2)] Let $A$ be a right noetherian commonly graded algebra satisfying the $\chi_1$-condition. Then $H1$, $H2$, $H3$ hold for $(\qgr A, \cA, s)$, and there is a canonical morphism $A\to B$ such that both its kernel and cokernel are bounded above, where $B=B(\qgr A,\cA,s)$.

    If $B=B(\qgr A,\cA,s)$ is bounded below (as is the case when $\depth_A(A)\geqslant 2$), then there is an isomorphism  between $(\qgr A,\cA, s)$ and  $(\qgr B,\cB,s)$.
\end{itemize}
\end{theorem}

Based on Theorem \ref{noncommutative Serre theorem for commonly graded algebra}, we give a definition of noncommutative projective scheme, which is stronger than the one defined in \cite{AZ}.

\begin{definition}\label{def-noncom-proj-scheme}
    Let $\cC$ be a $k$-linear abelian category and $s$ be a $k$-linear auto-equivalence of $\cC$. If, for some object $O$ in $\cC$, the algebraic triple $(\cC,O,s)$ satisfies $H1$, $H2$, $H3$ and $H4$ in Theorem \ref{noncommutative Serre theorem for commonly graded algebra}, then
    \begin{itemize}
        \item [(1)] $(\cC,O,s)$ is called a \textit{noncommutative projective scheme};
        \item [(2)] $O$ is called a \textit{structure sheaf} of $(\cC,O,s)$;
        \item [(3)] $s$ is called a \textit{polarization} of $(\cC,O,s)$;
        \item [(4)] $(\cC,s)$ is called a \textit{noncommutative quasi-projective space};
        \item [(5)] $B(\cC,O,s)$ is called a \textit{noncommutative projective coordinate ring}.
    \end{itemize}
\end{definition}

By Theorem \ref{noncommutative Serre theorem for commonly graded algebra}, $A$ is a noncommutative projective coordinate ring if and only if it is a right noetherian commonly graded algebra satisfying the $\chi_1$-condition and $\depth_A(A)\geqslant 2$. In this case, $(\qgr A,\cA,s)$ is a noncommutative projective scheme, and $A \cong B(\qgr A,\cA,s)$. Hence under the isomorphism, every noncommutative projective scheme (resp. noncommutative quasi-projective space) is of the form $(\qgr A,\cA,s)$ (resp. $(\qgr A,s)$).

\subsection{Artin-Schelter properties for commonly graded algebras}
We first introduce AS-Gorensteiness (resp. AS-regularity) for commonly graded algebras. For more details, see \cite{LW2} and \cite{RR}.

\begin{definition}\label{def-generalized-AS-Gorenstein}
A commonly graded algebra $A$ is called Artin-Schelter Gorenstein (for short, AS-Gorenstein) of dimension $d$ if the following conditions hold.
\begin{itemize}
\item[(1)] $A$ has left and right graded injective dimension $d$.
\item[(2)] For every graded simple $A$-module $M$, $\gExt^i_A(M,A)=0$ if $i\neq d$; for every graded simple $A^o$-module $N$, $\gExt_{A^o}^i(N,A)=0$ if $i\neq d$.
\item[(3)] $\gExt_A^d(-,A)$ and $\gExt_{A^o}^d(-,A)$ give a bijection between the isomorphism classes of graded simple $A$-modules and the isomorphism classes of graded simple $A^o$-modules.
\end{itemize}
If, moreover, $A$ has finite graded global dimension $d$, then $A$ is called Artin-Schelter regular (for short, AS-regular).

If $A$ is $\mathbb{N}$-graded, then it is called an $\mathbb{N}$-graded AS-Gorenstein (regular) algebra.
\end{definition}

An important property of noetherian commonly graded AS-Gorenstein algebras is the existence of balanced dualizing complexes. A balanced dualizing complex is defined as follows (see \cite{Ye} and \cite{V3}).

\begin{definition}
    Let $A$ be a noetherian commonly graded algebra. A complex $R\in \D^b(\Gr A^e)$ is called a {\it dualizing complex} of $A$, if it satisfies the following conditions:
\begin{itemize}
\item[(1)] $R$ has finite injective dimension over $A$ and $A^o$, respectively.
\item[(2)] The cohomologies of $R$ are finitely generated as $A$-modules and $A^o$-modules.
\item[(3)] The natural morphisms $ A\to R\gHom_A(R,R)$ and $A\to R\gHom_{A^o}(R,R)$ are isomorphisms in $\D(\Gr A^e)$.
\end{itemize}
If moreover, $R\Gamma_A(R)\cong D(A)$ and $R\Gamma_{A^o}(R)\cong D(A)$ in $\D(\Gr A^e)$, then $R$ is called a {\it balanced dualizing complex} of $A$.
\end{definition}

Local duality and an existence theorem for balanced dualizing complexes are proved in \cite{V3} for connected graded algebras. The results still hold for commonly graded algebras.

\begin{theorem}[Local Duality]\label{local duality}\cite[Theorem 5.1]{V3}
Suppose that $A$ is a noetherian commonly graded algebra 
with finite local cohomology dimension. Then
\begin{itemize}
\item[(1)] The injective dimension of $D(R\Gamma_A(A))$ is $0$ as an object in $\D(\Gr A)$.
\item[(2)] For any graded algebra $B$, and $M^\bullet\in \D(\Gr A\otimes B^o)$,
$$D(R\Gamma_A(M^\bullet))\cong R\gHom_A(M^\bullet,D(R\Gamma_A(A)))$$
in $\D(\Gr B\otimes A^o)$.
\end{itemize}
\end{theorem}

\begin{theorem}\label{existence of balanced dualizing complex}\cite[Theorem 6.3]{V3}
Let $A$ be a noetherian commonly graded algebra. Then $A$ admits a balanced dualizing complex if and only if $A$ satisfies the following two conditions:
\begin{itemize}
\item[(1)] both $A$ and $A^o$ satisfy the $\chi$-condition;
\item[(2)] both $\lcd A$ and $\lcd A^o$ are finite.
\end{itemize}
If $A$ admits a balanced dualizing complex $R$, then $R \cong D(R\Gamma_A(A))\cong D(R\Gamma_{A^o}(A))$.
\end{theorem}

We consider algebras whose balanced dualizing complexes are isomorphic to shifted stalk complexes.

\begin{definition}
    A noetherian commonly graded algebra $A$ is called a {\it balanced Cohen-Macaulay (CM for short) algebra} of dimension $d$ if it admits a balanced dualizing complex $R$ such that $R\cong D(R^d\Gamma_A(A))[d]$ in $\D(\Gr A^e)$. 
    The $(A,A)$-bimodule $D(R^d\Gamma_A(A))$ is called the \textit{canonical module} of $A$.
\end{definition}

Noetherian commonly graded AS-Gorenstein algebras are balanced CM. 

\begin{proposition}\label{AS-G has balanced dualizing complex and chi}
Let $A$ be a noetherian commonly graded algebra.
\begin{itemize}
    \item [(1)] If $A$ is a commonly graded AS-Gorenstein algebra of dimension $d$, then $A$ is a balanced CM algebra of dimension $d$ and $D(R^d\Gamma_A(A))$ is an invertible $(A,A)$-bimodule. 
    \item [(2)] If $A$ is a balanced CM algebra of dimension $d$, then both $A$ and $A^o$ satisfy the $\chi$-condition, $\lcd A=\lcd A^o=d$, and the graded injective dimension of $D(R^d\Gamma_A(A))$ is $d$ on both sides.
\end{itemize}
\end{proposition}
\begin{proof}
    (1) follows from \cite[Theorem 4.12]{LW2}. (2) follows from Theorem \ref{existence of balanced dualizing complex} and Theorem \ref{local duality}.
\end{proof}

If $A$ is a balanced CM algebra of dimension $d\geqslant 2$, then $\depth_A(A)\geqslant 2$ by Lemma \ref{depth and local cohomology} and $A$ satisfies the $\chi$-condition by Proposition \ref{AS-G has balanced dualizing complex and chi}. So $A$ is a noncommutative projective coordinate ring.

Next, we introduce several well-known results concerning MCM modules.

\begin{proposition}\label{Hom(-,Omega) induce duality on MCM}
    Let $A$ be a balanced CM algebra of dimension $d$ and $\Omega=D(R^d\Gamma_A(A))$. 
    \begin{itemize}
        \item [(1)] Suppose that $M$ is a finitely generated graded $A$-module. Then $M$ is MCM if and only if $\gExt_A^i(M,\Omega)=0$ for all $i>0$.
        \item [(2)] There is a duality
    $$\gHom_A(-,\Omega):\MCM (A)\rightleftarrows \MCM (A^o):\gHom_{A^o}(-,\Omega).$$
        \item [(3)] For $X,Y\in\MCM A$, let $X^\dagger=\gHom_A(X,\Omega)$ and $Y^\dagger=\gHom_A(Y,\Omega)$. Then, for any $i\geqslant 0$, there are natural isomorphisms of graded vector spaces
$$\gExt_A^i(X,Y)\cong\gExt_{A^o}^i(Y^\dagger,X^\dagger).$$
        \item [(4)] If $A$ is commonly graded AS-Gorenstein, then there is a duality
        $$\gHom_A(-,A):\MCM (A)\rightleftarrows \MCM (A^o):\gHom_{A^o}(-,A).$$
    \end{itemize}
\end{proposition}
\begin{proof}
The proofs of (1) and (2) follow the argument of \cite[Lemma 4.6]{Mor1}, which also applies to commonly graded algebras. By (1), $R\gHom_A(X,\Omega)\cong X^\dagger$ for $X\in\MCM A$; hence, the derived duality induced by the dualizing complex $\Omega[d]$ gives (3). Statement (4) follows from \cite[Proposition 2.25]{LSW}.
\end{proof}

\section{Homological dimensions and noncommutative isolated singularities}\label{homological dimensions in qgr A}
In this section, we first study homological dimensions in the category $\qgr A$ for a right noetherian commonly graded algebra $A$. 
Then we introduce the notion of balanced CM isolated singularities and study their properties. 

\subsection{Homological dimensions in noncommutative quasi-projective spaces}Recall the definition of noncommutative isolated singularities. 

\begin{definition}
    Let $A$ be a right noetherian commonly graded algebra. Then $A$ is called a {\it noncommutative isolated singularity} if $\gldim(\qgr A)$ is finite.
\end{definition}




Here are some useful lemmas.

\begin{lemma}\label{lemma for qgr A}
    Let $A$ be a right noetherian commonly graded algebra.
    \begin{itemize}
        \item [(1)] For any $\cX^\bullet\in\D^+(\QGr A)$ and $n$, if $\gExt_{\cA}^n(\cA,\cX^\bullet)=0$, then $H^n(\cX^\bullet)=0$.
         \item [(2)] For any $\cX^\bullet\in \D^b(\qgr A)$ and $n\in\mathbb{Z}$, the functor $\Ext^n_{\cA}(\cX^\bullet,-)$ commutes with direct limits; that is, for every directed system $\{\cN_i\}_i$ in $\QGr A$,
    $$\Ext^n_{\cA}(\cX^\bullet,\dlim_i\cN_i)\cong \dlim_i\Ext^n_{\cA}(\cX^\bullet,\cN_i).$$
    \end{itemize}
\end{lemma}
\begin{proof}
(1) Let $\cX^\bullet\to \cI^\bullet$ be a minimal injective resolution of $\cX^\bullet \in\D^+(\QGr A)$. Then 
$$\gExt_{\cA}^n(\cA,\cX^\bullet)=H^n(\gHom_{\cA}(\cA,\cI^\bullet))\cong H^n(\omega\cI^\bullet)=0.$$ 
Hence, $H^n(\cX^\bullet)\cong H^n(\cI^\bullet) \cong H^n(\pi\omega\cI^\bullet)\cong \pi H^n(\omega\cI^\bullet) =0$.

(2)  If $\cX\in \qgr A$, $\Ext^n_{\cA}(\cX,-)$ commutes with direct limits by \cite[Lemma 4.3.1]{BV}, that is, 
    $\Ext^n_{\cA}(\cX,\dlim_i\cN_i)\cong \dlim_i\Ext^n_{\cA}(\cX,\cN_i)$.

Suppose 
$\cX^\bullet=\cdots \to 0\to \cX^t\to\cX^{t+1}\to \cdots\to\cX^{s-1}\to \cX^s\to 0\to \cdots.$ 
The conclusion follows from an induction on $s-t$.
\end{proof}

The following lemma shows that, for any $\cX\in\qgr A$, the injective and projective dimensions of $\cX$ in $\qgr A$ coincide with those in $\QGr A$.

\begin{lemma}\label{idim cX and pdim cX}
Let $A$ be a right noetherian commonly graded algebra. 
\begin{itemize}
    \item [(1)] If $\cX^\bullet \in \D^+(\QGr A)$, then
    \begin{align*}
        \idim_{\QGr A}\cX^\bullet &=\sup\{i\mid \Ext_{\cA}^i(\cM,\cX^\bullet)\neq 0 \textrm{ for some } \cM\in \qgr A\}.
      \end{align*}

    \item [(2)] If $\cX^\bullet\in \D^b(\qgr A)$, then
     $$\pdim_{\QGr A}\cX^\bullet=\sup\{i\mid \Ext_{\cA}^i(\cX^\bullet,\cN)\neq 0 \textrm{ for some } \cN\in \qgr A\}.$$

     \item [(3)] If $\cX^\bullet\in \D^b(\qgr A)$ has finite projective dimension, then
     $$\pdim_{\QGr A}\cX^\bullet=\sup\{i\mid \gExt_{\cA}^i(\cX^\bullet,\cA)\neq 0\}.$$

    \item [(4)] $\gldim(\qgr A)=\sup\{\idim_{\QGr A}\cN\mid\cN\in\qgr A\}.$

    \item [(5)] If $A$ is a noncommutative isolated singularity, then $$\gldim(\qgr A)=\idim_{\QGr A}\cA.$$
\end{itemize}
\end{lemma}
\begin{proof}

(1) 
Let 
$\cI^\bullet$ be a minimal injective resolution of $\cX^\bullet$.
Suppose $\idim_{\QGr A}\cX^\bullet = \infty$.
If there is an integer $n$ such that, for any $i>n$ and any $\cM\in\qgr A$, $\Ext_{\cA}^i(\cM,\cX^\bullet)=0$,
then $\cX^\bullet$ is exact at the $i$-th place whenever $i>n$ by Lemma \ref{lemma for qgr A}.
Let $\cL=\im(\cI^n \to \cI^{n+1})$. 
Then $0\to \cL\to \cI^{n+1}\to \cI^{n+2}\to \cdots$ is a minimal injective resolution  of $\cL$. Hence, 
$$\Ext_{\cA}^j(\cM,\cL)=\Ext_{\cA}^{j+n+1}(\cM,\cX^\bullet)=0$$ for any $j\geqslant 1$ and any $\cM\in\qgr A$. 
By \cite[Lemma 4.2]{LW1}, $\idim_{\QGr A}\cL= 0$.
Thus $\cL$ is injective, contradicting $\idim_{\QGr A}\cX^\bullet = \infty$.
It follows that
\[\sup\{i\mid \Ext_{\cA}^i(\cM,\cX^\bullet)\neq 0 \text{ for some } \cM\in \qgr A\}=\infty.\]

Now suppose $\idim_{\QGr A}\cX^\bullet=n < \infty$. 
Then  $\idim_{\QGr A}\cX^\bullet=\sup\{i\mid \cI^i\neq 0 \}=n$.
If $\Ext_{\cA}^n(\cA,\cX^\bullet)\neq 0$, then the proof is finished. 
If $\Ext_{\cA}^n(\cA,\cX^\bullet)=0$, then  $H^n(\cI^\bullet)=0$ by Lemma \ref{lemma for qgr A}. 
Let $\cK=\Ker(\cI^{n-1}\to \cI^n)$. Then $0\to \cK\to \cI^{n-1}\to \cI^n\to 0$ is a minimal injective resolution of $\cK$. Hence, by \cite[Lemma 4.2]{LW1}, there exists $\cM\in \qgr A$ such that $\Ext_{\cA}^1(\cM,\cK)\neq 0$. 
It follows that $\Ext_{\cA}^n(\cM,\cX^\bullet) \cong \Ext_{\cA}^1(\cM,\cK) \neq 0$.
Thus,
$$\idim_{\QGr A}\cX^\bullet=\max\{i\mid \Ext_{\cA}^i(\cM,\cX^\bullet)\neq 0 \text{ for some } \cM\in \qgr A\}=n.$$ 

(2) For any $\cN\in \QGr A$, suppose that $\omega\cN=\dlim_iN_i$ where $N_i$ runs over all finitely generated graded submodules of $N$. Then $\cN\cong \dlim_i\pi N_i$. 
Since $\pi N_i\in\qgr A$,
(2) follows from Lemma \ref{lemma for qgr A} (2).

(3) Suppose that $\pdim_{\QGr A}\cX^\bullet=n < \infty$. By (2), there exists some $\cN\in \qgr A$ such that $\gExt_{\cA}^n(\cX^\bullet,\cN)\neq 0$. 
Then there is an exact sequence in $\qgr A$:
    $$0\to \cL\to \oplus_{i=1}^m \cA(r_i)\to \cN\to 0.$$
    It follows from the long exact sequence
    $$\cdots\to \gExt_{\cA}^n(\cX^\bullet,\cL)\to\gExt_{\cA}^n(\cX^\bullet, \oplus_{i=1}^m \cA(r_i))\to \gExt_{\cA}^n(\cX^\bullet,\cN)\to 0$$
    that 
    $\gExt_{\cA}^n(\cX^\bullet,\cA)\neq 0$.

    (4) This follows from (1); alternatively, see \cite[Lemma 4.2]{LW1}.

    (5) This follows from (3).
\end{proof}

\begin{corollary}\label{homo. dim. in qgr and QGr}
    If $\cX^\bullet\in \D^b(\qgr A)$, then
    $$\idim_{\QGr A}\cX^\bullet=\idim_{\qgr A}\cX^\bullet \text{ and } \pdim_{\QGr A}\cX^\bullet=\pdim_{\qgr A}\cX^\bullet.$$
\end{corollary}

By Corollary \ref{homo. dim. in qgr and QGr}, for any $\cX^\bullet\in\D^b(\qgr A)$, the injective and projective dimensions of $\cX^\bullet$ in $\D^b(\qgr A)$ coincide with those of $\cX^\bullet$ in $\D^b(\QGr A)$. Hence, for convenience, we denote them by $\idim_{\cA}\cX^\bullet$ and $\pdim_{\cA}\cX^\bullet$, respectively.

The following lemma describes the local cohomological dimension of $A$.

\begin{lemma}\label{local cohomological dimension and qgr}
Let $A$ be a right noetherian commonly graded algebra.
\begin{itemize}
    \item [(1)] $\lcd A=\pdim_{\cA}\cA+1$.
    \item [(2)] If $\lcd A$ is finite, then $\lcd A=\max\{i\mid R^i\Gamma_A(A)\neq 0\}$.
    \item [(3)] If $A$ is a noncommutative isolated singularity, then $\lcd A$ is finite.
\end{itemize}
\end{lemma}
\begin{proof}
    (1) For any $M\in \Gr A$ and $i>1$, by Lemma \ref{cohomology in tails},
    $$R^i\Gamma_A(M)\cong \gExt_{\cA}^{i-1}(\cA,\cM).$$
    It follows that $\lcd A=\pdim_{\cA}\cA+1$.

    (2) By \cite[Proposition 7.10]{AZ}.
    
    (3) It follows from (1). 
\end{proof}

Since $\pi:\Gr A\to \QGr A$ is exact, it induces a functor $\pi:\D(\gr A)\to \D(\qgr A)$ between the derived categories. 
If $A$ is noetherian and admits a balanced dualizing complex $R$, then $R\gHom_A(-,R)$ induces a functor $R\gHom_{\cA}(-,\pi R): \D^b(\qgr A) \to \D^b(\qgr A^o)$ and so does $R\gHom_{A^o}(-,R)$. In particular, 
$$R\gHom_{\cA}(-,\pi R):\D^b(\qgr A)\rightleftarrows \D^b(\qgr A^o):R\gHom_{\cA^o}(-,\pi R)$$ 
is a duality between $\D^b(\qgr A)$ and $\D^b(\qgr A^o)$.
Moreover, if $A$ is a noncommutative isolated singularity, then $\D^b(\qgr A)$ admits Serre duality, with Serre functor $-\otimes_{\cA}^L \pi R[-1]$ induced by $-\otimes_A^L R[-1]$. 
These facts are proved in \cite[Appendix]{NV} for connected graded algebras. Both the results and the proofs remain valid for commonly graded algebras. 

\begin{theorem}\cite[Proposition A.3, Theorem A.4]{NV}.\label{Serre duality in D(qgr A)}
    Let $A$ be a noetherian commonly graded algebra with a balanced dualizing complex $R$. For any $\cM^\bullet,\cN^\bullet\in\D^b(\qgr A)$, if $\cM^\bullet$ has finite projective dimension, there is a natural isomorphism
    $$\Hom_{\D^b(\qgr A)}(\cM^\bullet,\cN^\bullet)\cong D(\Hom_{\D^b(\qgr A)}(\cN^\bullet,\cM^\bullet\otimes_{\cA}^L\pi R[-1])).$$
    In particular, if $A$ is a noncommutative isolated singularity, then $-\otimes_{\cA}^L\pi R[-1]$ is a Serre functor on $\D^b(\qgr A)$ and is an auto-equivalence.
\end{theorem}


\begin{lemma}\label{homo. dim. in qgr A with b.d.c.}
    Let $A$ be a noetherian commonly graded algebra with a balanced dualizing complex $R$. Then
    \begin{itemize}
        \item [(1)] $\idim_{\cA}(\pi R)=\idim_{\cA^o}(\pi R)=-1$.
        \item [(2)] 
        $\lcd A=\lcd A^o=\pdim_{\cA}\cA+1=\pdim_{\cA^o}\cA+1\\
        \text{} \quad\quad\,=-\min\{i\mid H^i(\pi_A R)\neq 0\}=-\min\{i\mid H^i(\pi_{A^o} R)\neq 0\}.$
    \end{itemize}
\end{lemma}
\begin{proof}
(1)  For any $\cX\in\qgr A$, by Theorem \ref{Serre duality in D(qgr A)},
$$\Ext_{\cA}^i(\cA,\cX)\cong D\Ext^{-i}_{\cA}(\cX,\cA\otimes_{\cA}^L\pi R[-1])\cong D\Ext_{\cA}^{-i-1}(\cX,\pi R).$$

Since $\Ext_{\cA}^i(\cA,\cX)=0$ for all $\cX\in\qgr A$ and $i<0$, it follows that $\Ext_{\cA}^n(\cX,\pi R)=0$ for all $n\geqslant 0$.
On the other hand, for any $\cX\in\qgr A$, we have $\Ext_{\cA}^0(\cA,\cX)=\Hom_{\cA}(\cA,\cX)\cong (\omega\cX)_0$,
which implies that
$\Ext_{\cA}^{-1}(\cX',\pi R)\neq 0$ for some $\cX'$. Therefore, $\idim_{\cA}(\pi R)=-1$. 
By a dual argument, $\idim_{\cA^o}(\pi R)=-1$.

(2) The proof is similar to the connected graded case \cite[Theorem 4.2]{YZ}.
%
\end{proof}

\begin{corollary}
    Let $A$ be a balanced CM algebra of dimension $d$, and let $\Omega=D(R^d\Gamma_A(A))$. Then $E^d(\Omega)$ is torsion, where $E^d(\Omega)$ is the $d$-th term of the minimal graded injective resolution of $\Omega_A$.

    Moreover, if $A$ is commonly graded AS-Gorenstein, then $E^d(A)$ is torsion, where $E^d(A)$ denotes the $d$-th term in the minimal graded injective resolution of $A_A$.
\end{corollary}
\begin{proof}
    By Lemma \ref{homo. dim. in qgr A with b.d.c.}, the injective dimension of $\pi \Omega[d]$ is $-1$. Hence, $\idim_{\cA}\pi\Omega=d-1$. Since $\pi$ preserves minimal injective resolutions by \cite[Lemma 5.4]{LSW}, it follows that $\pi E^d(\Omega)=0$. Therefore, $E^d(\Omega)$ is torsion. 
        
    If $A$ is commonly graded AS-Gorenstein, then $\Omega$ is an invertible $(A,A)$-bimodule by Proposition \ref{AS-G has balanced dualizing complex and chi}. 
    Since $E^d(\Omega)$ is torsion, it follows that $E^d(A)$ is torsion.
\end{proof}

The next proposition characterizes the projective dimension of objects in $\qgr A$ for commonly graded AS-Gorenstein algebras $A$.

\begin{proposition}\label{homological dimension of qgr A 2:ASG iso. sing.}
    Let $A$ be a noetherian commonly graded AS-Gorenstein algebra of dimension $d$. 
  For any non-zero $\cM\in \qgr A$ with finite projective dimension, 
        $$\gExt_{\cA}^{d-1}(\cM,\cA)\neq 0 \text{ and } \pdim_{\cA}\cM=d-1.$$
\end{proposition}
\begin{proof}
    Let $U=D(R^d\Gamma_A(A))$. By Proposition \ref{AS-G has balanced dualizing complex and chi},  $U$ is an invertible $(A,A)$-bimodule.  
    For any $M,X\in \gr A$ with $\cM\neq 0$, if $\pdim_{\cA}\cM$ is finite, then Theorem \ref{Serre duality in D(qgr A)} yields an isomorphism
    $$\gExt_{\cA}^i(\cM,\cX)\cong D(\gExt_{\cA}^{d-i-1}(\cX,\pi(M\otimes_A U))).$$
    Therefore, $\pdim_{\cA}\cM\leqslant d-1$.
    
    Taking $\cX=\cA$, we obtain
    $$\gExt_{\cA}^{d-1}(\cM,\cA)\cong D(\gHom_{\cA}(\cA,\pi(M\otimes_A U))\cong D(\omega\pi (M\otimes_AU))).$$
    Since $U$ is invertible, $\pi(M\otimes_A U) \neq 0$. 
    Therefore, for any non-zero $\cM\in \qgr A$ with finite projective dimension, $\gExt_{\cA}^{d-1}(\cM,\cA)\neq 0$, and $\pdim_{\cA}\cM=d-1$.
\end{proof}

In the final part of this section, we show that if $A$ has a balanced dualizing complex, then 
$\gldim(\qgr A)=\gldim(\qgr A^o)$.

\begin{lemma}\label{bound of Ext for bounded complex}
    Let $\cC$ be an abelian category. Suppose that
    $$M^\bullet:\cdots \to 0\to M^t\to M^{t+1}\to \cdots \to M^{s-1}\to M^s\to 0\to \cdots \textrm{ and }$$
    $$N^\bullet:\cdots \to 0\to N^{t'}\to N^{t'+1}\to \cdots \to N^{s'-1}\to N^{s'}\to 0\to \cdots$$ are two complexes in $\D^b(\cC)$.
    If $\gldim \cC=n$, then $\Hom_{\D(\cC)}(M^\bullet,N^\bullet[i])=0$ for all $i>n+s'-t$.
\end{lemma}
\begin{proof}
   If $s'-t'=0$, we may assume $N^\bullet=\cdots\to 0\to N^0\to 0\to \cdots$. It follows from an induction  on $s-t$ that $\Hom_{\D(\cC)}(M^\bullet,N^\bullet[i])=0$ for all $i>n-t$.

   Now we argue by induction on the width of the bounded complex $N^\bullet$; we have proved the assertion when the width is $0$.
    Let $\tilde{N}^\bullet$ be the complex 
    $$\cdots\to 0\to N^{t'}\to \cdots\to N^{s'-1}\to 0\to \cdots$$ 
    and regard $N^{s'}$ as a stalk complex. Then there is a triangle in $\D(\cC)$:
    $$N^{s'}[-s']\to N^\bullet\to \tilde{N}^\bullet$$
    which induces the following exact sequence
    $$\Hom_{\D(\cC)}(M^\bullet,N^{s'}[-s'][i])\to \Hom_{\D(\cC)}(M^\bullet,N^\bullet[i])\to \Hom_{\D(\cC)}(M^\bullet,\tilde{N}^\bullet[i]).$$
    If $i>n+s'-t$, then $\Hom_{\D(\cC)}(M^\bullet,N^{s'}[-s'][i])=0$, and if $i>n+s'-1-t$, $\Hom_{\D(\cC)}(M^\bullet,\tilde{N}^\bullet[i])=0$ by the induction hypothesis.
    It follows that for all $i>n+s'-t$, $\Hom_{\D(\cC)}(M^\bullet,N^\bullet[i])=0$.
\end{proof}

\begin{theorem}\label{A and A^o are noncommutative iso. sing.}
    Let $A$ be a noetherian commonly graded algebra with a balanced dualizing complex. Then
    \begin{itemize}
        \item [(1)] $A$ is a noncommutative isolated singularity if and only if its opposite algebra $A^o$ is a noncommutative isolated singularity.
        \item [(2)] $\gldim(\qgr A)=\gldim(\qgr A^o)$.
    \end{itemize}
\end{theorem}
\begin{proof}
    (1) Let $R$ be the balanced dualizing complex of $A$. 
    Suppose $\gldim(\qgr A)=n<\infty$. It suffices to show that $\gldim(\qgr A^o)$ is finite.
    Since $R$ has finite injective dimension over $A^o$, it has a bounded graded injective resolution. 
    Thus there are integers $t\leqslant s$, such that for any $\cX\in \qgr A^o$ and $j\notin[t,s]$,
    $$H^j(R\gHom_{\cA^o}(\cX,\pi R))=0.$$
    It follows from the comments before Theorem \ref{Serre duality in D(qgr A)} that $R\gHom_{\cA^o}(\cX,\pi R)$ belongs to $\D^b(\qgr A)$.
    By truncation, it can be represented by a complex in $\qgr A$ whose terms are concentrated in degrees $t,\cdots,s$.
    For any $\cX,\cY\in\qgr A^o$, the duality preceding Theorem \ref{Serre duality in D(qgr A)} gives
    $$\Ext^i_{\cA^o}(\cX,\cY)\cong
    \Ext^i_{\cA}(R\gHom_{\cA^o}(\cY,\pi R),R\gHom_{\cA^o}(\cX,\pi R)).$$
    By Lemma \ref{bound of Ext for bounded complex}, these groups vanish for all $i>n+s-t$. Hence $\gldim(\qgr A^o)$ is finite. The converse follows by symmetry.

    (2) By (1), it suffices to consider the case in which both global dimensions are finite.
We first show that $H^i(R\otimes_A^L R)$ is finitely generated on both sides for every $i$.
Since $A$ is noetherian, if the $(A,A)$-bimodules $M$ and $N$ are finitely generated on both sides, then $\Tor_s^A(M,N)$ is finitely generated on both sides for every $s\in\mathbb N$. 

Let $N$ be a graded $(A,A)$-bimodule that is finitely generated on both sides.
Then there is a convergent spectral sequence
\[
E_2^{-s,q}=\Tor_s^A(H^q(R),N)
\quad\Longrightarrow\quad H^{q-s}(R\otimes_A^L N)
\]
By the preceding observation, each term on the $E_2$-page is finitely generated on both sides. 
The same holds for all terms in subsequent pages, as they are subquotients of those on the preceding page.
Since $R$ is bounded, for a fixed $n$, there are only finitely many pairs $(s,q)$ satisfying $q-s=n$ and $E_2^{-s,q}\neq 0$. 
Thus the spectral sequence gives a finite filtration of $H^n(R\otimes_A^L N)$ whose successive quotients are the corresponding $E_\infty$-terms.
Since finite extensions of finitely generated modules are still finitely generated, $H^n(R\otimes_A^L N)$ is finitely generated on both sides for every $n$.

Next, apply $R\otimes_A^L-$ to the second factor $R$. This gives a convergent spectral sequence
\[
E_2^{p,q}=H^p\bigl(R\otimes_A^L H^q(R)\bigr)
\quad\Longrightarrow\quad
H^{p+q}(R\otimes_A^L R).
\]
Each $H^q(R)$ is finitely generated on both sides, so the preceding result shows that every $E_2^{p,q}$ is finitely generated on both sides, and the same is true for the terms in all subsequent pages.
Again by the fact that $R$ is bounded,
a similar argument as before shows that $H^n(R\otimes_A^L R)$ is finitely generated on both sides for every $n$.



    By Theorem \ref{Serre duality in D(qgr A)}, the images of $(R\otimes_A^L R)[-1]$ under $\pi_A$ and $\pi_{A^o}$ belong to $\D^b(\qgr A)$ and $\D^b(\qgr A^o)$, respectively.
    We may therefore choose an integer $m$ sufficiently small and take the mild truncation
    $$M^\bullet=\tau_{\geqslant m}((R\otimes_A^L R)[-1])$$
    such that $\pi_AM^\bullet\cong \pi_A (R\otimes_A^L R)[-1]$ and $\pi_{A^o}M^\bullet\cong \pi_{A^o} (R\otimes_A^L R)[-1]$.
    Then the cohomologies of $M^\bullet$ are finitely generated on both sides. Moreover $M^\bullet$ is bounded.

    By the proof of \cite[Corollary 4.8]{V3}, which also applies to commonly graded algebras, we have
    $R\Gamma_A(M^\bullet)\cong R\Gamma_{A^o}(M^\bullet).$
    By Theorem \ref{local duality}, the complexes $R\gHom_A(M^\bullet,R)$ and $R\gHom_{A^o}(M^\bullet,R)$ consequently have isomorphic cohomologies.
    It follows that $H^i(R\gHom_A(M^\bullet,R))$ is finite-dimensional if and only if so is $H^i(R\gHom_{A^o}(M^\bullet,R))$.
As a result, $R\gHom_{\cA}(\pi_A M^\bullet,\pi R)$ and $R\gHom_{\cA^o}(\pi_{A^o}M^\bullet,\pi R)$ have nonzero cohomologies in the same degrees.

    For any $\cX\in\qgr A$, by the duality before Theorem \ref{Serre duality in D(qgr A)} and the Serre duality formula for $A^o$, we have
    \begin{align*}
        \Ext^i_{\cA}(\cX,\cA)
        &\cong \Ext^i_{\cA^o}(\pi R,R\gHom_{\cA}(\cX,\pi R))\\
        &\cong D\Ext^{-i}_{\cA^o}(R\gHom_{\cA}(\cX,\pi R),\pi_{A^o}M^\bullet)\\
        &\cong D\Ext^{-i}_{\cA}(R\gHom_{\cA^o}(\pi_{A^o}M^\bullet,\pi R),\cX).
    \end{align*}
    Here the choice of $M^\bullet$ is used in the second isomorphism.
    For any nonzero $\cY^\bullet\in\D^b(\qgr A)$, put $b=\max\{j\mid H^j(\cY^\bullet)\neq0\}$. 
   Then $\Ext^{-i}_{\cA}(\cY^\bullet,\cX)=0$ for all $i>b$ and $\cX\in\qgr A$.
   On the other hand, the canonical morphism
    $\cY^\bullet\to H^b(\cY^\bullet)[-b]$ is nonzero. 
    Taking $\cY^\bullet =R\gHom_{\cA^o}(\pi_{A^o}M^\bullet,\pi R)$ and applying this observation to the last isomorphism shows
    $$\idim_{\cA}\cA=
    \max\{i\mid H^i(R\gHom_{\cA^o}(\pi_{A^o}M^\bullet,\pi R))\neq0\}.$$
    Interchanging $A$ and $A^o$ gives the analogous formula for $\idim_{\cA^o}\cA$.
    Since 
    \[H^i(R\gHom_{\cA}(\pi_A M^\bullet,\pi R))\neq 0 \Leftrightarrow H^i(R\gHom_{\cA^o}(\pi_{A^o}M^\bullet,\pi R))\neq 0,\] 
    it follows that $\idim_{\cA}\cA=\idim_{\cA^o}\cA$.
    Finally, Lemma \ref{idim cX and pdim cX} (5) gives
    \[\gldim(\qgr A)=\idim_{\cA}\cA=\idim_{\cA^o}\cA=\gldim(\qgr A^o).
    \qedhere \]
\end{proof}

\subsection{Balanced CM isolated singularities}
The following are the definitions of balanced CM isolated singularities and commonly graded AS-Gorenstein isolated singularities.

\begin{definition}\label{balanced CM isolated singularity}
   Let $A$ be a noetherian commonly graded algebra.
    \begin{itemize}
        \item [(1)] If $A$ is a balanced CM algebra of dimension $d$ such that
        $\gldim(\qgr A)=\gldim(\qgr A^o)=d-1$,
        then $A$ is called {\it a balanced CM isolated singularity}.

        \item [(2)] If $A$ is a noetherian commonly graded AS-Gorenstein algebra and a noncommutative isolated singularity, then $A$ is called a {\it commonly graded AS-Gorenstein isolated singularity}.
    \end{itemize}
\end{definition}

We recall the Morita-type theory for noncommutative quasi-projective spaces developed in \cite[Theorems 4.5 and 4.10]{LSW}. 
Let $(\qgr A,s)$ and $(\qgr B,s)$ be noncommutative quasi-projective spaces. 
An equivalence $F:\qgr A\to \qgr B$ commuting with $s$
is called {\it an equivalence of noncommutative quasi-projective spaces}.

\begin{theorem}\label{Morita for nc quasi-proj space}
    Let $A$ and $B$ be noetherian noncommutative projective coordinate rings. Equivalences between $(\qgr A,s)$ and $(\qgr B,s)$ are determined uniquely up to isomorphism by
    modulo-torsion-invertible $(A,B)$-bimodules $M'$ satisfying $\depth_BM'\geqslant 2$ (or $(B,A)$-bimodules $M$ satisfying $\depth_AM\geqslant 2$).

    More precisely, suppose that
$$F:(\qgr A,s)\rightleftarrows (\qgr B,s): G$$
is an equivalence of categories of noncommutative quasi-projective spaces. Let $M=\omega G\cB$ and $M'=\omega F\cA$. Then
\begin{itemize}
    \item [(1)] ${}_BM_A$ and ${}_AM'_B$ are modulo-torsion-invertible bimodules.
    \item [(2)] $\depth_AM\geqslant 2$ and $\depth_BM'\geqslant 2$.
    \item [(3)] $B\cong \gEnd_A(M)\cong \gEnd_{\cA}(\cM)$ and $A\cong \gEnd_B(M')\cong \gEnd_{\cB}(\cM')$.
    \item [(4)] $M\cong\gHom_B(M',B)$ and $M'\cong \gHom_A(M,A)$ as graded bimodules. 
    \item [(5)] $F\cong \pi_B\gHom_{\cA}(\cM,-)\cong -\otimes_{\cA}\cM'$ and $G\cong\pi_A\gHom_{\cB}(\cM',-)\cong -\otimes_{\cB}\cM$.
    \item [(6)] There is a graded Morita context $(A,B,M,M',\tau,\mu)$ being isomorphic to the graded Morita context defined by $M_A$ or $M'_B$ such that it induces the equivalences $(\qgr A,s)\cong (\qgr B,s)$ and $(\qgr A^o,s)\cong (\qgr B^o,s)$.
\end{itemize}
\end{theorem}

The following theorem characterizes balanced CM isolated singularities in terms of the balanced dualizing bimodule $\Omega$.

\begin{theorem}\label{canonical bimodule and isolated singularity}
    Let $A$ be a balanced CM algebra of dimension $d\geqslant 2$, and let $\Omega=D(R^d\Gamma_A(A))$. Suppose that at least one of $\gldim(\qgr A)$ and $\gldim(\qgr A^o)$ is finite. Then the following conditions are equivalent.
    \begin{itemize}
        \item [(1)] $\gldim(\qgr A)=d-1$;
        \item [(2)] $\gldim(\qgr A^o)=d-1$;
        \item [(3)] $\Omega$ is a modulo-torsion-invertible $(A,A)$-bimodule.
    \end{itemize}
    Under these conditions, for every nonzero $\cX\in\qgr A$ and $\cX'\in\qgr A^o$,
    \[
    \pdim_{\cA}\cX=\idim_{\cA}\cX
    =\pdim_{\cA^o}\cX'=\idim_{\cA^o}\cX'=d-1.
    \]
\end{theorem}
\begin{proof}
    Take $n=d-1$. By Theorem \ref{A and A^o are noncommutative iso. sing.}, both global dimensions of $\qgr A$ and $\qgr A^o$ are finite and equal. 
    Hence, (1) and (2) are equivalent. 
    It follows from Theorem \ref{Serre duality in D(qgr A)} that the functor $F=-\otimes_{\cA}^{L}\pi\Omega$ is an auto-equivalence of $\D^b(\qgr A)$. 
    Denote a quasi-inverse of $F$ by $F^{-1}$.

    (1) $\Rightarrow$ (3). For $\cX\in\qgr A$, we have $F\cX\in\D^{\leqslant 0}(\qgr A)$. 
    For any $\cY\in\qgr A$ and $j>0$, Serre duality gives
    \[
    \Hom_{\D^b(\qgr A)}(\cY,F\cX[-j])=\Ext^{-j}_{\cA}(\cY,F\cX)
    \cong D\Ext_{\cA}^{n+j}(\cX,\cY)=0.
    \]
    It follows that $H^j(F\cX)=0$ whenever $j\neq 0$.
    Thus $F\cX$ belongs to the standard heart $\qgr A$ of $\D^b(\qgr A)$.
    Since $F$ is triangulated and the standard t-structure is bounded, it is t-exact. 
    For any $\cZ\in \qgr A\subseteq \D^b(\qgr A)$, there is some $\cY\in \D^b(\qgr A)$ such that $F\cY\cong \cZ$, as $F$ is an equivalence.
    Since $F$ is t-exact, it follows that $H^i\cZ\cong H^i(F\cY)\cong F(H^i\cY)=0$ for $i\neq 0$.
    Hence, $H^i\cY=0$ for $i\neq 0$, which implies that $\cY\in \qgr A\subseteq \D^b(\qgr A)$.
    Therefore, $F$ restricts to an auto-equivalence of $\qgr A$ commuting with $s$.

    Both $A$ and $A^o$ are noncommutative projective coordinate rings. By Proposition \ref{Hom(-,Omega) induce duality on MCM}, $\Omega$ is MCM on both sides, and hence $\omega\pi\Omega\cong\Omega$ by Lemma \ref{omega pi and depth 2}. The bimodule corresponding to $F$ in Theorem \ref{Morita for nc quasi-proj space} is therefore
    $\omega F\cA\cong\omega\pi\Omega\cong\Omega$.
Note that the action obtained by applying $F$ to graded left multiplication on $A$ is the original left $A$-action on $\Omega$.
    Hence, the identification $\omega F\cA\cong\Omega$ preserves the left $A$-action.
    It follows that $\Omega$ is a modulo-torsion-invertible $(A,A)$-bimodule.

    (3) $\Rightarrow$ (1). The induced auto-equivalence $-\otimes_{\cA}\pi\Omega$ sends $\cA$ to $\pi\Omega$ and preserves injective dimension. It follows from Lemma \ref{homo. dim. in qgr A with b.d.c.} (1) and Lemma \ref{idim cX and pdim cX} (5) that
    \[
    \gldim(\qgr A)=\idim_{\cA}\cA
    =\idim_{\cA}(\pi\Omega)=d-1.
    \]
    Thus (1) holds. 
    

    Finally, let $0\neq\cX\in\qgr A$. The global dimension gives upper bounds $n$ for its projective and injective dimensions. Serre duality gives
\[    \Ext_{\cA}^{n}(\cX,F\cX)
    \cong D\Hom_{\D^b(\qgr A)}(F\cX,F\cX)=D\Hom_{\cA}(\cX\otimes_{\cA}\pi \Omega,\cX\otimes_{\cA}\pi \Omega),\]
    which implies that $\pdim_{\cA}\cX=n$. Similarly,
    \[
    \Ext_{\cA}^{n}(F^{-1}\cX,\cX)
    \cong D\Hom_{\D^b(\qgr A)}(\cX,\cX) \cong  D\Hom_{\cA}(\cX,\cX),\]
    which implies that $\idim_{\cA}\cX=n$.
    The left module argument follows from a similar proof.
\end{proof}

\begin{corollary}\label{ASG+iso. is ASG iso.}
    A commonly graded AS-Gorenstein isolated singularity of dimension $d\geqslant2$ is a balanced CM isolated singularity.
\end{corollary}
\begin{proof}
    If $A$ is a commonly graded AS-Gorenstein isolated singularity of dimension $d\geqslant2$, then its canonical bimodule is invertible by Proposition \ref{AS-G has balanced dualizing complex and chi} (1).
    So Theorem \ref{canonical bimodule and isolated singularity} applies. 

    It also follows from \cite[Lemma 5.9]{LSW}.
\end{proof}

Next we give an example of a triangular matrix algebra $A$ that is a balanced CM algebra of dimension $d$ and is a noncommutative isolated singularity, but $\gldim(\qgr A)=d\neq d-1$.
Hence, it is not a balanced CM isolated singularity. This shows the necessity of the equation $\gldim(\qgr A)=\gldim(\qgr A^o)=d-1$ in Definition \ref{balanced CM isolated singularity}.

\begin{example}\label{triangular balanced CM example}
Let $\Lambda$ be the path algebra of the quiver
$1\longrightarrow 2$.
Let $S=k[x_0,\ldots,x_{d-1}]$ and $A=\Lambda\otimes_k S$ where $d\geqslant 2$.
Set $\deg\Lambda=0$ and $\deg x_i=1$.

We identify $\Lambda$ with the algebra of upper triangular $2\times2$ matrices over $k$, with the arrow represented by the matrix unit $e_{12}$ and the vertices represented by the matrix units $e_{11},e_{22}$. Thus
\[
A\cong\begin{pmatrix}S&S\\0&S\end{pmatrix}.
\]
We will show that $A$ is balanced CM of dimension $d$ and
\[
\gldim(\qgr A)=\gldim(\qgr A^o)=d.
\]

The algebra $A$ is finitely generated and free as a module over the central subalgebra $S$, so it is noetherian and it admits a balanced dualizing complex $R_A$.
The balanced dualizing complex $R_S$ of $S$ is of the form $S(-d)[d]$. 
By \cite[Proposition 2.8]{LSW} and Theorem \ref{local duality}, one obtains isomorphisms in $\D(\Gr A^e)$:
\begin{align*}
R_A
&\cong D(R\Gamma_A(A))
 \cong D(R\Gamma_S(A))\\
&\cong R\gHom_S(A,R_S)
 \cong\bigl(D\Lambda\otimes_k S(-d)\bigr)[d].
\end{align*}
Thus $A$ is balanced CM of dimension $d$, with the canonical module
\[
\Omega_A=D\Lambda\otimes_k S(-d).
\]

We now compute the global dimension of $\qgr A$. The polynomial ring $S$ has global dimension $d$, and its noncommutative quasi-projective space has global dimension $\gldim(\qgr S)=d-1$.
Let $\Rep A$ be the category of homomorphisms $U\xrightarrow{f}V$ of graded $S$-modules. 
For $U\xrightarrow[]{f} V$ and $U'\xrightarrow[]{f'}V'$ in $\Rep A$, a morphism between them is a pair $(g_U,g_V)$ such that $g_Vf=f'g_U$ where $g_U:U\to U'$ and $g_V:V\to V'$ are homomorphisms of graded $S$-modules.
Define
\[F:\Gr A\to \Rep A, \qquad M\mapsto (Me_{11} \xrightarrow[]{(e_{12})_r} Me_{22}),\]
where $(e_{12})_r$ is the right multiplication by $e_{12}$. Then $F$ is an equivalence. 
From this point,
a graded $A$-module is a morphism $U\xrightarrow{f}V$ of graded right $S$-modules. 

Define
\[
L_1: \Gr S\to \Gr A,\qquad
U\mapsto (U\xrightarrow{1}U)
\]
and
\[
L_2:\Gr S\to \Gr A,
\qquad U\mapsto(0\to U).
\]
The functors $L_1,L_2$ are exact and are left adjoints to the exact functors $R_1:\Gr A\to \Gr S, M\mapsto Me_{11}$ and $R_2:\Gr A\to \Gr S, M\to Me_{22}$, respectively.
All four functors preserve torsion modules, since torsion is determined component-wise. 
Therefore, they induce exact functors between the quotient categories. Using the same notation for these induced functors, we have
\[
\Ext_{\cA}^{j}(L_i\cU,\cN)
\cong\Ext_{\cS}^{j}(\cU,R_i\cN), \quad \text{ for } i=1,2.
\]
Since $\gldim(\qgr S)=d-1$, it follows that for any $\cU\in \qgr A$, the projective dimension of $L_i(\cU)$ is at most $d-1$.

Every finitely generated graded $A$-module $M=(U\xrightarrow{f}V)$ fits into an exact sequence
\[
0\to L_2(U)\to L_1(U)\oplus L_2(V)\to M\to 0.
\]
At the second component, the maps are $u\mapsto(u,-f(u))$ and $(u,v)\mapsto f(u)+v$, while the last map at the first component is $1_U$. 
After applying $\pi_A$, the exact sequence
\[
0\to L_2(\cU)\to L_1(\cU)\oplus L_2(\cV)\to \cM\to 0
\] 
shows that every object of $\qgr A$ has projective dimension at most $d$.

For the reverse inequality, take
$M=(S\to 0),N=(0\to S(-d))$.
Observe that $\Ext_{\cA}^{j}(L_1\cS,\cN)
\cong\Ext_{\cS}^{j}(\cS,0)=0$.
Then the exact sequence $0\to L_2(S)\to L_1(S)\to M\to0$ yields
\begin{align*}
\Ext_{\cA}^{d}(\cM,\cN)
&\cong\Ext_{\cA}^{d-1}(L_2\cS,\cN)\\
&\cong\Ext_{\cS}^{d-1}(\cS,\cS(-d))\\
&\cong R^d\Gamma_S(S(-d))_0
 \cong k.
\end{align*}
The third isomorphism follows from Lemma \ref{cohomology in tails}. Hence $\gldim(\qgr A)=d$. 

Thus $A$ is a balanced CM algebra and a noncommutative isolated singularity, but it is not a balanced CM isolated singularity. By Theorem \ref{canonical bimodule and isolated singularity}, its canonical bimodule $\Omega_A$ is not modulo-torsion-invertible.
\end{example}

\section{Equivalence of noncommutative quasi-projective spaces}\label{equivalent noncommutative quasi-projective spaces}
In this section, we investigate the equivalences $(\qgr A,s)\cong(\qgr B,s)$ between the noncommutative quasi-projective spaces associated with commonly graded noetherian algebras $A$ and $B$. 
In particular, we identify the properties in Diagram 1 that are transferred from $B$ to $A$ under such an equivalence. 

The following proposition follows immediately from the equivalence between noncommutative quasi-projective spaces and Lemmas \ref{idim cX and pdim cX} and \ref{local cohomological dimension and qgr}.

\begin{proposition}\label{dimension between qgrs}
    Let $A$ and $B$ be right noetherian commonly graded algebras, and let $M'$ be a modulo-torsion-invertible $(A,B)$-bimodule with inverse ${}_BM_A$.
    Suppose that $A$ is a noncommutative isolated singularity.
    Then
$$\gldim(\qgr A)=\gldim(\qgr B)=\idim_{\cA}\cA=\idim_{\cA}\cM=\idim_{\cB}\cB=\idim_{\cB}\cM'.$$
Moreover, $\lcd A=\pdim_{\cB}\cM'+1$ and $\lcd B=\pdim_{\cA}\cM+1.$
\end{proposition}

Suppose that $A$ is a right noetherian commonly graded algebra. Let
$$C_A=\{X\in\gr A\mid \depth_AX\geqslant 2\}, \textrm{ and }$$
$$\cC_{\cA}=\{\cX\in\qgr A\mid \cX\cong \pi \tilde{X} \text{ for some }\tilde{X}\in C_A\}.$$
By Lemma \ref{omega pi and depth 2}, $\cC_{\cA}=\{\cX\in \qgr A\mid \omega\cX \text{ is a finitely generated $A$-module}\}$. Hence, $\pi C_A\subseteq \cC_{\cA}$ and $\omega \cC_{\cA}\subseteq C_A$. The following result is immediate.

\begin{lemma}
    The restriction functors $\pi:C_A\to \cC_{\cA}$ and $\omega:\cC_{\cA}\to C_A$ give an equivalence between $C_A$ and $\cC_{\cA}$.
\end{lemma}
\begin{proof}
It follows from Lemma \ref{omega pi and depth 2} and $\pi\omega\cong \id_{\QGr A}$.
\end{proof}


Next, we show that any equivalence between the noncommutative quasi-projective spaces associated with $A$ and $B$ induces equivalences $C_A\cong C_B$ and $\cC_{\cA}\cong \cC_{\cB}$.

\begin{proposition}\label{C_A and C_B}
     Suppose that $A$ and $B$ are noetherian noncommutative projective coordinate rings. Let $M'$ be a modulo-torsion-invertible $(A,B)$-bimodule satisfying $\depth_B M'\geqslant 2$, and set $M=\gHom_B(M',B)$. Then
     \begin{itemize}
         \item [(1)] 
         $\gHom_A(M,-):C_A\rightleftarrows C_B:\gHom_B(M',-)$ defines an equivalence.
         \item [(2)] 
         $\pi_B\gHom_{\cA}(\cM,-):\cC_{\cA}\rightleftarrows \cC_{\cB}:\pi_A\gHom_{\cB}(\cM',-)$ defines an equivalence.
     \end{itemize}
\end{proposition}
\begin{proof}
    (1) This is proved in \cite[Theorem 4.12]{LSW}.

    (2) For any $\cX\in \cC_{\cA}$, we may assume $X\in C_A$. Then 
    $$\pi_B\gHom_{\cA}(\cM,\cX)\cong\pi_B\gHom_A(M,\omega\pi X)\cong \pi_B\gHom_A(M,X).$$ 
    By (1), $\gHom_A(M,X) \in C_B$. Hence $\pi_B\gHom_{\cA}(\cM,\cX)\cong\pi_B\gHom_A(M,X)\in \cC_{\cB}$. 
    Similarly, for any $\cY\in \cC_{\cB}$, $\pi_A\gHom_{\cB}(\cM',\cY)\in \cC_{\cA}$.
    By Theorem \ref{Morita for nc quasi-proj space}, there is an equivalence of noncommutative quasi-projective spaces
    $$\pi_B\gHom_{\cA}(\cM,-):(\qgr A,s)\rightleftarrows (\qgr B,s):\pi_A\gHom_{\cB}(\cM',-).$$
    So this equivalence restricts to an equivalence
    \[
    \pi_B\gHom_{\cA}(\cM,-):\cC_{\cA}\rightleftarrows \cC_{\cB}:\pi_A\gHom_{\cB}(\cM',-).
    \qedhere
    \]
\end{proof}

The following lemma and proposition characterize the depth of a $B$-module in $C_B$ and characterize MCM $B$-modules in terms of modulo-torsion-invertible bimodules, respectively.

\begin{lemma}\label{depth after equi. of qgr}
     Let $A$ and $B$ be noetherian noncommutative projective coordinate rings, and let $M'$ be a modulo-torsion-invertible $(A,B)$-bimodule with $\depth_BM'\geqslant 2$. Let $M=\gHom_B(M',B)$. Then, for any $N\in C_B$, 
     $$\depth_BN=\min\{i>0\mid \gExt_{\cA}^i(\cM,\cN\otimes_{\cB}\cM)\neq 0\}+1.$$
\end{lemma}

\begin{proof}
    By Lemma \ref{depth and local cohomology}, $\depth_BN=\min\{i\mid R^i\Gamma_B(N)\neq 0\}$. Moreover, Lemma \ref{cohomology in tails} implies that, for $i\geqslant 1$,
    $R^{i+1}\Gamma_B(N)\cong \gExt_{\cB}^i(\cB,\cN).$
    Since $\depth_BN\geqslant 2$ and $-\otimes_{\cB}\cM$ induces an equivalence between $(\qgr A,s)$ and $(\qgr B,s)$, 
    \begin{align*}
    \depth_BN&=\min\{i>0\mid \gExt_{\cB}^i(\cB,\cN)\neq 0\}+1 \\
             &=\min\{i>0\mid \gExt_{\cA}^i(\cM,\cN\otimes_{\cB}\cM)\neq 0\}+1.
    \qedhere
    \end{align*}
\end{proof}

\begin{proposition}\label{MCM B and Ext in qgr A}
    Let $A$ and $B$ be noetherian noncommutative projective coordinate rings, and let $M'$ be a modulo-torsion-invertible $(A,B)$-bimodule with $\depth_B M'\geqslant 2$. Let $M=\gHom_B(M',B)$. If $B$ has a balanced dualizing complex and $\lcd B=d$, then, up to isomorphism,
    $$\MCM(B)=\{\gHom_A(M,X)\mid X\in C_A \text{ and } \gExt_{\cA}^i(\cM,\cX)=0 \textrm{ for all } 0<i<d-1\}.$$
\end{proposition}
\begin{proof}
    Suppose that $X\in C_A$ satisfies $\gExt_{\cA}^i(\cM,\cX)=0$ for all $0<i<d-1$. Let $N=\gHom_A(M,X)$. Then
    $$\cN=\pi_B\gHom_A(M,X)\cong \pi_B\gHom_A(M,\omega\pi_AX)\cong \pi_B\gHom_{\cA}(\cM,\cX).$$
    By Theorem \ref{Morita for nc quasi-proj space}, $-\otimes_{\cB}\cM$ is a quasi-inverse of $\pi_B\gHom_{\cA}(\cM,-)$. So $\cN\otimes_{\cB}\cM\cong \cX$.
    It follows from Lemma \ref{depth after equi. of qgr} that $\depth_BN\geqslant d$.

    Since $B$ has a balanced dualizing complex, $\depth_BN$ is finite by \cite[Lemma 8.1]{LW2}. By Lemma \ref{depth and local cohomology}, $\depth_BN=\min\{i\mid R^i\Gamma_B(N)\neq 0\}\leqslant d$. Therefore, $\depth_BN=d$, and $N$ is an MCM $B$-module.

    Conversely, suppose that $N\in\MCM(B)$, and let $X=\gHom_B(M',N)$.
    By Proposition \ref{C_A and C_B}, $X\in C_A$ and $N\cong \gHom_A(M,X)$.
    Moreover,
    $$\cX=\pi_A\gHom_B(M',N)\cong \pi_A\gHom_B(M',\omega \pi_BN)\cong \pi_A\gHom_{\cB}(\cM',\cN).$$ 
    By Theorem \ref{Morita for nc quasi-proj space}, $\pi_A\gHom_{\cB}(\cM',-)$ is naturally isomorphic to $-\otimes_{\cB}\cM$. Hence, $\cX\cong \cN\otimes_{\cB}\cM$. Then, for any $0<i<d-1$, by Lemma \ref{cohomology in tails}
    \begin{align*}
        0=R^{i+1}\Gamma_B(N)&\cong \gExt_{\cB}^i(\cB,\cN)\\
        &\cong \gExt_{\cA}^i(\cM,\cN\otimes_{\cB}\cM)\\
        &\cong \gExt_{\cA}^i(\cM,\cX).
    \qedhere 
    \end{align*}
\end{proof}


Next, we discuss the $\chi$-conditions for two noetherian noncommutative projective coordinate rings under equivalences of their associated noncommutative quasi-projective spaces.

\begin{lemma}\label{chi between equi. of qgr}
    Let $A$ and $B$ be noetherian noncommutative projective coordinate rings, and let $M'$ be a modulo-torsion-invertible $(A,B)$-bimodule with $\depth_B M'\geqslant 2$. Suppose that $B$ satisfies the $\chi_n$-condition for some $n\geqslant 2$. Then $A$ satisfies the $\chi_n$-condition if and only if $\gExt^i_B(M',Y)$ is finite-dimensional for all $Y\in \gr B$ and $1\leqslant i\leqslant n-1$.
\end{lemma}
\begin{proof}
    For convenience, let (X) denote the statement that ``$A$ satisfies the $\chi_n$-condition." 
    
    Since $A$ is a noncommutative projective coordinate ring, it satisfies the $\chi_1$-condition. 
    Thus, by Lemma \ref{cohomology in tails} and Proposition \ref{facts about chi condition}, (X) holds if and only if, for any $1\leqslant i\leqslant n-1$ and $X\in\gr A$, $\gExt_{\cA}^i(\cA,\cX)$ is bounded above.
    
    The equivalence $-\otimes_{\cA}\cM':\qgr A\to \qgr B$ implies that
    (X) holds if and only if, for any $1\leqslant i\leqslant n-1$ and $X\in \gr A$, the module $\gExt_{\cB}^i(\cM',\cX\otimes_{\cA}\cM')$ is bounded above.
    Thus, (X) holds if and only if, for any $1\leqslant i\leqslant n-1$ and $Y\in \gr B$, the module $\gExt_{\cB}^i(\cM',\cY)$ is bounded above.
    
    Since $B$ satisfies the $\chi_n$-condition, Proposition \ref{facts about chi condition} implies that, for $1\leqslant i\leqslant n-1$, $\gExt_{\cB}^i(\cM',\cY)$ is bounded above if and only if $\gExt_B^i(M',Y)$ is finite-dimensional. 
    
    In conclusion, (X) holds if and only if $\gExt_B^i(M',Y)$ is finite-dimensional for any $1\leqslant i\leqslant n-1$ and $Y\in \gr B$.
\end{proof}

For objects in noncommutative quasi-projective spaces associated with balanced CM algebras, condition (2) of ampleness in Definition \ref{conditions of ample} is characterized as follows.

\begin{lemma}\label{equi. cond. for ample (2) for MCM}
    Suppose that $B$ is a balanced CM algebra of dimension $d\geqslant 2$, and let $N_B$ be an MCM $B$-module. Then the following conditions are equivalent.
    \begin{itemize}
        \item [(1)] For any epimorphism $\cY_1\to \cY_2$ in $\qgr B$, there exists an integer $n$ such that the induced morphism $\gHom_{\cB}(\cN,\cY_1)_{\geqslant n}\to \gHom_{\cB}(\cN,\cY_2)_{\geqslant n}$ is surjective.
        \item [(2)] For any $Y\in \MCM (B)$, $\gExt_B^1(N,Y)$ is finite-dimensional.
        \item [(3)] For any $Y\in \gr B$ and $i\geqslant 1$, $\gExt_B^i(N,Y)$ is finite-dimensional.
    \end{itemize}
    Moreover, if $B$ is a balanced CM isolated singularity, then every MCM $B$-module satisfies these conditions. In this setting, $(\cN,s)$ is ample whenever $N$ is an MCM generator.
\end{lemma}
\begin{proof}
    Let $\Omega=D(R^d\Gamma_B(B))$.

    (1) $\Rightarrow$ (2). Let $Y\in\MCM(B)$. Applying $\gHom_B(-,\Omega)$ to a finite projective presentation of $\gHom_B(Y,\Omega)$ and the duality given in Proposition \ref{Hom(-,Omega) induce duality on MCM} (2) yield an exact sequence
    \[
    0\to Y\to\Omega^0\to L\to0,
    \quad \Omega^0\in\add_B\Omega,\quad L\in\MCM(B).
    \]
    By Proposition \ref{Hom(-,Omega) induce duality on MCM} (1), $\gExt_B^1(N,\Omega^0)=0$. Therefore, applying $\gHom_B(N,-)$ gives an exact sequence
    \[
    \gHom_B(N,\Omega^0)\to\gHom_B(N,L)
    \to\gExt_B^1(N,Y)\to0.
    \]
    Since $\depth_B\Omega^0,\depth_BL\geqslant2$,
 one obtains $\gHom_B(N,\Omega^0)\cong\gHom_{\cB}(\cN,\pi\Omega^0)$ and $\gHom_B(N,L)\cong\gHom_{\cB}(\cN,\pi L)$, by Lemma \ref{omega pi and depth 2} and adjunction. 
 Condition (1), applied to the epimorphism $\pi\Omega^0\to\pi L$, implies that $\gExt_B^1(N,Y)$ is bounded above. 
 A resolution of $N$ by finitely generated graded projective modules shows that $\gExt_B^1(N,Y)$ is locally finite and bounded below. Hence, it is finite-dimensional.

    (2) $\Rightarrow$ (3). First, let $Y\in\MCM(B)$. Applying $\gHom_B(-,\Omega)$ to a finite projective resolution of $\gHom_B(Y,\Omega)$ and the duality given in Proposition \ref{Hom(-,Omega) induce duality on MCM} (2) yield a coresolution
    \[
    0\to Y\to\Omega^0\to\Omega^1\to\cdots,
    \quad \Omega^r\in\add_B\Omega.
    \]
   Since $\gExt_B^i(N,\Omega)=0$ for any $i\geqslant 1$, 
        $$\gExt_B^j(N,Y)\cong\gExt_B^{j-1}(N,\image \partial_1)\cong \cdots \cong \gExt_B^1(N,\image\partial_{j-1}) \text{ when } j\geqslant 2.$$
        By \cite[Lemma 2.10]{LSW}, $\image\partial_{i-1}$ is MCM. 
        Therefore, $\gExt_B^1(N,\image\partial_{i-1})$ is finite-dimensional, which implies that 
        $\gExt_B^i(N,Y)$ is finite-dimensional for all $i\geqslant 1$.

    For arbitrary $Y\in\gr B$, \cite[Proposition 8.7]{LW2} gives an exact sequence
    \[
    0\to L\to Z\to Y\to0,
    \]
    where $Z$ is MCM and $L$ has a finite resolution with terms in $\add_B\Omega$. By dimension shifting, $\gExt_B^i(N,L)=0$ for all $i>0$. 
    Hence $\gExt_B^i(N,Y)\cong\gExt_B^i(N,Z)$ for all $i>0$. This proves (3).

    (3) $\Rightarrow$ (1). Given an epimorphism $\cY_1\to\cY_2$ in $\qgr B$, let $\cK$ be its kernel and choose $K\in\gr B$ such that $\pi K=\cK$. The exact sequence
    \[
    \gHom_{\cB}(\cN,\cY_1)\to 
    \gHom_{\cB}(\cN,\cY_2)\to 
    \gExt_{\cB}^1(\cN,\cK),
    \]
    together with Proposition \ref{facts about chi condition}, shows that (3) implies (1), as the last term is bounded above.

    Now suppose that $B$ is a balanced CM isolated singularity. By Theorem \ref{canonical bimodule and isolated singularity}, the functor
    $F=-\otimes_{\cB}\pi\Omega$ is an auto-equivalence of $(\qgr B,s)$. 
    Let $F^{-1}$ be a quasi-inverse of $F$.
    It follows from Lemma \ref{cohomology in tails} and Proposition \ref{facts about chi condition} that
    for every $\cY\in\qgr B$ and $i>0$, 
    \[
    \gExt_{\cB}^i(\pi\Omega,\cY)
    \cong\gExt_{\cB}^i(\cB,F^{-1}\cY)
    \]
    is bounded above.
     
    Let $N$ be an MCM $B$-module.
    Then there exists a resolution
    \[
        \cdots\to \Omega^{-1}\to \Omega^0\to N\to 0
    \]
    such that $\Omega^r\in\add_B\Omega$.
    Decompose this exact sequence into short exact sequences
    \[
    0\to N_r\to\Omega^r\to N_{r+1}\to0,
    \]
    where each $N_r$ is MCM and $N_0=N$. 
Applying the exact quotient functor $\pi$ and $\gHom_{\cB}(-,\cY)$ to these short exact sequences
gives exact sequences
\[
\gExt^i_{\cB}(\pi\Omega^r,\cY)\to
\gExt^i_{\cB}(\pi N_r,\cY)\to
\gExt^{i+1}_{\cB}(\pi N_{r+1},\cY)
\]
for all $i>0$ and $r\geqslant 0$.

We prove by descending induction on $i$ that
$\gExt^i_{\cB}(\pi N_r,\cY)$ is bounded above for every $r\geqslant 0$.
Since $\operatorname{gldim}(\qgr B)=d-1$, it follows that
$\gExt^i_{\cB}(\pi N_r,\cY)=0$ for all $i\geq d$ and $r\geq 0$.
Now let $1\leqslant i\leqslant d-1$, and assume that
$\gExt^{i+1}_{\cB}(\pi N_r,\cY)$ is bounded above for every $r\geqslant 0$. 
For a fixed $r\geqslant 0$,
by the preceding paragraph and the induction hypothesis, there exist integers $a$ and $b$ such that
\[
    \gExt^i_{\cB}(\pi\Omega^r,\cY)_{>a}=0 \text{ and } \gExt^{i+1}_{\cB}(\pi N_{r+1},\cY)_{>b}=0.
\]
It follows from the preceding exact sequence
that $\gExt^i_{\cB}(\pi N_r,\cY)$ is bounded above.
This completes the induction.

Taking $r=0$ and using $N_0=N$, we conclude that $\gExt^i_{\cB}(\cN,\cY)$ is bounded above for every $i>0$.
By Proposition \ref{facts about chi condition}, $\gExt_B^i(N,Y)$ is bounded above for every $Y\in\gr B$. Hence, (3) holds.

Finally, if $N$ is a generator, then $B\in\add_BN$. Since $(\cB,s)$ is ample, $(\cN,s)$ satisfies ampleness condition (1), while the condition proved above gives ampleness condition (2).
\end{proof}

The following proposition transfers the $\chi$-condition from a balanced CM algebra $B$ to $A$ when the bimodule $M'$ inducing the equivalence is an MCM $B$-module.

\begin{proposition}\label{B balanced CM, M' MCM, then A chi}
    Suppose that $A$ and $B$ are noetherian noncommutative projective coordinate rings, and that $M'$ is a modulo-torsion-invertible $(A,B)$-bimodule. If $B$ is a balanced CM algebra and $M'_B$ is MCM, then $A$ satisfies the $\chi$-condition.
\end{proposition}
\begin{proof}
Since $-\otimes_{\cA}\cM':(\qgr A,s)\to (\qgr B,s)$
is an equivalence, $(\cM',s)$ is ample in $\qgr B$.
Thus, $M'$ satisfies the equivalent conditions in Lemma \ref{equi. cond. for ample (2) for MCM}. 
It follows from Proposition \ref{AS-G has balanced dualizing complex and chi} and Lemma \ref{chi between equi. of qgr} that $A$ satisfies the $\chi$-condition.
\end{proof}

If $B$ is a balanced CM isolated singularity, the following stronger transfer result holds.

\begin{proposition}\label{when B is GAS-Gorenstein iso. sing.}
    Suppose that $A$ and $B$ are noetherian noncommutative projective coordinate rings, and let $M'$ be a modulo-torsion-invertible $(A,B)$-bimodule. Assume further that $B$ is a balanced CM isolated singularity of dimension $d\geqslant2$. Then
    \begin{itemize}
        \item [(1)] $\gldim(\qgr A)=\gldim(\qgr A^o)=\idim_{\cA}\cA=\idim_{\cA^o}\cA=d-1$;
        \item [(2)] for every nonzero $\cX\in\qgr A$ and $\cX'\in\qgr A^o$,
        \[
        \pdim_{\cA}\cX=\idim_{\cA}\cX
        =\pdim_{\cA^o}\cX'=\idim_{\cA^o}\cX'=d-1;
        \]
        \item [(3)] $\lcd A=\lcd A^o=d$.
    \end{itemize}
    Moreover, if $M'_B$ is MCM, then
    \begin{itemize}
        \item [(4)] $A^o$ is a noncommutative projective coordinate ring;
        \item [(5)] $A$ admits a balanced dualizing complex.
    \end{itemize}
\end{proposition}
\begin{proof}
    Theorem \ref{Morita for nc quasi-proj space} provides equivalences $(\qgr A,s)\cong (\qgr B,s)$. 
    It follows from this equivalence and Theorem \ref{canonical bimodule and isolated singularity} that (1) and (2) hold. 
    Lemma \ref{local cohomological dimension and qgr} then gives (3).

    (4) Assume that $M'_B$ is MCM. Let
    $M=\gHom_B(M',B),\Omega=D(R^d\Gamma_B(B))$ and
    $N=\gHom_B(M',\Omega)$.
    By Theorem \ref{Morita for nc quasi-proj space}, $A\cong \gEnd_B(M')$ and $M$ is the inverse of the modulo-torsion-invertible bimodule $M'$. Hence one obtains the equivalence
    \[
    \cM\otimes_{\cA}-:(\qgr A^o,s)\to (\qgr B^o,s).
    \]
    Proposition \ref{Hom(-,Omega) induce duality on MCM} (2) implies that $N$ is MCM over $B^o$ and that
    $\gEnd_{B^o}(N)\cong\gEnd_B(M')^o\cong A^o$.

    By Theorem \ref{canonical bimodule and isolated singularity}, the functor $F=\pi\Omega\otimes_{\cB}-$ is an auto-equivalence of $\qgr B^o$. We claim that
    $F(\pi_{B^o}M)\cong\pi_{B^o}N$.
    Indeed, choose a presentation $P_1\to P_0\to M'\to0$ where $P_i$ are finitely generated graded projective $B$-modules. Applying $\gHom_B(-,B)$ gives
    \[
    0\to  M\to \gHom_B(P_0,B)
    \to \gHom_B(P_1,B).
    \]
    Applying the exact functor $F\circ \pi_{B^o}$ gives
    \[
    0\to  F(\pi_{B^o}M)\to F\pi_{B^o}\gHom_B(P_0,B)
    \to F\pi_{B^o}\gHom_B(P_1,B).
    \]
    On the other hand, applying $\gHom_B(-,\Omega)$ to $P_1\to P_0\to M'\to0$ gives
    \[
    0\to  N\to \gHom_B(P_0,\Omega)
    \to \gHom_B(P_1,\Omega).
    \]
    The natural isomorphisms
    \[
    F\pi_{B^o}\gHom_B(P_i,B)=\pi_{B^o}(\Omega\otimes_B\gHom_B(P_i,B))\cong \pi_{B^o}(\gHom_B(P_i,\Omega))
    \]
 then yield $F(\pi_{B^o}M)\cong\pi_{B^o}N$.

    Since $F$ is an equivalence and $\depth_{B^o} N\geqslant 2$, it follows that
    \[
    \gEnd_{\cB^o}(\pi_{B^o} M)
    \cong\gEnd_{\cB^o}(\pi_{B^o} N)
    \cong\gEnd_{B^o}(N)
    \cong A^o.
    \]
    For every $\cY\in\qgr B^o$ and $i>0$, we also have
    \[
    \gExt_{\cB^o}^i(\pi_{B^o} M,\cY)
    \cong\gExt_{\cB^o}^i(\pi_{B^o} N,F\cY).
    \]
    It follows from the left module version of Lemma \ref{equi. cond. for ample (2) for MCM} and Proposition \ref{facts about chi condition} that $\gExt_{\cB^o}^i(\pi_{B^o} N,F\cY)$ is bounded above. 
    Then the case $i=1$ gives ampleness condition (2) for $(\pi_{B^o} M,s)$ in $\qgr B^o$.
    Ampleness condition (1) holds for $(\pi_{B^o} M,s)$ in $\qgr B^o$ as well, because $(\cA,s)$ satisfies ampleness condition (1) in $\qgr A^o$ and $\cM\otimes_{\cA}-:(\qgr A^o,s)\to (\qgr B^o,s)$ is an equivalence. 
    Therefore, the remaining conditions of Theorem \ref{noncommutative Serre theorem for commonly graded algebra} follow, which yields that (4) holds.

    (5) Proposition \ref{B balanced CM, M' MCM, then A chi} shows that $A$ satisfies the $\chi$-condition. 
    For the other side, 
    since $\gExt_{\cB^o}^i(\pi_{B^o} M,\cY)$ is bounded above for any $\cY\in\qgr B^o$ and $B^o$ satisfies the $\chi$-condition, by Proposition \ref{facts about chi condition},
    the vector space $\gExt_{B^o}^i(M,Y)$ is finite-dimensional for every $Y\in\gr B^o$ and $i>0$. 
    Since projective left $B$-modules $\gHom_B(P_0,B)$ and $\gHom_B(P_1,B)$ have depth at least two, the preceding exact sequence $0\to  M\to \gHom_B(P_0,B)
    \to \gHom_B(P_1,B)$ together with \cite[Lemma 2.10]{LSW} implies that $\depth_{B^o}M\geqslant2$.
    It follows from the left module version of Lemma \ref{chi between equi. of qgr} that $A^o$ satisfies $\chi$. 
    Combining this with (3), we obtain (5) from Theorem \ref{existence of balanced dualizing complex}.
\end{proof}

\begin{corollary}\label{balanced CM from self extensions}
    Under the assumptions of Proposition \ref{when B is GAS-Gorenstein iso. sing.}, suppose that $M'_B$ is MCM. Then $A$ is a balanced CM isolated singularity of dimension $d$ if and only if
    \[
    \gExt_B^i(M',M')=0\quad(0<i<d-1).
    \]
    In particular, if $d=2$, every noetherian noncommutative projective coordinate ring whose noncommutative quasi-projective space is equivalent to $(\qgr B,s)$ is a balanced CM isolated singularity of dimension two.
\end{corollary}
\begin{proof}
    Since $A$ is a projective coordinate ring, $R^0\Gamma_A(A)=R^1\Gamma_A(A)=0$. For $0<i<d-1$, Lemma \ref{cohomology in tails}, the equivalence $(\qgr A,s)\cong (\qgr B,s)$, and \cite[Lemma 5.5]{LSW} give
    \[
    R^{i+1}\Gamma_A(A)
    \cong\gExt_{\cA}^i(\cA,\cA)
    \cong\gExt_{\cB}^i(\cM',\cM')
    \cong\gExt_B^i(M',M').
    \]
    This proves the first assertion.
    
    When $d=2$, the condition  $\gExt_B^i(M',M')=0$ for $0<i<d-1$ is empty. Hence, the second assertion holds.
\end{proof}

If $B$ is a noetherian commonly graded AS-regular algebra, then the noncommutative projective scheme $(\qgr B,\cB,s)$ is smooth. Following \cite[Remark 6.3]{LSW}, 
we use the equivalence of noncommutative quasi-projective spaces as a substitute for ``birational equivalence." 
Hence, if $(\qgr A,s)\cong (\qgr B,s)$, then $B$ may, in some sense, be regarded as a noncommutative resolution of $A$. 

\begin{definition}\label{def-noncom-quasi-resolution}
    Let $A$ be a noetherian noncommutative projective coordinate ring. If there exists a noetherian commonly graded AS-regular algebra $B$ such that $(\qgr A,s)\cong (\qgr B,s)$, then $B$ is termed a {\it noncommutative quasi-resolution} of $A$.
\end{definition}

\begin{remark}\label{remark of nqr}
    Noncommutative quasi-resolutions were first introduced in \cite{QWZ}. The definition presented here is that of a $(d-2)$-noncommutative quasi-resolution given in \cite[Definition 3.2]{QWZ}.
\end{remark}

Although the definition of noncommutative quasi-resolutions is weaker than that of the noncommutative resolutions introduced in the next section, not every noncommutative isolated singularity admits a noncommutative quasi-resolution.

\begin{example}
Let $A$ be the algebra given in Example \ref{triangular balanced CM example}. 
If $B$ is a noncommutative quasi-resolution of $A$, then $\gldim(\qgr B)=\gldim (\qgr A)=d$.
Since $B$ is AS-regular, it is a balanced CM isolated singularity of dimension $d+1$.
It follows from Proposition \ref{when B is GAS-Gorenstein iso. sing.} that $\lcd A=d+1$, which leads to a contradiction.
\end{example}

By Theorem \ref{Morita for nc quasi-proj space}, if $B$ is a noncommutative quasi-resolution of $A$, then there exists a graded Morita context $(A,B,M,M',\tau,\mu)$ such that $\depth_A M\geqslant 2$ and $\depth_B M'\geqslant 2$, which induces an equivalence $(\qgr A,s)\cong (\qgr B,s)$.
Moreover, $(A,B,M,M',\tau,\mu)$ is isomorphic to the graded Morita context defined by either $M_A$ or $M'_B$.
Thus, we say that the noncommutative quasi-resolution $B$ is {\it given by ($M, M'$)}.
Note that if $A$ has a noncommutative quasi-resolution $B$, then Proposition \ref{when B is GAS-Gorenstein iso. sing.} (1)--(3) applies to $A$.

When $B$ is a noncommutative resolution (see Definition \ref{def-nc-reso-AS-G}) of a commonly graded AS-Gorenstein isolated singularity $A$ of dimension $d$, given by $(M,M')$, 
then $M_A$ and ${}_AM'$ are $(d-1)$-cluster tilting modules by \cite[Theorem 6.6]{LSW}. 
The following theorem shows that the module $M_A$, which gives a noncommutative quasi-resolution of $A$, is close to being a $(d-1)$-cluster tilting module.

\begin{theorem}\label{d-1 CT and NQR}
     Let $A$ be a noetherian noncommutative projective coordinate ring. Suppose that $A$ has a noncommutative quasi-resolution $B$ given by $(M,M')$ and $\gldim B=d$. Then
    \begin{align*}
        \add_AM&=\{X\in C_A\mid \gExt_{\cA}^i(\cM,\cX)=0,\forall \, 0<i<d-1\}\\
        &=\{X\in C_A\mid \gExt_{\cA}^i(\cX,\cM)=0,\forall \, 0<i<d-1\}
    \end{align*}
    and 
    \begin{align*}
        \add_{\cA}\cM&=\{\cX\in \cC_{\cA}\mid \gExt_{\cA}^i(\cM,\cX)=0,\forall \, 0<i<d-1\}\\
        &=\{\cX\in \cC_{\cA}\mid \gExt_{\cA}^i(\cX,\cM)=0,\forall \, 0<i<d-1\}.
    \end{align*}
\end{theorem}
\begin{proof}
    It suffices to prove the second part. Since $B$ is commonly graded AS-regular, by \cite[Corollary 8.13]{LW2}, $\MCM(B)=\proj B$ and $B$ is an MCM $B$-module. 
    By Theorem \ref{Morita for nc quasi-proj space}, $B\cong \gHom_A(M,M)$. 
    Then, the proof of Proposition \ref{MCM B and Ext in qgr A} implies that $\gExt_{\cA}^i(\cM,\cM)=0$ for $0<i<d-1$.

   Suppose that $\cX\in \cC_{\cA}$ and that $\gExt_{\cA}^i(\cM,\cX)=0$ for all $0<i<d-1$. Choose $X\in C_A$ such that $\cX=\pi X$, and set $Y=\gHom_A(M,X)$. 
    Then Proposition \ref{MCM B and Ext in qgr A} implies that $Y$ is an MCM $B$-module. Hence, $Y$ is projective as a $B$-module.
 
    By \cite[\uppercase\expandafter{\romannumeral6} Lemma 3.1]{SS}, $\gHom_A(M,-)$ induces an equivalence $\add_AM\to \proj B$. Hence, there exists a graded $A$-module $\tilde{X}\in\add_AM$ such that $\gHom_A(M,\tilde{X})\cong \gHom_A(M,X)$. 
    Note that $\add_AM\subseteq C_A$ and $\proj B\subseteq C_B$. 
    By Proposition \ref{C_A and C_B}, $\gHom_A(M,-)$ induces an equivalence from $C_A$ to $C_B$.  
    It follows that $\tilde{X}\cong X$. Hence, $X\in \add_AM$ and $\cX\in \add_{\cA}\cM$.
    Therefore,
    $$\add_{\cA}\cM=\{\cX\in \cC_{\cA}\mid \gExt_{\cA}^i(\cM,\cX)=0,\forall 0<i<d-1\}.$$

    Suppose that $\cX\in\cC_{\cA}$ and $\gExt_{\cA}^i(\cX,\cM)=0$ for all $0<i<d-1$. Choose $X\in C_A$ such that $\cX=\pi X$, and set $Y=\gHom_A(M,X)$. By Proposition \ref{C_A and C_B}, $Y\in C_B$ and hence
    $Y= \gHom_A(M,X)\cong \gHom_A(M,\omega\pi X)\cong\gHom_{\cA}(\cM,\cX)$.
    Theorem \ref{Morita for nc quasi-proj space} implies that, for $0<i<d-1$,
    $$\gExt_{\cB}^i(\cY,\cB)\cong\gExt_{\cB}^i(\pi_B\gHom_{\cA}(\cM,\cX),\pi_B\gHom_{\cA}(\cM,\cM))\cong \gExt_{\cA}^i(\cX,\cM)=0.$$
    Since $\depth_B B=d$, it follows from \cite[Lemma 5.5]{LSW} that for $0<i<d-1$,
    $\gExt_B^i(Y,B)\cong\gExt_{\cB}^i(\cY,\cB)=0$.
    Let $U=D(R^d\Gamma_B(B))$. Then $U$ is an invertible $(B,B)$-bimodule. By Theorem \ref{local duality},
    $$D(R^i\Gamma_B(Y))\cong \gExt_B^{d-i}(Y,U)\cong U\otimes_B\gExt_B^{d-i}(Y,B).$$
    Hence, $R^i\Gamma_B(Y)=0$ for $2\leqslant i\leqslant d-1$. 
    Moreover, $\depth_BY\geqslant 2$ implies that $\Gamma_B(Y)=R^1\Gamma_B(Y)=0$. 
    Then \cite[Lemma 8.1]{LW2} and Lemma \ref{depth and local cohomology} imply that $\depth_BY=d$; hence
     $Y$ is an MCM $B$-module. Therefore, $Y$ is a projective $B$-module.
    By an argument similar to that in the preceding paragraph, $X\in \add_AM$ and $\cX\in \cC_{\cA}$. Therefore,
    \[
    \add_{\cA}\cM=\{\cX\in \cC_{\cA}\mid \gExt_{\cA}^i(\cX,\cM)=0,\forall 0<i<d-1\}.
    \qedhere
    \]
\end{proof}

In dimension $2$, the noncommutative version of the Bondal-Orlov conjecture holds for noncommutative quasi-resolutions. A noetherian commonly graded algebra $A$ of finite cohomological dimension is called {\it Cohen-Macaulay-finite} if it has only finitely many indecomposable MCM $A$-modules, up to isomorphism and shifts.

\begin{theorem}\label{BO conjecture for NQR}
Let $A$ be a noetherian noncommutative projective coordinate ring. Suppose that $A$ has a noncommutative quasi-resolution $B$ of dimension $2$ given by $(M,M')$. Then $A$ is Cohen-Macaulay-finite, and all the noncommutative quasi-resolutions of $A$ are Morita equivalent.
\end{theorem}
\begin{proof}
    By Proposition \ref{when B is GAS-Gorenstein iso. sing.}, $\lcd A=2$. Hence, $C_A=\MCM (A)$. It follows from Theorem \ref{d-1 CT and NQR} that $\add_A M=C_A=\MCM (A)$. Therefore, $A$ is Cohen-Macaulay-finite.

    If $B'$ is another noncommutative quasi-resolution of $A$ given by $(X,X')$, then $\add_A X=\MCM (A)=\add_A M$. It follows that $B'\cong \gEnd_A(X)$ is graded Morita equivalent to $B\cong\gEnd_A(M)$. 
\end{proof}



\section{Noncommutative resolutions of noncommutative isolated singularities}\label{noncommutative resolutions of isolated singularities}
In this section, we study noncommutative resolutions of balanced CM isolated singularities and characterize them by cluster tilting modules. We also address Question \ref{question} by determining when the resolved algebra is commonly graded AS-Gorenstein.

As noted in the Introduction, it is natural to assume that Hypothesis \ref{hypothesis} holds in order to address Question \ref{question}.
To apply the Morita-like theory of noncommutative quasi-projective spaces, we must show that, under Hypothesis \ref{hypothesis}, $A$ is a noncommutative projective coordinate ring. We establish this result in a more general setting.




\begin{lemma}\label{M' proj imply A is coor. ring}
    Let $B$ be a noncommutative projective coordinate ring, let $M'_B$ be a finitely generated projective $B$-module, and set $A=\gEnd_B(M')$ and $M=\gHom_B(M',B)$. If the cokernel of the natural morphism $M\otimes_A M'\to B$ is finite-dimensional, then $A$ is a noncommutative projective coordinate ring. Moreover, if $B$ is noetherian, then $A$ is noetherian and $M'$ is a modulo-torsion-invertible $(A,B)$-bimodule.
\end{lemma}
\begin{proof}
Since $\depth_BM'\geqslant 2$, $A=\gEnd_B(M')\cong \gEnd_{\cB}(\cM')$. By Definition \ref{def-noncom-proj-scheme}, it suffices to show $(\cM',s)$ is ample in $\qgr B$.

Let $\mu:M\otimes_AM'\to B$ be the natural morphism. Then $B/\image \mu$ is finite-dimensional, which implies that $\pi(\image\mu) \cong \cB$ in $\qgr B$.    

Suppose that $\image\mu$ is generated by $y_1,\cdots,y_s$ as a graded right $B$-module.
For each $y_i$, let $y_i=f_{i1}(n_{i1})+\cdots+f_{ir_i}(n_{ir_i})$ where $n_{ij}\in M'$ and $f_{ij}\in M$. Define
$$f=(f_{ij}):\mathop{\oplus}\limits_{i=1}^s\mathop{\oplus}\limits_{j=1}^{r_i}M'(-\deg f_{ij})\to \image\mu.$$
Then $f$ is surjective. For convenience, let $\tilde{M'}=\mathop{\oplus}\limits_{i=1}^s\mathop{\oplus}\limits_{j=1}^{r_i}M'(-\deg f_{ij})$. 
It follows that the composition $\pi\tilde{M'}\to \pi(\image\mu)\to \cB$ is epic. 

Let $m=\min\{\deg f_{ij}\}$. For any $\cY\in \qgr B$, since $(\cB,s)$ is ample in $\qgr B$, there are positive integers $n_1,\cdots, n_t$ and an epimorphism
$\oplus_k\cB(-n_k)\to \cY(-m)$.
So the following composition of morphisms is still epic 
$$\oplus_k\pi\tilde{M'}(-n_k+m)\to \oplus_k\cB(-n_k+m)\to \cY.$$
Note that for any $1\leqslant i\leqslant s$, $1\leqslant j\leqslant r_i$ and $1\leqslant k\leqslant t$, $a_{ijk}:=\deg f_{ij}+n_k-m>0$. 
Then there is
an epimorphism
$$\mathop{\oplus}\limits_{ijk}\cM'(-a_{ijk})\to \cY.$$

For any epimorphism $\cY_1\to \cY_2$, there is an integer $n$ such that
$$\gHom_{\cB}(\cB,\cY_1)_{\geqslant n}\to \gHom_{\cB}(\cB,\cY_2)_{\geqslant n}$$
is surjective.
Since $M'$ is a finitely generated projective $B$-module, $\cM'\in \add_{\cB}\cB$. It follows that
$$\gHom_{\cB}(\cM',\cY_1)_{\geqslant n}\to \gHom_{\cB}(\cM',\cY_2)_{\geqslant n}$$
is surjective.
In conclusion, $(\cM',s)$ is ample in $\qgr B$.

If $B$ is noetherian, then the endomorphism ring of a finitely generated projective $B$-module is noetherian by \cite[Lemma 3.5.6]{MR}. Hence, $A$ is noetherian. By definition, $M'$ is a modulo-torsion-invertible $(A,B)$-bimodule.
\end{proof}

In fact, under Hypothesis \ref{hypothesis}, $A$ is a balanced CM isolated singularity.

\begin{proposition}\label{H1-H3 imply A balanced CM}
    Assume that Hypothesis \ref{hypothesis} holds. Then $A$ is a balanced CM isolated singularity of dimension $d$, and $M_A$ is an MCM generator.
\end{proposition}
\begin{proof}
By Lemma \ref{M' proj imply A is coor. ring}, $A$ is a noetherian noncommutative projective coordinate ring, and $M'$ is a modulo-torsion-invertible $(A,B)$-bimodule. Since $B$ is AS-regular of dimension $d$, it is a balanced CM isolated singularity. The projective $B$-module $M'$ is MCM and satisfies
\[
\gExt_B^i(M',M')=0\quad(0<i<d-1).
\]
Corollary \ref{balanced CM from self extensions} therefore implies that $A$ is a balanced CM isolated singularity of dimension $d$.

Since $M'_B$ is finitely generated projective, $M=\gHom_B(M',B)$ is a finitely generated projective left $B$-module and a generator in $\Gr A$. Therefore, $\depth_{B^o}M=d$. As both $A$ and $B^o$ satisfy $\chi$, \cite[Corollary 4.8]{V3} and Lemma \ref{depth and local cohomology} imply that $\depth_AM=d$. Hence, $M_A$ is an MCM generator.
\end{proof}

Proposition \ref{H1-H3 imply A balanced CM} shows that, under Hypothesis \ref{hypothesis}, $B$ can be regarded as a noncommutative resolution of the balanced CM isolated singularity $A$. We therefore introduce the following more general definition.

\begin{definition}\label{def-noncom-reso}
    Let $A$ be a noetherian noncommutative projective coordinate ring. Suppose that $B$ is a noncommutative quasi-resolution of $A$ given by $(M,M')$, with $M_A$ a generator, or equivalently, with $M'_B$ projective. Then $B$ is called a {\it noncommutative resolution of $A$}.
\end{definition}

Note that $M'_B$ is MCM if and only if it is projective, by \cite[Corollary 8.13]{LW2}. 

Definition \ref{def-nc-reso-AS-G} and Definition \ref{def-noncom-reso} coincide when $A$ is a commonly graded AS-Gorenstein algebra.

\begin{proposition}\label{M is MCM iff M is a generator}
Let $A$ be a noetherian noncommutative projective coordinate ring, and let $B$ be a noncommutative quasi-resolution of $A$ given by $(M,M')$, with $\gldim B=d\geqslant2$. Then the following conditions are equivalent:
\begin{itemize}
    \item [(1)] $M_A$ is MCM;
    \item [(2)] $M_A$ is a generator;
    \item [(3)] $M'_B$ is projective;
    \item [(4)] $B$ is a noncommutative resolution of $A$ given by $(M,M')$.
\end{itemize}
If these conditions hold, $A$ is a balanced CM isolated singularity of dimension $d$.
\end{proposition}

\begin{proof}
By Theorem \ref{Morita for nc quasi-proj space}, one obtains $B\cong\gEnd_A(M)$, $A\cong\gEnd_B(M')$, $M'\cong\gHom_A(M,A)$, and $M\cong\gHom_B(M',B)$. 
It then follows from Proposition \ref{when B is GAS-Gorenstein iso. sing.} that $\lcd A=d$.

(1) $\Rightarrow$ (2). For $0<i<d-1$, by Lemma \ref{cohomology in tails},
\[
\gExt_{\cA}^i(\cA,\cM)\cong R^{i+1}\Gamma_A(M)=0.
\]
Hence Theorem \ref{d-1 CT and NQR} yields $A\in\add_A M$, which implies $M_A$ is a generator.

(2) $\Leftrightarrow$ (3). It is a standard result in Morita theory.

(3) $\Leftrightarrow$ (4). It follows from Definition \ref{def-noncom-reso}.

(3) $\Rightarrow$ (1). By Theorem \ref{equ. of quot. cat. induced by Morita context for commonly graded}, the cokernel $\Coker(M\otimes_AM'\to B)$ is finite-dimensional.
Hence, the Morita context defined by $M'$ satisfies Hypothesis \ref{hypothesis}. Proposition \ref{H1-H3 imply A balanced CM} then establishes (1) and the final assertion.
\end{proof}

It follows immediately from the proof of Proposition \ref{M is MCM iff M is a generator} and \cite[Theorem 6.5]{LSW} that, if $A$ is a commonly graded AS-Gorenstein algebra, then $B$ is a noncommutative resolution of $A$ in the sense of Definition \ref{def-noncom-reso} if and only if it is a noncommutative resolution in the sense of Definition \ref{def-nc-reso-AS-G}.

\begin{theorem}\label{hypo. is equi. to B is ncr of A}
    Hypothesis \ref{hypothesis} holds for $B$ if and only if $B$ is a noncommutative resolution of the noetherian noncommutative projective coordinate ring $A$ given by $(M,M')$. In this case, $A$ is a balanced CM isolated singularity.
\end{theorem}

We next characterize noncommutative resolutions in terms of cluster tilting modules. We first record their generator and duality properties.

\begin{lemma}\label{CT generators and canonical duality}
Let $A$ be a balanced CM algebra of dimension $d\geqslant2$, let $\Omega_A=D(R^d\Gamma_A(A))$, and let $M_A$ be a $(d-1)$-cluster tilting module. Then
\begin{enumerate}
    \item $A,\Omega_A\in\add_A M$.
    \item $\widetilde M=\gHom_A(M,\Omega_A)$ is a $(d-1)$-cluster tilting module over $A^o$.
    \item $\gEnd_{A^o}(\widetilde M)\cong\gEnd_A(M)^o$.
    \item $A,\Omega_A\in\add_{A^o} \widetilde{M}$.
\end{enumerate}
\end{lemma}

\begin{proof}
(1) It follows from the definition of $(d-1)$-cluster tilting modules and Proposition \ref{Hom(-,Omega) induce duality on MCM} (1).

(2) and (3). It follows from Proposition \ref{Hom(-,Omega) induce duality on MCM} (2) and (3).

(4) It is proved by a left module version of the argument of (1).
\end{proof}

\begin{theorem}\label{M_A is CT module}
Let $A$ be a balanced CM isolated singularity of dimension $d\geqslant2$, and let $M\in\MCM(A)$. Set
$B=\gEnd_A(M)$ and $M'=\gHom_A(M,A)$.
Then the following conditions are equivalent:
\begin{itemize}
    \item [(1)] $M_A$ is a $(d-1)$-cluster tilting module;
    \item [(2)] $B$ is a noncommutative resolution of $A$ given by $(M,M')$.
\end{itemize}
Under these conditions, $B$ is a noetherian commonly graded AS-regular algebra of dimension $d$. Moreover, ${}_AM'$ and the duals
${}_A\widetilde M=\gHom_A(M,\Omega_A),
\widetilde M'_A=\gHom_{A^o}(M',\Omega_A)$
are $(d-1)$-cluster tilting modules.
\end{theorem}

\begin{proof}
(2) $\Rightarrow$ (1). It follows from Proposition \ref{when B is GAS-Gorenstein iso. sing.} that $\gldim B=\lcd A=d$. 
For any $X\in\MCM(A)$ and $0<i<d-1$, \cite[Lemma 5.5]{LSW} gives
\[
\gExt_A^i(M,X)\cong\gExt_{\cA}^i(\cM,\cX),\quad
\gExt_A^i(X,M)\cong\gExt_{\cA}^i(\cX,\cM).
\]
By Theorem \ref{d-1 CT and NQR}, $M_A$ is a $(d-1)$-cluster tilting module.

(1) $\Rightarrow$ (2). 
First, we prove that $B$ is a balanced CM isolated singularity of dimension $d$.
By Lemma \ref{CT generators and canonical duality}, $A,\Omega_A\in\add_AM$. Hence, $M_A$ is an MCM generator. By Lemma \ref{equi. cond. for ample (2) for MCM}, $(\cM,s)$ is ample in $\qgr A$. Since $\depth_AM=d\geqslant2$, Lemma \ref{omega pi and depth 2} gives $B\cong\gEnd_{\cA}(\cM)$.
Theorem \ref{noncommutative Serre theorem for commonly graded algebra} therefore implies that $B$ is a right noetherian noncommutative projective coordinate ring and there is an equivalence of noncommutative projective schemes $(\qgr A,\cM,s)\cong (\qgr B,\cB,s)$.
By Lemma \ref{CT generators and canonical duality}, $B^o$ is isomorphic to $\gEnd_{A^o}(\widetilde{M})$.
Applying the same argument to the cluster tilting module $\widetilde M$ over $A^o$ shows that $B^o$ is also a noncommutative projective coordinate ring.
In particular, $B$ is noetherian on both sides.


Applying Proposition \ref{when B is GAS-Gorenstein iso. sing.} with the roles of $A$ and $B$ interchanged shows that $B$ admits a balanced dualizing complex and that $\lcd B=\lcd B^o=d$. 
The cluster tilting condition gives
$\gExt_A^i(M,M)=0$ for $0<i<d-1$.
Hence, Corollary \ref{balanced CM from self extensions}, with the roles of the rings interchanged, implies that $B$ is a balanced CM isolated singularity of dimension $d$. 

Next we prove $\gldim B=d$.
By projectivization (say see \cite[\uppercase\expandafter{\romannumeral6} Lemma 3.1]{SS}), for $Y\in\gr B$, there is a finite projective presentation
\[
\gHom_A(M,M_1)\xrightarrow{\gHom_A(M,f)}
\gHom_A(M,M_0)\to  Y\to  0,
\]
where $M_0,M_1\in\add_AM$ and $f:M_1\to M_0$ is a graded $A$-module morphism.
Let $K_1=\Ker(f)$. Since $M_0$ and $M_1$ are MCM, it follows from \cite[Lemma 2.10]{LSW} that $\depth_AK_1\geqslant2$. 
Since $0\to \gHom_A(M,K_1)\to \gHom_A(M,M_1)\to
\gHom_A(M,M_0)$ is exact, the second syzygy of $Y$ in this presentation is $\gHom_A(M,K_1)$.

By Proposition \ref{C_A and C_B}, $\gHom_A(M,K_1)$ is a finitely generated $B$-module.
Observe that the proof in \cite[Proposition 5.13 (2)]{LSW} also works here. Hence, there is a surjective morphism $M_2\to K_1$ where $M_2$ is a finite direct sum of shifts of $M$, such that
\[
\gHom_A(M,M_2)\to \gHom_A(M,K_1)
\]
is surjective.
Let $K_2$ be the kernel of $M_2\to K_1$. 
Then we obtain a short exact sequence
\[
0\to \gHom_A(M,K_2)\to \gHom_A(M,M_2) \to \gHom_A(M,K_1)\to 0,
\]
which implies that $\gHom_A(M,K_2)$ is the third syzygy of $Y$.
Inductively, we can construct a finite projective resolution of $Y$:
\[
\cdots\to \gHom_A(M,M_r)\to\cdots\to\gHom_A(M,M_1)\to
\gHom_A(M,M_0)\to  Y\to  0
\]
such that $M_r\in\add_A M$ and $\gHom_A(M,K_r)$ is the $(r+1)$-th syzygy of this resolution where
$K_r=\Ker(M_r\to M_{r-1})$.
It follows from \cite[Lemma 2.10]{LSW} that $\depth_AK_r\geqslant\min\{d,r+1\}$. 

For any $r\geqslant 2$, consider the exact sequence
\[
    0\to K_r\to M_r\to K_{r-1}\to 0.
\]
Applying $\gHom_A(M,-)$ to this short exact sequence shows that $\gExt_A^1(M,K_r)=0$ and $\gExt_A^i(M,K_{r-1})\cong \gExt_A^{i+1}(M,K_r)$ for $1\leqslant i\leqslant d-3$, as $\gHom_A(M,M_r)\to \gHom_A(M,K_{r-1})$ is surjective and $\gExt_A^j(M,M)=0$ for $0<j<d-1$.
It follows that when $r=d-1$, $K_{d-1}$ is an MCM $A$-module and $\gExt_A^i(M,K_{d-1})=0$ for $1\leqslant i\leqslant d-2$, which implies that $K_{d-1}\in \add_A M$.
Hence, $\gHom_A(M,K_{d-1})$ is a projective $B$-module and thus $\pdim_B Y\leqslant d$.
Since $Y$ is an arbitrary finitely generated graded $B$-module, one obtains $\gldim B\leqslant d$, and consequently $\gldim B=d$. 

It follows from \cite[Lemma 6.4]{LSW} that $B$ is a commonly graded AS-regular algebra of dimension $d$.
By Definition \ref{def-noncom-reso}, $B$ is a noncommutative resolution of $A$ given by $(M,M')$.

Finally, under either equivalent condition, $M'_B$ is finitely generated projective, and $M\cong\gHom_B(M',B)$ is finitely generated projective as a left module. The opposite graded Morita context therefore satisfies the left-module version of Hypothesis \ref{hypothesis}. Applying Proposition \ref{H1-H3 imply A balanced CM} on this side shows that ${}_A M'$ is MCM. Applying the argument for (2) $\Rightarrow$ (1) to the left modules shows that ${}_A M'$ is a $(d-1)$-cluster tilting module.
It follows from Lemma \ref{CT generators and canonical duality} that ${}_A\widetilde M$ and
$\widetilde M'_A$ are $(d-1)$-cluster tilting modules.
\end{proof}

\begin{corollary}\label{dimension two existence of resolutions}
Let $A$ be a balanced CM isolated singularity of dimension $2$. Then $A$ admits a noncommutative resolution if and only if it has finitely many indecomposable MCM modules up to isomorphism and shift.
\end{corollary}

\begin{proof}
A module $M$ is $1$-cluster tilting if and only if $\add_A M=\MCM(A)$; if and only if $\MCM(A)$ admits a finite additive generator;
if and only if there are finitely many indecomposable MCM $A$-modules up to isomorphism and shift. 
The assertion then follows from Proposition \ref{H1-H3 imply A balanced CM} and Theorem \ref{M_A is CT module}.
\end{proof}

Under Hypothesis \ref{hypothesis}, the balanced dualizing complex of $A$ is characterized as follows.

\begin{proposition}\label{dualizing complex of A}
    Suppose that $A$ is a balanced CM isolated singularity of dimension $d\geqslant 2$, and that $B$ is a noncommutative resolution of $A$ given by $(M,M')$. Let $U=D(R^d\Gamma_B(B))$. Then $(M'\otimes_B U\otimes_BM)[d]$ is a balanced dualizing complex of $A$.
\end{proposition}
\begin{proof}
    By Theorem \ref{Serre duality in D(qgr A)}, $-\otimes_B^L U[d-1]$ induces a Serre functor $F=-\otimes_{\cB}^L\cU[d-1] $ of $\D^b(\qgr B)$.
    Let $\Omega_A=D(R^d\Gamma_A(A))$. Then $\Omega_A[d]$ is a balanced dualizing complex of $A$. 
    Similarly, $-\otimes_A^L\Omega_A[d-1]$ induces a Serre functor $G=-\otimes_{\cA}^L\pi\Omega_A[d-1]$ of $\D^b(\qgr A)$.

    Since $-\otimes_{\cA} \cM'$ and $-\otimes_{\cB} \cM$  give an equivalence between $\qgr A$ and $\qgr B$, they induce an equivalence
    $$-\otimes_{\cA} \cM':\D^b(\qgr A)\rightleftarrows \D^b(\qgr B):-\otimes_{\cB} \cM.$$
    Then $(-\otimes_{\cB}\cM)\circ F\circ (-\otimes_{\cA}\cM')$ is also a Serre functor of $\D^b(\qgr A)$. Since a Serre functor is unique up to natural isomorphism, 
    there is a commutative diagram
    \begin{center}
        \begin{tikzcd}
\pi(A\otimes_A\Omega_A) \arrow[d, "\pi(a_l)"'] \arrow[r, "\cong"] & \pi(A\otimes_AM'\otimes_BU\otimes_BM) \arrow[d, "\pi(a_l)"] \\
\pi(A\otimes_A\Omega_A(r)) \arrow[r, "\cong"]                        & \pi(A\otimes_AM'\otimes_BU\otimes_BM(r))                      
\end{tikzcd}
    \end{center}
    where $a_l:A\to A (r)$ is the left multiplication by $a\in A$ and $r=\deg a$.

    Note that $M'\otimes_BU\otimes_BM\cong \gHom_B(M',M'\otimes_BU)$ and $\depth_B(M'\otimes_BU)=d$.
    By Proposition \ref{C_A and C_B}, $\depth_A(\gHom_B(M',M'\otimes_B U))\geqslant 2$.
    Hence, by Lemma \ref{omega pi and depth 2},
    $$\omega\pi (M'\otimes_BU\otimes_BM)\cong  M'\otimes_BU\otimes_BM.$$  
    On the other hand, $\Omega_A=D(R^d\Gamma_A(A))$ is an MCM $A$-module, so $\omega\pi \Omega_A\cong \Omega_A$. 
    Applying $\omega$ to the above diagram yields an isomorphism
    $\Omega_A\cong M'\otimes_BU\otimes_BM$ in $\Gr A^e$.
 Therefore, $(M'\otimes_B U\otimes_B M)[d]$ is a balanced dualizing complex of $A$.
\end{proof}


\begin{theorem}\label{when A is GAS Gorenstein}
    Suppose that $A$ is a balanced CM isolated singularity of dimension $d\geqslant 2$, and that $B$ is a noncommutative resolution of $A$ given by $(M,M')$. Let $U=D(R^d\Gamma_B(B))$. Then $A$ is commonly graded AS-Gorenstein if and only if $M'\otimes_BU\otimes_B M$ is a projective $A$-module (or $A^o$-module).
\end{theorem}
\begin{proof}
    For convenience, set $V=M'\otimes_BU\otimes_BM$.
    
    Suppose that $A$ is commonly graded AS-Gorenstein. 
    Then, by Proposition \ref{AS-G has balanced dualizing complex and chi}, $D(R^d\Gamma_A(A))$ is an invertible $(A,A)$-bimodule, 
    and consequently, so is $V$.

    Conversely, suppose that $V$ is a projective $A$-module. By the definition of a dualizing complex, $$R\gHom_A(V[d],V[d])\cong \gHom_A(V,V)\cong A.$$
    Thus, $V_A$ is an invertible $(A,A)$-bimodule. 
    It follows from \cite[Theorem 7.8 (6)]{LW2} that $A$ is  commonly graded AS-Gorenstein.
\end{proof}

In fact, Theorems \ref{hypo. is equi. to B is ncr of A} and  \ref{when A is GAS Gorenstein} address Question \ref{question}.

\begin{corollary}\label{Omega_B cong B}
    Keep the assumptions of Theorem \ref{when A is GAS Gorenstein}.
    If $U\cong B(l)$ for some integer $l$, then $A$ is a commonly graded AS-Gorenstein algebra with balanced dualizing complex $A(l)[d]$.
\end{corollary}
\begin{proof}
    It follows that $M'\otimes_BU\otimes_BM\cong M'\otimes_BM(l)\cong A(l)$. 
    By Proposition \ref{dualizing complex of A} and Theorem \ref{when A is GAS Gorenstein}, $A$ is a commonly graded AS-Gorenstein algebra of dimension $d$ and has a balanced dualizing complex $A(l)[d]$. 
\end{proof}

By Corollary \ref{Omega_B cong B}, if $B$ is an $\mathbb{N}$-graded Calabi-Yau algebra, then any finitely generated projective $B$-module $M'$ satisfying (H3) yields a commonly graded AS-Gorenstein algebra $A$ for which $B$ is a noncommutative resolution. 

In \cite{LSW}, a noncommutative version of the Bondal--Orlov conjecture is established in dimensions $2$ and $3$ for noncommutative resolutions of commonly graded AS-Gorenstein isolated singularities. 
See also \cite[Theorem 5.3.2]{I2} for a version concerning module-finite algebras.
For noncommutative resolutions in the sense of Definition \ref{def-noncom-reso}, this conjecture holds in dimensions $2$ and $3$. As the proof is similar, we omit it.


\begin{theorem}\label{BO conjecture}
    Suppose that $A$ is a balanced CM isolated singularity of dimension $d\geqslant 2$.
    \begin{itemize}
        \item [(1)] If $d=2$, then all noncommutative resolutions of $A$ are Morita equivalent.
        \item [(2)] If $d=3$, then all noncommutative resolutions of $A$ are derived Morita equivalent.
    \end{itemize}
\end{theorem}

\section{Noncommutative resolutions of invariant subrings}\label{noncommutative resolutions of invariant rings}
In this section, we consider the invariant subrings of Hopf actions on commonly graded algebras and study their noncommutative resolutions.
For the theory and notation of Hopf algebras and Hopf actions on algebras, we primarily follow \cite{Mo}.




Suppose that $H$ is a finite-dimensional semisimple Hopf $k$-algebra and that $A$ is a left $H$-module $k$-algebra.
Assume that $A$ is noetherian.
Let $A\#H$ denote the smash product of $A$ and $H$.
Both $H$ and $A$ may be regarded as subalgebras of $A\#H$ by identifying $H$ and 
$A$ with $1\#H$ and $A\#1$, respectively. 
Since $H$ is finite-dimensional, $A\#H$ is finitely generated as both a left and a right $A$-module; hence, it is noetherian.
Let $A^H=\{ a \in A \mid h \rightharpoonup a = \varepsilon(h)a, \forall h\in H\}$ be the invariant subring of $A$ under the Hopf action. By \cite[Chapter 4]{Mo}, $A^H$ is noetherian, and $A$ is finitely generated as both a left and a right $A^H$-module. If $H=kG$ is the group algebra of a group $G$, then $A\#H$ and $A^H$ are denoted by $A\#G$ and $A^G$, respectively.

Throughout the remainder of this section, we assume that $A$ is a noetherian commonly graded algebra and that $H$ is a finite-dimensional semisimple Hopf algebra acting homogeneously on $A$, namely, $H \rightharpoonup A_i\subseteq A_i$ for all $i$. Then both $A^H$ and $A\#H$ are commonly graded noetherian algebras.


Let $\int$ denote an integral of $H$ such that $\varepsilon(\int)=1$, where $\varepsilon$ is the counit of $H$.
The following well-known results, originally proved for connected graded algebras, have proofs that also apply to commonly graded algebras. 

\begin{lemma}\label{property of A^H}
Keep the assumptions and notation as above. 
\begin{itemize}
\item [(1)] \cite[Lemma 2.4]{KKZ} The map $\int \rightharpoonup \, : A\to A$ is a projection of $(A^H,A^H)$-bimodules such that $A^H=\int \rightharpoonup A$ and $A=A^H\oplus C$, where $C=\Ker(\int\rightharpoonup)$.
\item [(2)] \cite[Lemma 2.5]{KKZ} $\int\rightharpoonup \, : A\to A$ induces a projection of $(A^H,A^H)$-bimodules
$$\textstyle R^i\Gamma_{(A^H)^o}(\int):R^i\Gamma_{(A^H)^o}(A)\to R^i\Gamma_{(A^H)^o}(A)$$
whose image is $R^i\Gamma_{(A^H)^o}(A^H)$. 
The natural left $A\#H$-module structure of $A$ induces a left $A\#H$-module structure on $R^i\Gamma_{(A^H)^o}(A)$ so that $R^i\Gamma_{(A^H)^o}(\int)$ agrees with the left $H$-action of the element $\int$ on $R^i\Gamma_{(A^H)^o}(A)$.
\item [(3)] \cite[Lemma 3.2]{KKZ} Suppose that $A$ is a balanced CM algebra of dimension $d$. Let $\Omega_A=D(R^d\Gamma_A(A))$ and $\Omega_{A^H}=\Omega_A\cdot \int$. Then $\Omega_{A^H}[d]$ is a balanced dualizing complex of $A^H$. Consequently, $A^H$ is a balanced CM algebra of dimension $d$.
\end{itemize}
\end{lemma}

The relationships between the MCM modules over $A$ and $A^H$ are given as follows.
\begin{lemma}\label{MCM A^H and MCM A}
Keep the assumptions and notation above. Suppose that $A$ is a balanced CM algebra of dimension $d$.
\begin{itemize}
    \item [(1)] For any $M\in \gr A$, $M_A$ is MCM if and only if $M_{A^H}$ is MCM.
    \item [(2)] Both ${}_{A^H}A$ and $A_{A^H}$ are MCM generators in their respective categories.
\end{itemize}
\end{lemma}
\begin{proof}
(1) By Lemma \ref{property of A^H}, $A^H$ is a balanced Cohen-Macaulay algebra of dimension $d$. 
By \cite[Proposition 2.8]{LSW}, for any $M\in \gr A$ and $i\geqslant 0$, $R^i\Gamma_{A^H}(M)\cong R^i\Gamma_A(M)$. Therefore, $M_A$ is MCM if and only if $M_{A^H}$ is MCM.

(2) It follows from (1) and Lemma \ref{property of A^H}.
\end{proof}

\subsection{Gorenstein invariant subrings}
In this subsection, we extend the definition of the homological determinant for Hopf actions on connected graded AS-Gorenstein algebras to the commonly graded setting. We prove that, if the homological determinant of the Hopf action is trivial, then the invariant subring is commonly graded AS-Gorenstein.

Recall that $A$ is called {\it basic} if $A/J_A$ is a direct product of division algebras. 
In this subsection, we assume that \textit{$A$ is a noetherian basic commonly graded AS-Gorenstein algebra of dimension $d$}.



Let $\Omega_A=D(R^d\Gamma_A(A))$. Since $A$ is a basic commonly graded AS-Gorenstein algebra of dimension $d$, $\Omega_A$ is a graded invertible $(A,A)$-bimodule and, by \cite[Proposition 5.8]{LW2}, is isomorphic, as an ungraded $(A,A)$-bimodule, to ${}^{\mu}A^1$, where $\mu$ denotes the Nakayama automorphism of $A$.  
Let $\mathbf{e}$ be an element of $\Omega_A$ such that $\Omega_A=\mathbf{e}A$ as a free $A$-module. 
Then, for any $a\in A$, $a\cdot \mathbf{e}=\mathbf{e}\mu(a)$. 
Since $R\Gamma_A(A)\cong R\Gamma_{A^o}(A)\cong R\Gamma_{(A\#H)^o}(A)$, $\Omega_A$ admits a right $A\#H$-module structure whose restriction to $A$ coincides with the usual right $A$-module structure of ${}^{\mu}A^1$.

Define the $k$-linear map $\eta_{\mathbf{e}}:H\to A$ by
$$\mathbf{e}\cdot h=\mathbf{e} \, \eta_{\mathbf{e}}(h)\in \Omega_A = \mathbf{e}A.$$

\begin{definition}\label{homological determinant} Let $S$ be the antipode of $H$.
 With the above assumptions and notation, the map $\eta_{\mathbf{e}}\circ S$ is called the {\it homological determinant} of the $H$-action on $A$ with respect to $\mathbf{e}$ and is denoted by $\hdet_{\mathbf{e}}$. If $\hdet_{\mathbf{e}}$ is equal to the composition
    \[H\xrightarrow[]{\varepsilon}k\to A,\]
then the homological determinant $\hdet_{\mathbf{e}}$ of the $H$-action on $A$ is called {\it trivial}.
\end{definition}

\begin{remark}
If $A$ is connected graded, then Definition \ref{homological determinant} coincides with the definition given in \cite{JZ,KKZ}, and $\hdet_{\mathbf{e}}$ is independent of the choice of $\mathbf{e}$. 
In this case, $\hdet_{\mathbf{e}}$ is denoted by $\hdet$.
If $A$ is a skew Calabi-Yau algebra, then $A$ is an $\mathbb{N}$-graded AS-regular algebra by \cite{RR}, and the present definition coincides with that in \cite{WZ}. 
    
In \cite{RRZ}, Reyes, Rogalski, and Zhang introduced a class of generalized AS-Gorenstein algebras, which is a special case of Definition \ref{def-generalized-AS-Gorenstein} (see \cite[Remark 5.19]{LW2}). They also defined the homological determinant under the assumption that an $H$-stable generator exists (see \cite[Definition 3.7]{RRZ}). In our setting, $\mathbf{e}$ serves as an $H$-stable generator.
    
When $A$ is $\mathbb{N}$-graded and $\gExt_A^d(A_0,A)$ is regarded as an $H$-subspace of $R^d\Gamma_A(A)$, 
the image of $\hdet_{\mathbf{e}}$ is contained in $A_0$.
\end{remark}

The following proposition establishes that $\Omega_{A^H}$ is a graded invertible $(A^H,A^H)$-bimodule and describes the relationship between the Nakayama automorphisms of $A$ and $A^H$.

\begin{proposition}\label{hdet is trivial}
Keep the assumptions and notation introduced above.
Suppose that $\hdet_{\mathbf{e}}$ is trivial. Then $\Omega_{A^H}$ is a graded invertible $(A^H,A^H)$-bimodule, and
the restriction of the Nakayama automorphism $\mu$ of $A$ to $A^H$ is an automorphism, denoted by $\mu'$, such that $\Omega_{A^H}\cong {}^{\mu'}(A^H)^1$ as ungraded $(A^H,A^H)$-bimodules.
\end{proposition}
\begin{proof}
Since $a\#h=(1\#h_2) (S^{-1}h_1 \rightharpoonup a\# 1) \in A \# H$ and $\hdet_{\mathbf{e}}$ is trivial, it follows that
\begin{equation}\label{right-H-module-str-Omega}
    (\mathbf{e}a)\cdot h=\mathbf{e}\sum (\hdet_{\mathbf{e}}(S^{-1}h_2)(S^{-1}h_1\rightharpoonup a))=\mathbf{e}(S^{-1}h \rightharpoonup a).
\end{equation} 
Thus, for any $a\in A^H$, $(\mathbf{e}a)\cdot \int=\mathbf{e}a$. Therefore, $\mathbf{e}A^H\subseteq \mathbf{e}A\cdot \int$. Conversely,
for any $a\in A$, $(\mathbf{e}a)\cdot \int=\mathbf{e}(\int \rightharpoonup a)\in \mathbf{e}A^H$.
Hence, $\mathbf{e}A\cdot \int=\mathbf{e}A^H$. By Lemma \ref{property of A^H} (3), $\Omega_{A^H} \cong A^H$ as ungraded right $A^H$-modules. 
Therefore, $\Omega_{A^H}$ is a projective generator in $\gr A^H$.
Since $\Omega_{A^H}[d]$ is a dualizing complex of $A^H$, $\gEnd_{A^H}(\Omega_{A^H})\cong A^H$ as graded $(A^H,A^H)$-bimodules. Therefore $\Omega_{A^H}$ is a graded invertible $(A^H,A^H)$-bimodule.

Since $\Omega_A$ is also a graded $(A^H,A\#H)$-bimodule, $(a\cdot \mathbf{e})\cdot h=a\cdot(\mathbf{e}\cdot h)$ for any $a\in A^H$ and $h\in H$.  
On one hand, by \eqref{right-H-module-str-Omega},
\[(a\cdot \mathbf{e})\cdot h=(\mathbf{e}\mu(a))\cdot h=\mathbf{e}(S^{-1}h \rightharpoonup \mu(a)).\]
On the other hand, since $\hdet_{\mathbf{e}} = \varepsilon$, 
\[a\cdot (\mathbf{e}\cdot h)=a\cdot \varepsilon(h)\mathbf{e}=\mathbf{e}(\varepsilon(h)\mu(a)).\]
Therefore, $S^{-1}h \rightharpoonup \mu(a)=\varepsilon(h)\mu(a)$.
It follows that $\mu(a)\in A^H$ for any $a\in A^H$.
Hence, the restriction of $\mu$ to $A^H$ defines an endomorphism $\mu'$ of $A^H$. 
Moreover,
$\Omega_{A^H}$ is isomorphic to ${}^{\mu'}(A^H)^1$ as an ungraded $(A^H,A^H)$-bimodule.
Since $\Omega_{A^H}$ is an invertible $(A^H,A^H)$-bimodule, $\mu'$ is an automorphism of $A^H$.
\end{proof}

\begin{theorem}\label{A^H is AS Gorenstein}
Let $H$ be a finite-dimensional semisimple Hopf algebra, and let $A$ be a noetherian basic commonly graded AS-Gorenstein algebra of dimension $d$.
Assume that $A$ is a left $H$-module algebra, that the $H$-action on $A$ is homogeneous, and that the homological determinant $\hdet_{\mathbf e}$ of the $H$-action on $A$ is trivial. 
Then $A^H$ is a noetherian commonly graded AS-Gorenstein algebra of dimension $d$.
\end{theorem}
\begin{proof}
It follows from Proposition \ref{hdet is trivial} and \cite[Theorem 7.8 (6)]{LW2}.
\end{proof}

In the special case where
$A$ is the preprojective algebra of an extended Dynkin diagram and the Hopf action preserves primitive idempotents, \cite{Wei2} proved the above result.



\subsection{Noncommutative resolutions of invariant subrings}
In this subsection, we investigate the noncommutative resolution of the invariant subring $A^H$.
Recall several facts concerning $H$-module algebras. 

There is a standard right $A\#H$-module structure on $A$, defined by
$$a\cdot (b\#h)=S^{-1}h \rightharpoonup (ab),$$
under which $A$ is an $(A^H,A\#H)$-bimodule. Similarly, there is a standard left $A\#H$-module structure on $A$, defined by 
$$(a\#h)\cdot b=a(h \rightharpoonup b),$$
under which $A$ is an $(A\#H,A^H)$-bimodule.

Let
$$\textstyle\tau:A\otimes_{A^H}A\to A\#H, \, a\otimes_{A^H}b \, \mapsto (a\#1)(1\#\int)(b\# 1),$$
$$\textstyle\mu:A\otimes_{A\#H} A\to A^H, \, a\otimes_{A\#H}b \, \mapsto (\int \rightharpoonup (ab)).$$
Then $(A\#H,A^H,{}_{A^H}A_{A\#H},{}_{A\#H}A_{A^H},\tau,\mu)$ forms a Morita context \cite[Theorem 4.5.3]{Mo}.



The following lemma is a consequence of \cite[Lemma 0.3]{CFM}, \cite[Lemma 1.4]{MU}, and \cite[Lemma 3.1]{BHZ1}.

\begin{lemma}\label{A^H and A smash H}
Keep the assumptions and notation as above. Let $e=1_A\#\int\in A\#H$.
    \begin{itemize}
        \item [(1)] The map $A^H\to e(A\#H)e, \, a\mapsto e(a\#1)e$, is an isomorphism of graded algebras. 
        \item [(2)] $A\to (A\#H)e, \, a\mapsto (a\#1)e$, is an isomorphism of graded $(A\#H,A^H)$-bimodules.
        \item [(3)] $A\to e(A\#H), \, a\mapsto e(a\#1)$, is an isomorphism of graded $(A^H,A\#H)$-bimodules.
    \end{itemize}
\end{lemma}

\begin{lemma}\label{A' and A}
    Keep the assumptions and notation as above. Let $e=1_A\#\int\in A\#H$. Suppose that the map
    $A\#H\to \gEnd_{A^H}(A), a\#h\mapsto (b\mapsto a(h \rightharpoonup b))$, is an algebra isomorphism. Then $\gHom_{A^H}(A,A^H)\cong A$ as graded $(A^H,A\#H)$-bimodules, and the graded Morita context $(A\#H,A^H,{}_{A^H}A_{A\#H},{}_{A\#H}A_{A^H},\tau,\mu)$ is equivalent to the Morita context induced by the module $A_{A^H}$.
\end{lemma}
\begin{proof}
    By Lemma \ref{A^H and A smash H},
\begin{align*}
    \gHom_{A^H}(A,A^H)&\cong \gHom_{e(A\#H)e}((A\#H)e,e(A\#H)e)\\
    &\cong e(A\#H)\otimes_{A\#H}\gHom_{e(A\#H)e}((A\#H)e,(A\#H)e)\\
    &\cong e(A\#H)\otimes_{A\#H} \gHom_{A^H}(A,A)\\
    &\cong e(A\#H)\otimes_{A\#H} A\#H\\
    &\cong e(A\#H)\\
    &\cong A
\end{align*}
as graded $(A^H,A\#H)$-bimodules. The remaining verification is routine.
\end{proof}

\begin{lemma}\label{A smash H is AS-regular}
    Keep the assumptions and notation as above. If $A$ is a noetherian commonly graded AS-regular algebra of dimension $d$, then so is $A\#H$.
\end{lemma}
\begin{proof}  
    For any $X\in\gr A\#H$, $R\Gamma_{A\#H}(X)\cong R\Gamma_A(X)$ by \cite[Proposition 2.8]{LSW}. Hence $\lcd (A\#H)=\lcd A$, and $R^i\Gamma_{A\#H}(X)$ is bounded above for any $i>0$. Thus $\lcd (A\#H)$ is finite and $A\#H$ satisfies the $\chi$-condition by Proposition \ref{facts about chi condition}. 
    Dually, $\lcd(A\#H)^o$ is finite, and $(A\#H)^o$ satisfies the $\chi$-condition. 
    Therefore, by Theorem \ref{existence of balanced dualizing complex}, $A\#H$ admits a balanced dualizing complex. 
    Moreover, since $A\#H$ is free as an $A$-module, it is an MCM $A$-module. The isomorphism $R\Gamma_{A\#H}(A\#H)\cong R\Gamma_A(A\#H)$ implies that $A\#H$ is a balanced CM algebra of dimension $d$. By \cite[Theorem 1.1]{Liu}, $\gldim(A\#H)= d$. 
    It then follows from \cite[Lemma 6.4]{LSW} that $A\#H$ is a commonly graded AS-regular algebra of dimension $d$.
\end{proof}

To study when the invariant subring $A^G$ of an $\mathbb{N}$-graded algebra $A$ under the action of a group $G$ is a noncommutative isolated singularity, Mori and Ueyama introduced the notion of an ample group action in \cite[Definition 2.11]{MU}. 
An ample Hopf action may be defined analogously.

\begin{definition}\label{def-ample-Hopf-action}
    Let $A$ be a noetherian commonly graded algebra, and let $H$ be a finite-dimensional semisimple Hopf algebra such that $A$ is a left $H$-module algebra and the $H$-action on $A$ is homogeneous. The $H$-action on $A$ is called {\it ample}
    if the Morita context $(A\#H,A^H,{}_{A^H}A_{A\#H},{}_{A\#H}A_{A^H},\tau,\mu)$ induces an equivalence
    $$-\otimes_{\cA\#\cH}\cA:(\qgr A\#H,s)\rightleftarrows (\qgr A^H,s):-\otimes_{\cA^{\cH}}\cA.$$
\end{definition}

By Definition \ref{def-ample-Hopf-action}, any ample Hopf action on a commonly graded AS-regular algebra induces a noncommutative resolution of its invariant subring. 
More precisely, we establish the following equivalent characterizations.

\begin{theorem}\label{ample Hopf action when hdet not trivial}
Let $A$ be a noetherian commonly graded AS-regular algebra of dimension $d\geqslant 2$, and let $H$ be a finite-dimensional semisimple Hopf algebra. Suppose that $A$ is a left $H$-module algebra and that the $H$-action on $A$ is homogeneous. Then the following statements are equivalent.
\begin{itemize}
    \item [(1)] The $H$-action on $A$ is ample.
    \item [(2)] The functor $-\otimes_{A\#H}A$ induces an equivalence:
    $$-\otimes_{\cA\#\cH}\cA:(\qgr A\#H,s)\to (\qgr A^H,s).$$
    \item [(3)] $A\#H/(1\#\int)$ is finite-dimensional, where $(1\#\int)$ denotes the ideal generated by $1\#\int$.
     \item [(4)] $\gldim(\qgr A^H)=d-1$, and the natural map
    $$A\#H\to \gEnd_{A^H}(A), \, a\#h \mapsto \big(b\mapsto a(h\rightharpoonup b)\big)$$
    is an isomorphism of graded algebras.
    \item [(5)] $A\#H$ is a noncommutative resolution of $A^H$, given by $({}_{A\#H}A_{A^H},{}_{A^H}A_{A\#H})$.
\end{itemize}
\end{theorem}
\begin{proof}
(1) $\Rightarrow$ (2) Trivial.

(2) $\Rightarrow$ (3) Observe that $A\#H\cong \gEnd_{\cA\#\cH}(\cA\#\cH)\cong \gEnd_{\cA^{\cH}}(\cA)\cong \gEnd_{A^H}(A)$, where the isomorphism is given by
$$A\#H\to \gEnd_{A^H}(A),\, a\#h\mapsto (b\mapsto a(h \rightharpoonup b)).$$
By Lemma \ref{A' and A}, the graded Morita context $(A\#H,A^H,{}_{A^H}A_{A\#H},{}_{A\#H}A_{A^H},\tau,\mu)$ is equivalent to that defined by $A_{A^H}$.
Therefore, by Theorem \ref{Morita for nc quasi-proj space},
$$-\otimes_{\cA\#\cH}\cA:\qgr A\#H\rightleftarrows\qgr A^H:-\otimes_{\cA^{\cH}}\cA$$
is an equivalence. 

Since the image of $\tau:A\otimes_{A^H}A\to A\#H$ is $(1\#\int)$, 
Theorem \ref{equ. of quot. cat. induced by Morita context for commonly graded} implies that $A\#H/(1\#\int)$ is finite-dimensional.

(3) $\Rightarrow$ (5) 
Since $A_{A\#H}$ is projective, the map $\mu:A\otimes_{A\#H}A\to A^H$ is surjective. 
By Theorem \ref{equ. of quot. cat. induced by Morita context for commonly graded} and Lemma \ref{A smash H is AS-regular}, $A\#H$ is a noncommutative quasi-resolution of $A^H$ given by $(A,A)$. 
By Lemma \ref{MCM A^H and MCM A}, $A_{A^H}$ is a generator; hence $A\#H$ is a noncommutative resolution of $A^H$.

(5) $\Rightarrow$ (1) It follows directly from the definition of a noncommutative resolution.

(5) $\Rightarrow$ (4) By Theorem \ref{hypo. is equi. to B is ncr of A}, $A^H$ is a balanced CM isolated singularity of dimension $d$. The natural isomorphism follows from the argument used to establish (2) $\Rightarrow$ (3).

(4) $\Rightarrow$ (2) By Lemma \ref{property of A^H}(3), $A^H$ is a balanced CM algebra of dimension $d$. 
It follows from (4) and Theorem \ref{A and A^o are noncommutative iso. sing.} that $A^H$ is a balanced CM isolated singularity. By Lemma \ref{MCM A^H and MCM A}, $A_{A^H}$ is an MCM generator. Hence, Lemma \ref{equi. cond. for ample (2) for MCM} implies that $(\pi_{A^H}A,s)$ is ample in $\qgr A^H$. Since $\depth_{A^H}A=d\geqslant 2$, its graded endomorphism ring in $\qgr A^H$ is isomorphic to $\gEnd_{A^H}(A)$, which is identified with $A\#H$ via the natural map in (4). Therefore, Theorems \ref{noncommutative Serre theorem for commonly graded algebra} and \ref{Morita for nc quasi-proj space} yield the equivalence in (2).
\end{proof}


In the final part of this section, we consider the noncommutative resolution of the invariant subring when the homological determinant of the Hopf action is trivial.

\begin{lemma}\label{Hom_A^H(A,-)}
Keep the assumptions and notation as above. Suppose that $A$ is a balanced CM algebra of dimension $d\geqslant 2$, that $(\cA,s)$ is ample in $\qgr A^H$, and that $\gEnd_{A^H}(A)$ is a finitely generated graded $A^H$-module.
\begin{itemize}
    \item [(1)] If $N$ is an MCM $A^H$-module satisfying $\gExt^i_{A^H}(A,N)=0$ for $0<i<d-1$, then $\gHom_{A^H}(A,N)$ is an MCM $A$-module.
    \item [(2)] If $\gldim A=d$, then 
    $$\{N\in\MCM A^H \mid \gExt^i_{A^H}(A,N)=0, \forall \, 0<i<d-1\}\subseteq \add_{A^H}A.$$
\end{itemize}
\end{lemma}
\begin{proof}
(1)  
Since $(\cA,s)$ is ample in $\qgr A^H$,
it follows from the proof of \cite[Theorem 4.12]{LSW}
that $\gHom_{A^H}(A,N)$ is a finitely generated $\gEnd_{A^H}(A)$-module for any $N\in \MCM A^H$.


%
There is a natural algebra morphism $A^H \hookrightarrow A \to \gEnd_{A^H}(A)$.
Since $\gEnd_{A^H}(A)$ is, by assumption, a finitely generated graded $A^H$-module, 
it 
is also a finitely generated graded $A$-module.
Thus, $\gHom_{A^H}(A,N)$ is a finitely generated graded $A$-module.

Let $S=A/J_A$ and $P^\bullet\to S\to 0$ be a minimal graded projective resolution of $S_A$. Let $0\to N\to I^\bullet$ be a minimal graded injective resolution of $N_{A^H}$. The double complex $\gHom_A(P^\bullet,\gHom_{A^H}(A,I^\bullet))$ yields the following convergent spectral sequence:
$$\gExt_A^p(S,\gExt^q_{A^H}(A,N))\Rightarrow \gExt_{A^H}^{p+q}(S,N).$$
Since $\gExt^i_{A^H}(A,N)=0$ for $0<i<d-1$, it follows that $\gExt^p_A(S,\gHom_{A^H}(A,N))$ is a subquotient of $\gExt^p_{A^H}(S,N)$ for $p<d$. Since $N$ is an MCM $A^H$-module, $\gExt^i_{A^H}(X,N)=0$ for any finite-dimensional $A^H$-module $X$ and any $i<d$. 
Therefore, for any $i<d$,
$$\gExt^i_A(S,\gHom_{A^H}(A,N))=0,$$ 
which implies that $\depth_A\gHom_{A^H}(A,N)\geqslant d$; hence, $\gHom_{A^H}(A,N)$ is an MCM $A$-module.

(2) Suppose $\gldim A=d$. Then, by \cite[Lemma 6.4]{LSW}, $A$ is a commonly graded AS-regular algebra.
By the Auslander-Buchsbaum formula \cite[Theorem 1.7]{HY}, 
$$\pdim_A\gHom_{A^H}(A,N)=0.$$
Hence, either as a graded $A$-module or a graded $A^H$-module, $\gHom_{A^H}(A,N)$ is a direct summand of a finite direct sum of shifts of $A$. 
By Lemma \ref{property of A^H}, $A=A^H\oplus C$ as a graded $(A^H,A^H)$-bimodule. Hence,
$$\gHom_{A^H}(A,N)\cong \gHom_{A^H}(A^H,N)\oplus\gHom_{A^H}(C,N)\cong N\oplus\gHom_{A^H}(C,N)$$
as graded $A^H$-modules.
Therefore, $N$ is a direct summand of $\gHom_{A^H}(A,N)$ in $\gr A^H$. Consequently, $N\in\add_{A^H}A$.
\end{proof}



\begin{proposition}\label{A is d-1 CT A^H module}
    Let $A$ be a noetherian basic commonly graded AS-regular algebra of dimension $d\geqslant 2$, and let $H$ be a finite-dimensional semisimple Hopf algebra. Suppose that $A$ is a left $H$-module algebra, that the $H$-action on $A$ is homogeneous, and that the homological determinant $\hdet_{\mathbf{e}}$ is trivial. If $A^H$ is a noncommutative isolated singularity and $B=\gEnd_{A^H}(A)$ is an MCM $A$-module, then
    \begin{itemize}
        \item [(1)] $A_{A^H}$ is a $(d-1)$-cluster tilting $A^H$-module;
        \item [(2)] $B$ is a noetherian commonly graded AS-regular algebra of dimension $d$;
        \item [(3)] $B$ is a noncommutative resolution of $A^H$ given by $(A,\gHom_{A^H}(A,A^H))$.  
    \end{itemize}
\end{proposition}

\begin{proof}
By Theorem \ref{A^H is AS Gorenstein}, $A^H$ is a commonly graded AS-Gorenstein isolated singularity of dimension $d$. By Lemma \ref{MCM A^H and MCM A}, $A_{A^H}$ is an MCM generator. Therefore, \cite[Proposition 5.14]{LSW} implies that $B$ has a balanced dualizing complex.

By \cite[Proposition 2.8]{LSW}, $R\Gamma_A(B)\cong R\Gamma_B(B)$, which implies that
$B$ is an MCM $B$-module. 
Thus, $B$ is a balanced Cohen-Macaulay algebra of dimension $d$. Therefore,
$\gExt_{A^H}^i(A,A)=0$
for all $0<i<d-1$ by \cite[Proposition 5.15]{LSW}. It follows from Lemma \ref{Hom_A^H(A,-)} that
$$\{N\in\MCM A^H\mid \gExt^i_{A^H}(A,N)=0, \forall \, 0<i<d-1\}= \add_{A^H}A.$$

The assertions then follow from \cite[Theorem 6.5, Theorem 6.6]{LSW}. 
\end{proof}

The following theorem generalizes \cite[Theorem 3.10]{MU}.

\begin{theorem}\label{ample Hopf action}
Let $A$ be a noetherian basic commonly graded AS-regular algebra of dimension $d\geqslant 2$, and let $H$ be a finite-dimensional semisimple Hopf algebra. Suppose that $A$ is a left $H$-module algebra, that the $H$-action on $A$ is homogeneous, and that the homological determinant $\hdet_{\mathbf{e}}$ is trivial. Then the following statements are equivalent.
\begin{itemize}
    \item [(1)] The $H$-action on $A$ is ample.
    \item [(2)] The functor $-\otimes_{A\#H}A$ induces an equivalence:
    $$-\otimes_{\cA\#\cH}\cA:(\qgr A\#H,s)\to (\qgr A^H,s).$$
    \item [(3)] $A\#H/(1\#\int)$ is finite-dimensional where $(1\#\int)$ is the ideal generated by $1\#\int$.
    \item [(4)] $A^H$ is a noncommutative isolated singularity, and 
    $$A\#H\to \gEnd_{A^H}(A),a\#h\mapsto (b\mapsto a(h\rightharpoonup b))$$
    is an isomorphism of graded algebras.
    \item [(5)] $A\#H$ is a noncommutative resolution of $A^H$ given by $({}_{A\#H}A_{A^H},{}_{A^H}A_{A\#H})$.
\end{itemize}
\end{theorem}
\begin{proof}
By Theorem \ref{ample Hopf action when hdet not trivial}, it remains to prove that (4) implies $\gldim(\qgr A)=d-1$.
It follows from Theorem \ref{A^H is AS Gorenstein} that $A^H$ is a commonly graded AS-Gorenstein algebra of dimension $d$. 
Hence, Corollary \ref{ASG+iso. is ASG iso.} shows that $\gldim(\qgr A)=d-1$.
\end{proof}

\section{Illustrative examples}\label{examples}
\subsection{An ample group action with trivial homological determinant on a two-dimensional preprojective algebra}
In this subsection, we study an example of group actions on $\mathbb{N}$-graded AS-regular algebras, introduced in \cite{Wei2}. 

Assume that the characteristic of the base field $k$ is not $2$. For $n\geqslant 2$, let $Q_{n-1}$ denote the quiver associated with the extended Dynkin diagram $\tilde{A}_{n-1}$ given below.
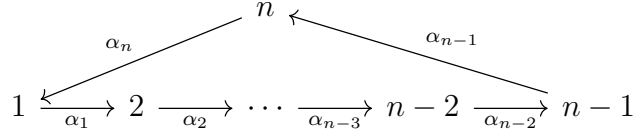
\begin{center}
    \begin{tikzcd}
                                 &                                  &  n \arrow[lld, "\alpha_n"'] &                                    &                                          \\
 1 \arrow[r, "\alpha_1"'] &  2 \arrow[r, "\alpha_2"'] & \cdots \arrow[r, "\alpha_{n-3}"']  & n-2 \arrow[r, "\alpha_{n-2}"'] &  n-1 \arrow[llu, "\alpha_{n-1}"']
\end{tikzcd}
\end{center}
Let $\bar{Q}_{n-1}$ denote the double quiver of $Q_{n-1}$; that is, $\bar{Q}_{n-1}$ has the same vertex set as $Q_{n-1}$ and additional arrows $\alpha_i':i+1\to i$ for $i\in \mathbb{Z}_n$.
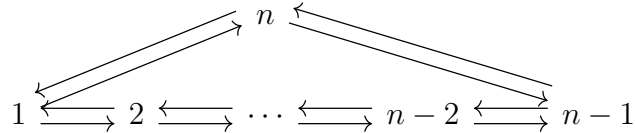
\begin{center}
    \begin{tikzcd}
                                     &                                                 & n \arrow[lld, shift right=2] \arrow[rrd]             &                                                   &                                                       \\
1 \arrow[r, shift right] \arrow[rru] & 2 \arrow[r, shift right] \arrow[l, shift right] & \cdots \arrow[r, shift right] \arrow[l, shift right] & n-2 \arrow[r, shift right] \arrow[l, shift right] & n-1 \arrow[llu, shift right=2] \arrow[l, shift right]
\end{tikzcd}
\end{center}

Let $A$ be the preprojective algebra with respect to $Q_{n-1}$, that is,
$$A=k\bar{Q}_{n-1}/(\sum_{i=1}^n(\alpha_i\alpha_i'-\alpha_i'\alpha_i)).$$ 
Then $A$ is a noetherian $\mathbb{N}$-graded AS-regular algebra of dimension $2$ by \cite[Proposition 2.11]{Wei1}. 

For convenience, we continue to use $\alpha_i$ and $\alpha_i'$ to denote the corresponding elements of $A$. Let $e_i$ denote the trivial arrow at vertex $i$. For each $i\in \mathbb{Z}_n$, left multiplication of $\sum_{i=1}^n(\alpha_i\alpha_i'-\alpha_i'\alpha_i)$ by $e_i$ gives
$$\alpha_i\alpha_i'-\alpha'_{i-1}\alpha_{i-1}=0 \in A.$$ 

Let $g$ be the graded automorphism of $A$ defined by $g(\alpha_1)=\alpha_1$, $g(\alpha_1')=\alpha_1'$, and $g(\alpha_i)=-\alpha_i$, $g(\alpha_i')=-\alpha_i'$ for $i\geqslant 2$. 
Let $G=<\id,g>$ be the cyclic group generated by $g$, $H=kG$, and $\int=(\id+g)/2$. By \cite[Corollary 5.8]{Wei2}, the homological determinant of $g$ is trivial. Hence, by Theorem \ref{A^H is AS Gorenstein} (or \cite[Theorem 3.12]{Wei2}), $A^H$ is a noetherian $\mathbb{N}$-graded AS-Gorenstein algebra of dimension $2$. 

We claim that $A\#H/(1\#\int)$ is a finite-dimensional algebra and that
$A\#H$ is a noncommutative resolution of $A^H$. Let $I=(1\#\int)$ be the ideal of $A\#H$ generated by $1\#\int$. 
 
Note that, for any $x\in A$, $x\#\id  + x\#g = 2x\#\int\in I.$
Now, suppose that $x\in A$ satisfies $g(x)=-x$. Then
$$\textstyle x\#\id-x\#g=(1\#\int)\cdot 2(x\#\id)\in I,$$ 
which implies that $x\#\id\in I$ and $x\# g\in I$. 
Therefore, for any $i\geqslant 2$, 
$$\alpha_i\#\id,\ \alpha_i\#g,\ \alpha_i'\#\id,\  \alpha_i'\#g\in I.$$  
Since $\alpha_1\alpha_1'=\alpha_n'\alpha_n$ and $\alpha_1'\alpha_1 = \alpha_2\alpha_2' \in A$, it follows that $y\#\id, y\#g \in I$ for any $y \in A$ represented by a path of length greater than $1$. Hence, $(A\#H)_{\geqslant 2}\subseteq I$, and $A\#H/(1\#\int)$ is finite-dimensional. So we have the following proposition.

\begin{proposition}
    Keep the assumptions as above. 
    \begin{itemize}
        \item [(1)] The $H$-action on $A$ is ample.
        \item [(2)] $A^H$ is an AS-Gorenstein isolated singularity.
        \item [(3)] $\gEnd_{A^H}(A)\cong A\#H$.
        \item [(4)] $A\#H$ is a noncommutative resolution of $A^H$ given by $(A,A)$.
        \item [(5)] $A_{A^H}$ is a $1$-cluster tilting module. In particular, $A^H$ has only finitely many indecomposable MCM modules up to shifts and isomorphism.
    \end{itemize} 
\end{proposition}
\begin{proof}
    It follows from Theorem \ref{ample Hopf action}.
\end{proof}

\subsection{Quasi-Veronese algebras as resolutions of Veronese algebras}
We study the noncommutative resolution of (quasi-)Veronese algebras. We first recall the definition.

\begin{definition}
    Let $A$ be a $\mathbb{Z}$-graded algebra and $r$ be a positive integer. 
    \begin{itemize}
        \item [(1)] The $r$-th quasi-Veronese of $A$ is a $\mathbb{Z}$-graded algebra:
    $$A^{[r]}:=\bigoplus_{i\in \mathbb{Z}}\left(
    \begin{array}{cccc}
        A_{ri} & A_{ri+1} & \cdots & A_{r(i+1)-1}\\
        A_{ri-1} & A_{ri} & \cdots & A_{r(i+1)-2}\\
        \vdots & \vdots & \vdots &\vdots\\
        A_{r(i-1)+1}& A_{r(i-1)+2} & \cdots & A_{ri}
    \end{array}
    \right)$$
    with the matrix multiplication in each degree.
        \item [(2)] The $r$-th Veronese algebra of $A$ is a $\mathbb{Z}$-graded algebra:
    $$A^{(r)}:=\bigoplus_{i\in\mathbb{Z}}A_{ri}.$$
    \end{itemize}
\end{definition}

\begin{lemma}\cite[Lemma 3.9]{Mor2}\label{equivalence of A and quasi-Veronese}
If $A$ is a $\mathbb{Z}$-graded algebra and $r$ is a positive integer, then $\Gr A$ is equivalent to $\Gr A^{[r]}$. 
\end{lemma}
\begin{proof}
    For any $M\in \Gr A$, define $F(M)=\bigoplus_{i\in\mathbb{Z}}(\bigoplus_{j=0}^{r-1}M_{ri-j})$. Then $F$ is an equivalence of categories.
\end{proof}

Note that the equivalence functor $F:\Gr A\to \Gr A^{[r]}$ does not generally commute with the shift functor.

\begin{proposition}\label{qV algebra is AS regular}
    If $A$ is a noetherian commonly graded AS-regular algebra of dimension $d$, then $A^{[r]}$ is also a noetherian commonly graded AS-regular algebra of dimension $d$.
\end{proposition}
\begin{proof}
    Since $\Gr A$ is equivalent to $\Gr A^{[r]}$, the graded global dimension of $A^{[r]}$ equals $d$, and $A^{[r]}$ is noetherian. 
    
    Let $F:\Gr A\to \Gr A^{[r]}$ be an equivalence of categories, and let $G$ be a quasi-inverse of $F$. 
    Then, for any $n\in\mathbb{Z}$, $G(A^{[r]}(n))$ is a projective module in $\Gr A$,
    and for any graded simple module $X$ in $\Gr A^{[r]}$, $GX$ is a graded simple module in $\Gr A$. 
    Hence, 
    \[\Ext_{\Gr A^{[r]}}^i(X,A^{[r]}(n))\cong \Ext_{\Gr A}^i(GX,G(A^{[r]}(n)))=0\] for all $i\neq d$ and all $n\in \mathbb{Z}$. 
    It follows that, for any graded simple $A^{[r]}$-module $X$, $\gExt_{A^{[r]}}^i(X,A^{[r]})=0$ for $i\neq d$. 

    Let $0\to A^{[r]}\to I^0\to \cdots \to I^d\to 0$ be a minimal graded injective resolution. 
    Then $0\to GA^{[r]}\to GI^0\to \cdots \to GI^d\to 0$ is a minimal graded injective resolution of the projective $A$-module $GA^{[r]}$. 
    Since $A$ satisfies the $\chi$-condition, $\gExt_A^d(A/J_A,GA^{[r]})\cong \gHom_A(A/J_A,GI^d)\cong \soc GI^d$ is finite-dimensional, 
    which implies that $\soc GI^d$ is a finite direct sum of graded simple $A$-modules. Consequently, $\soc I^d$ is a finite direct sum of graded simple $A^{[r]}$-modules. 
    Hence, $\gExt_{A^{[r]}}^d(X,A^{[r]})$ is finite-dimensional for any graded simple $A^{[r]}$-module.

    Dually, for any graded left simple $A^{[r]}$-module $Y$, $\gExt_{(A^{[r]})^o}^i(Y,A^{[r]})=0$ for $i\neq d$, and $\gExt_{(A^{[r]})^o}^d(Y,A^{[r]})$ is finite-dimensional. By an argument analogous to that of \cite[Proposition 2.25]{WY}, $A^{[r]}$ is a noetherian commonly graded AS-regular algebra of dimension $d$.
\end{proof}

For the $r$-th quasi-Veronese algebra $A^{[r]}$ of $A$, let
$$e=\left(
    \begin{array}{cccc}
        1 &0 & \cdots & 0\\
        0 & 0 & \cdots & 0\\
        \vdots & \vdots & \vdots &\vdots\\
        0& 0 & \cdots & 0
    \end{array}
    \right)\in (A^{[r]})_0.$$
Then $e$ is an idempotent in $A^{[r]}$ such that $eA^{[r]}e\cong A^{(r)}$.

\begin{theorem}\label{Veronese resolution}
    Suppose that $A$ is a noetherian $\mathbb{N}$-graded AS-regular algebra of dimension $d\geqslant 2$, generated by $A_1$ over $A_0$. Then $A^{(r)}$ is a balanced CM isolated singularity of dimension $d$, and $A^{[r]}$ is a noncommutative resolution of $A^{(r)}$ given by $(A^{[r]}e,eA^{[r]})$. Moreover, $(A^{[r]}e)_{A^{(r)}}$ is a $(d-1)$-cluster tilting module. If $U=D(R^d\Gamma_{A^{[r]}}(A^{[r]}))$, then $eUe$ is the canonical bimodule of $A^{(r)}$ and is modulo-torsion-invertible.
\end{theorem}

\begin{proof}
    It follows from Proposition \ref{qV algebra is AS regular} that $A^{[r]}$ is a noetherian $\mathbb{N}$-graded AS-regular algebra of dimension $d$.

    Since $A$ is generated by $A_1$ over $A_0$, $A_iA_j=A_{i+j}$ for any $i,j$. Then, for any $n\geqslant 1$,
    \begin{align*}
        A^{[r]}_1eA^{[r]}_{n-1}&=\left(
    \begin{array}{cccc}
        A_{r}  & \cdots & A_{2r-1}\\
        \vdots  & \vdots &\vdots\\
        A_{1} & \cdots & A_{r}
    \end{array}
    \right)
    e
    \left(
    \begin{array}{cccc}
        A_{r(n-1)}  & \cdots & A_{rn-1}\\
        \vdots  & \vdots &\vdots\\
        A_{r(n-2)+1} & \cdots & A_{r(n-1)}
    \end{array}
    \right)\\
    &=\left(
    \begin{array}{cccc}
        A_{rn}  & \cdots & A_{r(n+1)-1}\\
        \vdots  & \vdots &\vdots\\
        A_{r(n-1)+1} & \cdots & A_{rn}
    \end{array}
    \right)\\
    &=A^{[r]}_n.
    \end{align*}
    Hence, $(A^{[r]}eA^{[r]})_{\geqslant 1}=A^{[r]}_{\geqslant 1}$, which implies that $A^{[r]}/(A^{[r]}eA^{[r]})$ is finite-dimensional. 
    
    Let $M'=eA^{[r]}$. Then $\gEnd_{A^{[r]}}(M')\cong eA^{[r]}e\cong A^{(r)}$.
    Theorem \ref{hypo. is equi. to B is ncr of A} shows that $A^{(r)}$ is a balanced CM isolated singularity of dimension $d$, and $A^{[r]}$ is a noncommutative resolution of $A^{(r)}$ given by $(A^{[r]}e,eA^{[r]})$.
    By Theorem \ref{M_A is CT module}, $(A^{[r]}e)_{A^{(r)}}$ is a $(d-1)$-cluster tilting module. 
    At last, Proposition \ref{dualizing complex of A} and Theorem \ref{canonical bimodule and isolated singularity} imply that 
    $$eA^{[r]}\otimes_{A^{[r]}}U\otimes_{A^{[r]}}A^{[r]}e\cong eUe,$$
    and $eUe$ is modulo-torsion-invertible.
\end{proof}

\begin{example}
Suppose that $A$ is a noetherian connected AS-regular algebra of dimension $d\geqslant 2$ with Gorenstein parameter $l$ over an algebraically closed field $k$ of characteristic $0$. 
Assume that $A$ is generated by $A_1$ over $A_0$.
Let $\zeta$ be a primitive $l$-th root of unity, and define an automorphism $\tau: A \to A$ by $x \mapsto \zeta^{\deg x}x$ (for any homogeneous element $x \in A$).
Then $A^{(l)}$ is isomorphic, as an ungraded algebra, to $A^G$, where $G$ is the finite group generated by $\tau$. Since the homological determinant $\hdet$ of the $G$-action is trivial, both $A^G$ and $A^{(l)}$ are AS-Gorenstein (see \cite[Theorem 3.6]{JZ}). Therefore, by Theorem \ref{when A is GAS Gorenstein}, $eA^{[l]}\otimes_{A^{[l]}} D(R^d\Gamma_{A^{[l]}}(A^{[l]}))\otimes_{A^{[l]}}A^{[l]}e$ is an invertible $(A^{(l)},A^{(l)})$-bimodule. 
\end{example}

\subsection{Free group actions with non-trivial homological determinants on polynomial rings}
In this subsection, we investigate the ampleness of free group actions on polynomials without assuming that the invariant subring is Gorenstein. 

\begin{setting}\label{group action}
    Let $R=k[x_1,\cdots,x_n]$ be a graded polynomial ring over a field $k$ of characteristic $0$, where $n\geqslant 2$ and $\deg x_i=1$. Let $G$ be a finite subgroup of $\GL_n(k)$ acting linearly on $x_1,\cdots,x_n$. Then $G$ can be viewed as a subgroup of the group of graded automorphisms of $R$. Let $\int=\frac{1}{\mid G\mid}\sum_{g\in G}g$. 
\end{setting}

By Lemma \ref{A smash H is AS-regular}, $R\#G$ is an $\mathbb{N}$-graded AS-regular algebra of dimension $n$. Moreover, by \cite[Theorem 4.1]{RRZ}, $R\#G$ is a skew Calabi-Yau algebra of dimension $n$.

A non-identity element $g$ of $\GL_n(k)$ is called a {\it pseudo-reflection} if $\rank(g-\id)\leqslant 1$. The group $G$ is said to act freely on $R$ if none of the eigenvalues of any non-identity element of $G$ is $1$. If $G$ acts freely on $R$, then it contains no pseudo-reflections.

For a graded prime ideal $\mathfrak{p}$ of a commutative $\mathbb{Z}$-graded ring $A$, the localization of $A$ with respect to the homogeneous elements of $A\backslash \mathfrak{p}$ is denoted by $A_{(\mathfrak{p})}$.
A commutative $\mathbb{Z}$-graded ring $A$ is called an {\it isolated singularity} if, for every non-maximal graded prime ideal $\mathfrak{p}$, the localization $A_{(\mathfrak{p})}$ has finite graded global dimension. 
If $A$ is a graded quotient of a polynomial ring, then $A$ is an isolated singularity if and only if $\gldim(\qgr A)$ is finite, by \cite[Corollary 4.13]{LW1}.

\begin{proposition}\label{propoerties of R^G}
Keep the assumptions and notation of Setting \ref{group action}. Suppose that $G$ contains no pseudo-reflections. Then the following hold.
    \begin{itemize}
        \item [(1)] \cite{A1,A2} $\End_{R^G}(R)\cong R\# G$.
        \item [(2)] \cite[Theorem 1]{W2} $R^G$ is Gorenstein if and only if $G\subset \SL_n(k)$.
        \item [(3)] \cite[Lemma 2.1]{MS} If $G$ acts on $R$ freely, then $R^G$ is an isolated singularity.
    \end{itemize}
\end{proposition}


We claim that if $G$ acts freely on $R$, then
the group action is ample in the sense of Definition \ref{def-ample-Hopf-action}. 
Thus, $R\#G$ is a noncommutative resolution of $R^G$. This was proved in \cite[Corollary 3.11]{MU} under the additional assumption that $R^G$ is Gorenstein (equivalently, that $G$ is a subgroup of $\SL_n(k)$, or that $\hdet$ is trivial). 

\begin{theorem}\label{G acts freely on polynomial}
    Keep the assumptions and notation of Setting \ref{group action}. Suppose that $G$ acts freely on $R$. Then the following statements hold.
    \begin{itemize}
        \item [(1)] The $G$-action on $R$ is ample.
        \item [(2)] $R\#G/(1\#\int)$ is finite-dimensional.
        \item [(3)] $R\#G$ is a noncommutative resolution of $R^G$ given by $({}_{R\#G}R_{R^G},{}_{R^G}R_{R\#G})$.
        \item [(4)] $R^G$ is a balanced CM isolated singularity of dimension $n$, $R_{R^G}$ is an $(n-1)$-cluster tilting module, and the canonical bimodule of $R^G$ is modulo-torsion-invertible.
    \end{itemize}
\end{theorem}
\begin{proof}
    By Theorem \ref{ample Hopf action when hdet not trivial}, it suffices to prove 
    $$-\otimes_{\cR\#\cG}\cR:(\qgr R\#G,s)\to (\qgr R^G,s)$$
    is an equivalence of categories. By Proposition \ref{propoerties of R^G}, $\End_{R^G}(R)\cong R\#G$ and $R^G$ is an isolated singularity. 
    Hence, for any $M\in \MCM R^G$ and non-maximal graded prime ideal $\mathfrak{p}$, the homogeneous localization $M_{(\mathfrak{p})}$ is projective. 
    It follows that $\gExt_{R^G}^i(M,N)_{(\mathfrak{p})}=0$ for any $i>0$ and $N\in \gr R^G$. 
    Thus $\gExt_{R^G}^i(M,N)$ is finite-dimensional by \cite[Lemma 4.8]{LW1}. 

    For any epimorphism $\cN_1\to \cN_2$ in $\qgr R^G$ with kernel $\cK$,
choose $K\in \gr R^G$ such that $\pi K=\cK$. Since $\gExt_{R^G}^1(R,K)$ is finite-dimensional, Proposition \ref{facts about chi condition} implies that $\gExt_{\cR^G}^1(\cR,\cK)$ is bounded above.  
    It follows from the exact sequence $0\to \cK\to \cN_1\to \cN_2\to 0$ that
    $$\gHom_{\cR^G}(\cR,\cN_1)\to \gHom_{\cR^G}(\cR,\cN_2)\to \gExt_{\cR^G}^1(\cR,\cK)$$
    is exact.
    There exists an integer $n_0$ such that
    $$\gHom_{\cR^G}(\cR,\cN_1)_{\geqslant n_0}\to \gHom_{\cR^G}(\cR,\cN_2)_{\geqslant n_0}$$
    is surjective. Since $R_{R^G}$ is a generator, $(\cR,s)$ is ample in $\qgr R^G$. Therefore, there is an equivalence
    $$\pi\gHom_{\cR^G}(\cR,-):(\qgr R^G,s)\to (\qgr \gEnd_{R^G}(R),s).$$
    Since $\gEnd_{R^G}(R)\cong R\#G$, the functor
    $$\pi\gHom_{\cR^G}(\cR,-):(\qgr R^G,s)\to (\qgr R\#G,s)$$
    is an equivalence.
    By Lemma \ref{A' and A}, $R\cong \gHom_{R^G}(R,R^G)$ as graded $(R^G,R\#G)$-bimodules.
    By Theorem \ref{Morita for nc quasi-proj space}, $-\otimes_{\cR\#\cG}\cR$ is  quasi-inverse to $\pi\gHom_{\cR^G}(\cR,-)$. Thus, (1)--(3) follow. Finally, (4) follows from Theorems \ref{hypo. is equi. to B is ncr of A}, \ref{M_A is CT module}, and \ref{canonical bimodule and isolated singularity}.
\end{proof}

We note that, in the ungraded case, namely, when $G$ acts freely on $R=k[[x_1,\cdots,x_n]]$, Iyama proved that the module $R$ is an $(n-1)$-cluster tilting $R^G$-module \cite[Theorem 2.5]{I1}. 
In the graded setting, the cluster tilting assertion follows directly from the resolution in Theorem \ref{G acts freely on polynomial} and the equivalence of Theorem \ref{M_A is CT module}.

\subsection*{Acknowledgment} The authors gratefully acknowledge Xuyang Chen and Silu Liu for their valuable discussions and suggestions.

\subsection*{AI statement} AI was used to assist the authors in correcting grammar and spelling.

\end{document}